\documentclass{article}
\usepackage{amsmath}
\usepackage{amsfonts}
\usepackage{amsthm}
\usepackage{enumerate}
\usepackage{amssymb}

\newtheorem{thm}{Theorem}[section]
\newtheorem{prop}[thm]{Proposition}
\newtheorem{cor}[thm]{Corollary}
\newtheorem{lem}[thm]{Lemma}

\newtheorem*{thmnn}{Theorem}

\theoremstyle{remark}
\newtheorem{rmk}[thm]{Remark}

\theoremstyle{definition}
\newtheorem{defn}[thm]{Definition}

\begin{document}

\title{Hermitian Jacobi Forms Having Modules as their Index and Vector-Valued Jacobi Forms}

\author{Shaul Zemel}

\maketitle

\section*{Introduction}

The theory of Jacobi forms is detailed, in the classical setting, in \cite{[EZ]}, with the Skew-holomorphic analogue described in \cite{[Sk]}. They serve as a bridge between (elliptic) modular forms and Siegel modular forms of degree 2, by showing up as the Fourier--Jacobi coefficients of the latter. This notion was expanded to have a lattice index initially in \cite{[BK]}, but was developed in many following works, like \cite{[CWR]}, \cite{[BRi]}, and \cite{[BRR]}, the theses \cite{[A]} and \cite{[Mo]}, and recently \cite{[BS]}, as well as \cite{[Ze3]} in the indefinite lattice index case.

We note that the lattice index can be interpreted, in the appropriate coordinates, as a matrix index for Jacobi forms. This is important for considering (scalar-valued) Siegel modular forms of degree $r+1$, as the coefficients of their Fourier--Jacobi expansions of co-genus $r$ (in the terminology of \cite{[BRa]} and \cite{[Xi]}) are Jacobi forms which have $r \times r$ matrices as their indices.

A crucial relation to (vector-valued) elliptic modular forms is the fact that Jacobi forms admit a theta decomposition, which is described classically in \cite{[EZ]}, but has its precursors in the relation with the plus-spaces from \cite{[K1]} and \cite{[K2]}, generalized later for some congruence subgroups in \cite{[LZ]}. This amounts to the identification of (scalar-valued) Jacobi forms with vector-valued modular forms with the appropriate Weil representation, and have neat descriptions for indices which are either scalars, lattices, or matrices.

When the even lattice $L$ is positive definite of degree $b_{+}$, so that the constructions behave well in terms of holomorphicity, the theta decomposition produces isomorphisms \[J_{k,L}(\Gamma) \longleftrightarrow M_{k-b_{+}/2}(\Gamma,\rho_{\overline{L}})\quad\mathrm{and}\quad J_{k,L}^{!}(\Gamma) \longleftrightarrow M_{k-b_{+}/2}^{!}(\Gamma,\rho_{\overline{L}})\] for every integral weight $k$ and finite index subgroup $\Gamma$ of the metaplectic double cover of $\operatorname{SL}_{2}(\mathbb{Z})$ between the corresponding spaces of holomorphic and weakly holomorphic Jacobi and modular forms, where $\rho_{\overline{L}}$ is the Weil representation dual to that of $L$.

\smallskip

Siegel modular forms have Hermitian analogues, now known as Hermitian modular forms, and whose theory was initiated in \cite{[Bra1]}, \cite{[Bra2]}, and \cite{[Bra3]}. In degree 2 they also admit Fourier--Jacobi expansions, whose members are known as Hermitian Jacobi forms. Their basic systematic theory was initiated in \cite{[Hav1]} (or the shorter version \cite{[Hav2]}), and a deeper treatise in the thesis \cite{[De]} also involving the relation to the Hermitian modular forms. The thesis \cite{[He]} was concerned with some arithmetic calculations involving these objects. Virtually all of these manuscripts consider Hermitian Jacobi forms with a rational (or integral) index.

To make the distinction between the notions more visible, note that in the rank 1 case, the classical (also called \emph{orthogonal}) Jacobi forms have one variable from the upper half-plane $\mathcal{H}$ and one from $\mathbb{C}$. However, the Hermitian Jacobi forms considered in the references mentioned above have the variable from $\mathcal{H}$ and \emph{two} variables from $\mathbb{C}$, with the complex conjugation on the field showing up in the parts of the defining formulae involving one of them. In symbols, \[\Phi(\tau,\zeta)\ -\ \mathrm{orthogonal\ Jacobi\ form},\qquad\Phi(\tau,\zeta,\omega)\ -\ \mathrm{Hermitian\ Jacobi\ form}\] (with an orthogonal lattice $L$, the variable $\zeta$ is from $L_{\mathbb{C}}$, and when $M$ is a Hermitian lattice, the variables $\zeta$ and $\omega$ come from the appropriate parts of its complexification---see Definition \ref{Jacdef} below).

The interplay between the natural setting of a lattice index and of a matrix index shows up in many applications, in particular for Siegel modularity results of cycles, as in \cite{[Zh]} and \cite{[BRa]}. Recently, \cite{[Xi]} established the Hermitian modularity of cycles in some cases of the unitary setting (namely when the imaginary quadratic field is norm-Euclidean), a result which involves Hermitian theta series.

Now, in all these references, the Siegel or Hermitian modular forms are vector-valued, and their Fourier--Jacobi coefficients can thus have non-integral matrices as their indices. This point was mentioned briefly in these references by considering the Jacobi group as living inside the Siegel or Hermitian modular group, which allowed these references to use these Jacobi forms without needing to deal with some intricate details of their behavior.

\smallskip

This paper thus has 3 goals, which are interwoven with one another.
\begin{enumerate}[1.]
\item The construction in detail the basic parts of a systematic theory of Hermitian Jacobi forms having lattice index of arbitrary rank over some order in an imaginary quadratic field.
\item The investigation of vector-valued Jacobi forms (both for orthogonal and Hermitian lattices).
\item The establishment of good definitions and theories for Jacobi forms (again in both settings) having non-integral matrices as their indices.
\end{enumerate}
Regarding the first goal, it seems that such a theory does not exist in the literature. The second one allows for, and requires, more restrictive and delicate periodicity properties, as we shall soon explain. The third goal was the initial motivation for writing this paper, and it serves as the motivation for the second one (despite the latter being interesting in its own right), since the desired Jacobi forms must indeed be vector-valued in general, and present this more delicate periodicity.

We will give more details after a few remarks. First, orthogonal and Hermitian Jacobi forms of non-integral matrix index, as in our third goal, are expected to have many applications, like in the Kudla program, concerning the modularity of special cycles of any codimension on orthogonal and unitary Shimura varieties. Indeed, the modularity results from \cite{[BRa]} and \cite{[Xi]} involve vector-valued Siegel and Hermitian modular forms, which are established to be such using their Fourier--Jacobi expansions, the members of which are precisely Jacobi forms on the sort investigated in this paper. This is also the case in our joint work \cite{[BRZ]}, the work on which has led to me investigate the questions answered in this paper.

\smallskip

Note that many papers proving modularity results, like, e.g., \cite{[KM]}, work in the Ad\'{e}lic setting, and thus consider more general but less explicit Schwartz functions on the finite Ad\'{e}lic part in their results and calculations. As the cycles in \cite{[BRa]}, \cite{[Xi]}, \cite{[BRZ]}, and others come with well-defined indices and associated finite-dimensional Weil representations, it is useful and important to have the tools to establish such results in this level of explicitness. The tools developed in this paper serve this purpose, in addition to the interesting relations that they satisfy among themselves and with vector-valued elliptic modular forms.

\smallskip

We remark that while for most applications only the holomorphic, positive definite setting is required, in \cite{[Ze3]} we could show how to define Jacobi forms with indefinite lattice index in a way that all the important properties (except for holomorphicity in the variable from the upper half-plane) carry out to this indefinite setting, with the same proofs. In this paper we thus work with the same generality as in \cite{[Ze3]}. However, we do give at the end, for future use, the statements of the results for positive definite lattices, as they are a but simpler to present, and show up, due to their holomorphicity, in the applications.

\smallskip

We can now describe the content of this paper with more details. The vector-valued Jacobi theta function $\Theta_{L}$ of an even lattice $L$ (holomorphic when $L$ is positive definite, or the more general one defined in \cite{[Ze3]}) is a function of $\tau$ from the upper half-plane $\mathcal{H}$ and $\zeta$ from the complex vector space $L_{\mathbb{C}}$. For fixed $\tau$ it is ``periodic'' in the variable $\zeta \in L_{\mathbb{C}}$ with respect to $L$ and $\tau L$, with some explicit exponential factors. In symbols, when $L$ is positive definite, this property reads \[\Theta_{L}(\tau,\zeta+\tau\sigma+\nu)=\mathbf{e}\big(-\tau\sigma^{2}/2-(\sigma,\zeta)\big)\Theta_{L}(\tau,\zeta)\quad\mathrm{for\ }\sigma\mathrm{\ and\ }\nu\mathrm{\ in\ }L.\] Thus any combination of its scalar-valued components, with coefficients that are functions only of the variable $\tau$ from the upper half-plane $\mathcal{H}$, also shares this periodicity property.

Conversely, any function on $\mathcal{H} \times L_{\mathbb{C}}$ exhibiting this behavior, and is smooth so that periodicity in $L$ implies the existence of a Fourier expansion, must be obtained in this way. This is proved in Lemma 2.1 of \cite{[Ze3]} in the indefinite orthogonal setting (though for definite lattices it was clearly known before), and see our Lemma \ref{Fourinv} for the Hermitian one, to be explained below. Relating the modularity property on both sides produces the isomorphisms mentioned above, which explicitly becomes \[F \in M_{k-b_{+}/2}^{!}(\Gamma,\rho_{L}^{*})\quad\mapsto\quad\textstyle{\Phi:=\sum_{\gamma \in L^{*}/L}f_{\gamma}\theta_{\gamma} \in J_{k,L}^{!}(\Gamma)}\] (with the image of elements of $M_{k-b_{+}/2}(\Gamma,\rho_{\overline{L}})$ covering precisely $J_{k,L}(\Gamma)$).

However, when one considers the full vector-valued theta function $\Theta_{L}$, then one obtains periodicity with respect to the larger lattice $L^{*}$ and $\tau L^{*}$, with the former multiplying the elements of the natural basis of $\mathbb{C}[L^{*}/L]$ by appropriate roots of unity, while the latter translates the basis elements of $\mathbb{C}[L^{*}/L]$ in which $\Theta_{L}$ is valued. Using the notation from Definition \ref{opersDM}, the periodicity property from above extends to \[\Theta_{L}(\tau,\zeta+\tau\sigma+\nu)=\mathbf{e}\big(-\tau\sigma^{2}/2-(\sigma,\zeta)\big)\chi_{\nu}t_{\sigma}\Theta_{L}(\tau,\zeta)\quad\mathrm{for\ }\sigma\mathrm{\ and\ }\nu\mathrm{\ in\ }L^{*}.\] The idea is that by constructing vector-valued combinations of the components of $\Theta_{L}$, if those actions of $\chi_{\nu}$ and $t_{\sigma}$ remain well-defined for some $\nu$ and $\sigma$ from $L^{*}$ that are not necessarily in $L^{*}$, then the periodicity property for such $\nu$ and $\sigma$ carries over, with these definitions, to the combination in question.

To describe that in a clearer and more explicit manner, take $D$ to be any discriminant form, let $H$ be an isotropic subgroup of $D_{L} \oplus D$ such that its projections $H_{L}$ onto $D_{L}$ and $H_{D}$ onto $D$ are isomorphic to $H$ (such subgroups are called \emph{horizontal} in Definition \ref{horizdef} below), and set $\Delta_{H}:=H^{\perp}/H$. This is a very natural situation, occurring when $L$ is a primitive sublattice of some lattice $K$, the discriminant form $D$ is the one arising from the orthogonal complement $L^{\perp}$ of $L$ in $K$, and $\Delta_{H}$ produces $D_{K}$ from $D_{L} \oplus D_{L^{\perp}}$. Using the arrow operator $\downarrow_{H}$ we then obtain a map \[F:\mathcal{H}\to\mathbb{C}[D]\quad\mapsto\quad\Phi:=\downarrow_{H}(\Theta_{L} \otimes F):\mathcal{H} \times L_{\mathbb{C}}\to\mathbb{C}[\Delta_{H}].\] A crucial point is that the operators $t_{\sigma}$ and $\chi_{\nu}$ remain well-defined on $\mathbb{C}[\Delta_{H}]$, as long as $\sigma$ and $\nu$ from $L^{*}$ are such that their $D_{L}$-images are in the subgroup $H_{L}^{\perp} \subseteq D_{L}$ that is perpendicular to $H_{L}$. Moreover, the aforementioned extended periodicity property of $\Theta_{L}$ implies that $\Phi$ exhibits the vector-valued periodicity with respect to elements of $L^{*}$ and of $\tau L^{*}$ using these operators, provided that the $D_{L}$-images of the elements from $L^{*}$ are in $H_{L}^{\perp}$ (see Proposition \ref{pairper} below for the Hermitian setting).

In fact, we prove a much stronger result, in Theorems \ref{VVmain} and \ref{VVorth} for the Hermitian and orthogonal settings respectively.
\begin{thmnn}
The construction from the previous paragraph produces an isomorphism between the space of modular forms having representation $\rho_{D}$ and vector-valued Jacobi forms having representation $\rho_{\Delta_{H}}$ and this extended periodicity property, which in the holomorphic setting can be written as \[J_{k,L}^{H}(\Gamma,\rho_{\Delta_{H}}) \longleftrightarrow M_{k-b_{+}/2}(\Gamma,\rho_{D})\quad\mathrm{and}\quad J_{k,L}^{H,!}(\Gamma,\rho_{\Delta_{H}}) \longleftrightarrow M_{k-b_{+}/2}^{!}(\Gamma,\rho_{D}).\]
\end{thmnn}
Our theorem generalizes the scalar-valued case, which is obtained when we take $D$ to be $D_{\overline{L}}$ and an appropriate $H$, with trivial $H_{L}^{\perp}$ and $\Delta_{H}$. In order to understand the subspaces of Jacobi forms carrying higher periodicity properties (either scalar-valued or vector-valued), we consider the case where the function $F$ is $\uparrow_{I}G$ for an isotropic subgroup $I$ of $D$ in Proposition \ref{FuparrowG} below, and show in what sense these higher periodicity properties characterize the images of such functions $F$. This requires a lot more notations for introducing explicitly, and the results are given in Theorems \ref{peruparrow} and \ref{largerHperp} below. This establishes the second goal of this paper (in both settings).

\smallskip

Before turning to the third one, we describe the construction of Hermitian Jacobi forms of lattice index. For this we consider an order $\mathcal{O}$ in an imaginary quadratic field $\mathbb{K}$, which we view as embedded in $\mathbb{C}$ by some fixed embedding, and a projective module $M$ over $\mathcal{O}$. Denote by $V$ the vector space generated by $M$ over $\mathbb{K}$, and let $\mathcal{D}$ be the different of $\mathcal{O}$, with inverse $\mathcal{D}^{-1}$. We take a non-degenerate Hermitian pairing \[\langle\cdot,\cdot\rangle:V \times V\to\mathbb{K},\mathrm{\ with\ }\langle M,M \rangle \subseteq\mathcal{D}^{-1},\mathrm{\ and\ set\ }(\cdot,\cdot):=\operatorname{Tr}^{\mathbb{K}}_{\mathbb{Q}}\circ\langle\cdot,\cdot\rangle:V \times V\to\mathbb{Q}\] be the associated orthogonal pairing. The dual $M^{*}$ of $M$ defined by pairing with $M$ to $\mathcal{D}^{-1}$ is the same as its dual as an orthogonal lattice, and we assume an appropriate even condition. Thus $M$ admits a Hermitian Jacobi theta function $\Theta_{M}$ (which in our extended, possibly indefinite setting, is a Hermitian Siegel--Jacobi theta function), as well as a theta function $\widetilde{\Theta}_{M}$ obtained from $M$ as an orthogonal lattice.

Now, the $\mathcal{O}$-lattice $M$ (or the space $V$) has three complexifications, one via the embedding of $\mathbb{K}$ into $\mathbb{C}$ (we denote this one by $V_{\mathbb{C}}$ in Definition \ref{tensprod}), one through letting $\mathbb{K}$ be embedded into $\mathbb{C}$ in the complex conjugate way (producing the space $V_{\overline{\mathbb{C}}}$), and the tensor with $\mathbb{C}$ over $\mathbb{Q}$, which is the direct sum of these spaces. These subspaces are the eigenspaces of the action of a traceless element of $\mathbb{K}$, and this element can be used to reproduce the Hermitian pairing from the $\mathbb{Q}$-valued one, as it behaves like a generalized complex structure (see Lemma \ref{pairchar} below). The Hermitian Jacobi theta function of $M$ has variables $\tau$ from $\mathcal{H}$, $\zeta$ from $V_{\mathbb{C}}$, and also $\omega$ from $V_{\overline{\mathbb{C}}}$.

With Proposition \ref{decompair} explaining the relations of the two pairings with the spaces $V_{\mathbb{C}}$ and $V_{\overline{\mathbb{C}}}$, we have the equality \[\Theta_{M}(\tau,\zeta,\omega)=\widetilde{\Theta}_{M}(\tau,\zeta+\omega)\quad\mathrm{for}\quad\tau\in\mathcal{H},\ \zeta \in V_{\mathbb{C}},\mathrm{\ and\ }\omega \in V_{\overline{\mathbb{C}}},\] with $\zeta+\omega$ considered as a variable from the complexification $V\otimes_{\mathbb{Q}}\mathbb{C}$. Using this equality (generalized in Lemma \ref{compThetaM} to the indefinite setting) and the relations from Proposition \ref{decompair}, we can transfer all the important properties of $\widetilde{\Theta}_{M}$ already established in \cite{[Ze3]} for $M$ as an even lattice to the Hermitian setting and the theta function $\Theta_{M}$. Then Hermitian Jacobi forms with lattice index $M$ are defined in such a way that their properties transfer to those of orthogonal Jacobi forms of lattice index $M$ over $\mathbb{Z}$, via $\Theta_{M}$ playing the role of $\widetilde{\Theta}_{M}$.

Note that the signature of $M$ over $\mathbb{Z}$ is twice the one of $M$ over $\mathcal{O}$, affecting the change of weights between Jacobi forms and the associated elliptic modular forms. In particular, in the positive definite setting and holomorphic forms, the isomorphisms will be \[J_{k,M/\mathcal{O}}(\Gamma) \longleftrightarrow M_{k-b_{+}}(\Gamma,\rho_{\overline{M}})\quad\mathrm{and}\quad J_{k,M/\mathcal{O}}^{!}(\Gamma) \longleftrightarrow M_{k-b_{+}}^{!}(\Gamma,\rho_{\overline{M}}),\] with the index $M/\mathcal{O}$ implying Hermitian Jacobi forms, but the Weil representations are the same ones. As the weights are integral, $\Gamma$ is just a finite index subgroup of $\operatorname{SL}_{2}(\mathbb{Z})$ (without a metaplectic cover), a group which for all but finitely many orders $\mathcal{O}$ coincides with the unitary group $\operatorname{U}(1,1)(\mathbb{Z})$. For the remaining orders, the latter group can be a bit larger, and the Weil representation extends in Proposition \ref{extWeil} below. This achieves the first goal, and allows us to complete the other two in the orthogonal and Hermitian settings simultaneously.

We remark that the theory of Jacobi froms over totally real fields as developed in \cite{[Boy]} can probably combine with our observations to produce the basis for a theory of Hermitian Jacobi forms over CM fields of larger degree over $\mathbb{Q}$.

\smallskip

For the third goal we restrict attention to positive definite lattices, but consider an extended notion in which the matrix can be positive semidefinite (this is required in \cite{[BRa]}, \cite{[Xi]}, and \cite{[BRZ]}). After presenting the form of our results in the holomorphic setting, we define a Jacobi form of non-integral matrix index $S$ to take values in $\mathbb{C}[\Delta]$ for some discriminant form $\Delta$, using a map $\psi$ from integral vectors to $\Delta$ that must preserve the quadratic values in the appropriate sense. This translates to \[\psi:\mathbb{Z}^{c_{+}}\to\Delta\quad\mathrm{with}\quad\tfrac{(\psi a,\psi a)}{2}=a^{t}Sa+\mathbb{Z}\quad\mathrm{and}\quad\psi(a)=0\mathrm{\ when\ }Sa=0,\] in the orthogonal setting (where $S$ is of size $c_{+} \times c_{+}$), and similarly with $\mathcal{O}^{c_{+}}$ and the equality $\langle\psi a,\psi a\rangle=a^{*}Sa+\mathbb{Z}$ in the Hermitian one. This property of $\psi$ is the reason why for non-integral $S$ one cannot define scalar-valued Jacobi forms, and a good notion only exists for vector-valued functions, using the results established in our second goal.

Indeed, using $\psi$ we determine the lattice which can produce a Jacobi form with coordinates whose pairings are represented by our matrix $S$ as expected, but the image of the basis used in the definition must be contained in the dual lattice, and not just the lattice itself if $S$ if not integral. The operators analogous to $t_{\sigma}$ and $\chi_{\nu}$ from above are defined in terms of the map $\psi$, as is the group corresponding to $H_{L}^{\perp}$ (or $H_{M}^{\perp}$ in the Hermitian setting), and then we can complete the definition of Jacobi forms with index $S$ as vector-valued Jacobi forms with the appropriate lattice and enhanced periodicity properties. This is done in Definition \ref{indSdef} and \ref{defindS} in the orthogonal and Hermitian settings respectively, with their respective properties given in Theorems \ref{JacindS} and \ref{HJS} below. We conclude by presenting the form of the Fourier expansions of these Jacobi forms in Remark \ref{FourS} at the end.

\smallskip

This paper is divided into six sections. Section \ref{HermMod} introduces the fine details of the notions involving Hermitian linear spaces over imaginary quadratic fields. Section \ref{HJTheta} presents the theory of Hermitian Siegel--Jacobi theta functions, while Section \ref{HJacForms} defines the Hermitian Jacobi forms in question and establishes their properties. Section \ref{VVJF} then gives the details of the constructions involving orthogonal and Hermitian Jacobi forms in the vector-valued setting, and Section \ref{RelsArrow} investigates the relations with the arrow operators. Finally, Section \ref{HolJForms} contains the statements of the results for positive definite lattices, and shows how to transform the formulations into Hermitian matrix indices as well.

\smallskip

I would like to thank E. Rosu, B. Williams, and J. Xia for useful comments which helped to improve the presentation of this paper. Many thanks also to J. Bruinier  and A. Krieg for interesting discussions around this topic, as well as the introduction of some references, including the new pre-print \cite{[BS]}.

\section{Hermitian Spaces and Modules \label{HermMod}}

Let $\mathbb{K}$ be an imaginary quadratic field, so that $\mathbb{K}=\mathbb{Q}(\sqrt{-d})$ for some integer $d>0$, which we consider as embedded in $\mathbb{C}$, with $\sqrt{-d}$ standing for $\sqrt{d} \cdot i$. Then a finite-dimensional vector space $V$ over $\mathbb{K}$ can be viewed as a vector space over $\mathbb{Q}$, endowed with an endomorphism $J\in\operatorname{End}_{\mathbb{Q}}(V)$, representing the action of $\sqrt{-d}\in\mathbb{K}$ on $V$, which satisfies $J^{2}=-d\operatorname{Id}_{V}$ in $\operatorname{End}_{\mathbb{Q}}(V)$. One can think of $J$ as some kind of generalized complex structure on $V$.

Using this structure, we define the following tensor products.
\begin{defn}
The shorthand $V_{\mathbb{C}}$ stands for $V\otimes_{\mathbb{K}}\mathbb{C}$, through the fixed embedding of $\mathbb{K}$ into $\mathbb{C}$. When $\mathbb{K}$ is embedded via the complex conjugate embedding, the corresponding tensor product of $V$ over $\mathbb{K}$ with $\mathbb{C}$ is denoted by $V_{\overline{\mathbb{C}}}$. The notation $V_{\mathbb{R}}$ stands for $V\otimes_{\mathbb{Q}}\mathbb{R}$. \label{tensprod}
\end{defn}
Indeed, for the last notation in Definition \ref{tensprod}, the field $\mathbb{K}$ cannot be embedded into $\mathbb{R}$, so that $V_{\mathbb{R}}$ has the only possible meaning.

We shall be using the following decomposition of the complexification $V\otimes_{\mathbb{Q}}\mathbb{C}$, and its projections.
\begin{lem}
The following assertions hold:
\begin{enumerate}[$(i)$]
\item The space $V\otimes_{\mathbb{Q}}\mathbb{C}$ is the direct sum of $V_{\mathbb{C}}$ and $V_{\overline{\mathbb{C}}}$ from Definition \ref{tensprod}.
\item Given $\lambda \in V\otimes_{\mathbb{Q}}\mathbb{C}$, write it as $\lambda_{\mathbb{C}}+\lambda_{\overline{\mathbb{C}}}$ according to the decomposition from part $(i)$. Then we have $\lambda_{\mathbb{C}}=\frac{\lambda}{2}-\sqrt{-d}\frac{J\lambda}{2d}$ and $\lambda_{\overline{\mathbb{C}}}=\frac{\lambda}{2}+\sqrt{-d}\frac{J\lambda}{2d}$.
\item The complex conjugation on $\mathbb{C}$ interchanges the two summands from part $(i)$. The space $V_{\mathbb{R}}$ from Definition \ref{tensprod} embeds into $V\otimes_{\mathbb{Q}}\mathbb{C}$ as the set of pairs $\{(\lambda,\overline{\lambda})|\lambda \in V_{\mathbb{C}}\}$, and the restrictions to it of the projections onto $V_{\mathbb{C}}$ and onto $V_{\overline{\mathbb{C}}}$ are isomorphisms of real vector spaces.
\end{enumerate} \label{VQdecom}
\end{lem}

\begin{proof}
The fact that $J^{2}=-d\operatorname{Id}_{V}$ in $\operatorname{End}_{\mathbb{Q}}(V)$ implies that over $\mathbb{C}$, the endomorphism $J$ decomposes $V$ into the direct sum of its $\pm\sqrt{-d}$-eigenspaces, yielding part $(i)$ via Definition \ref{tensprod}. For part $(ii)$ we just note that the two asserted vectors $v_{\mathbb{C}}$ and $v_{\overline{\mathbb{C}}}$ are eigenvectors of $J$ with eigenvalues $\sqrt{-d}$ and $-\sqrt{-d}$ respectively, and they sum to $V$. As complex conjugation interchanges the aforementioned eigenvalues of $J$, we deduce the first assertion in part $(iii)$, from which the second one follows and immediately implies the third one as well. This proves the lemma.
\end{proof}

The following simple consequence of our definitions will turn out crucial in what follows.
\begin{cor}
Consider $V$ with its structure as a vector space over $\mathbb{K}$, with $\mathbb{K}$ embedded in $\mathbb{C}$. Then the two projections from part $(ii)$ of Lemma \ref{VQdecom} are $\mathbb{Q}$-linear, with $\lambda\mapsto\lambda_{\mathbb{C}}$ being $\mathbb{K}$-linear, while $\lambda\mapsto\lambda_{\overline{\mathbb{C}}}$ is $\mathbb{K}$-conjugate-linear. \label{Konspaces}
\end{cor}

\begin{proof}
The fact that both maps are $\mathbb{Q}$-linear is immediate, and for the statements about $\mathbb{K}$, we only need to consider the action of $\sqrt{-d}\in\mathbb{K}$. But on $V$ this is precisely the endomorphism $J$, and using the fact that $J^{2}\lambda=-d\lambda$ one easily verifies from the formulae that $(J\lambda)_{\mathbb{C}}=\sqrt{-d}\cdot\lambda_{\mathbb{C}}$ in $V_{\mathbb{C}}$, while in $V_{\overline{\mathbb{C}}}$ we have $(J\lambda)_{\overline{\mathbb{C}}}=-\sqrt{-d}\cdot\lambda_{\overline{\mathbb{C}}}$. This proves the corollary.
\end{proof}

\smallskip

We now make the following definition, in order to set our conventions for the rest of the paper.
\begin{defn}
A \emph{Hermitian pairing} on $V$ is a map $\langle\cdot,\cdot\rangle:V \times V\to\mathbb{K}$, which is $\mathbb{K}$-linear in the first variable and $\mathbb{K}$-conjugate-linear in the second one. In this case we say that $\big(V,\langle\cdot,\cdot\rangle\big)$ is a \emph{Hermitian space over $\mathbb{K}$}. Such a pairing also endows $V$ with the $\mathbb{Q}$-bilinear pairing $(\cdot,\cdot):=\operatorname{Tr}^{\mathbb{K}}_{\mathbb{Q}}\circ\langle\cdot,\cdot\rangle:V \times V\to\mathbb{Q}$. \label{pairdef}
\end{defn}
If $\big(V,\langle\cdot,\cdot\rangle\big)$ is a Hermitian space over $\mathbb{K}$, as in Definition \ref{pairdef}, then we shorthand the pairing $\langle\lambda,\lambda\rangle\in\mathbb{Q}$ of $\lambda \in V$ with itself to $|\lambda|^{2}$. The resulting pairing $(\cdot,\cdot)$ makes every unitary space over $\mathbb{K}$ also an orthogonal space over $\mathbb{Q}$. If we write $(\lambda,\lambda)$ for $\lambda \in V$ as simply $\lambda^{2}$, then the associated quadratic form on $V$ over $\mathbb{Q}$ takes $\lambda$ to $\frac{\lambda^{2}}{2}$, which also coincides with $|\lambda|^{2}$.

Given a bilinear pairing on $V$, the fact that it arises from a Hermitian pairing as in Definition \ref{pairdef} yields several consequences, which in turn characterize the pairings that are obtained in this way.
\begin{lem}
For a bilinear pairing $(\cdot,\cdot):V \times V\to\mathbb{Q}$, with associated quadratic form $\lambda\mapsto\frac{\lambda^{2}}{2}$, the following are equivalent:
\begin{enumerate}[$(i)$]
\item The pairing $(\cdot,\cdot)$ is $\operatorname{Tr}^{\mathbb{K}}_{\mathbb{Q}}\circ\langle\cdot,\cdot\rangle$ for a Hermitian pairing $\langle\cdot,\cdot\rangle:V \times V\to\mathbb{K}$.
\item We have $(J\lambda,J\mu)=d(\lambda,\mu)$ for every $\lambda$ and $\mu$ in $V$.
\item The quadratic form satisfies $\frac{(J\lambda)^{2}}{2}=d\frac{\lambda^{2}}{2}$ for each $\lambda \in V$.
\item The equality $(J\lambda,\mu)=-(\lambda,J\mu)$ holds for any $\lambda$ and $\mu$ in $V$.
\item The pairing $(J\lambda,\lambda)$ vanishes for every element $\lambda \in V$.
\item The expression $\langle\lambda,\mu\rangle:=\frac{(\lambda,\mu)}{2}-\sqrt{-d}\frac{(J\lambda,\mu)}{2d}$ defines a Hermitian pairing $\langle\cdot,\cdot\rangle:V \times V\to\mathbb{K}$.
\item By defining $\langle\lambda,\mu\rangle:=\frac{(\lambda,\mu)}{2}+\sqrt{-d}\frac{(\lambda,J\mu)}{2d}$ one obtains a Hermitian pairing $\langle\cdot,\cdot\rangle:V \times V\to\mathbb{K}$.
\end{enumerate} \label{pairchar}
\end{lem}

\begin{proof}
By assuming $(i)$ and noting that $\langle J\lambda,J\mu \rangle=d\langle\lambda,\mu\rangle$ by the action of $J$ on both sides, we obtain $(ii)$. The equivalence of $(ii)$ and $(iii)$ is an immediate consequence of the relations between a quadratic form and the associated bilinear form in characteristic different from 2. If we assume $(ii)$, then $-d(\lambda,J\mu)=(J^{2}\lambda,J\mu)$ equals $d(J\lambda,\mu)$, and $(iv)$ follows. Conversely, if $(iv)$ holds then $(J\lambda,J\mu)$ equals $-(J^{2}\lambda,\mu)=d(\lambda,\mu)$, completing the equivalence of $(ii)$ and $(iv)$. Moreover, setting $\mu=\lambda$ in $(iii)$ immediately implies $(v)$. As if $(v)$ holds then applying it to $\lambda$, $\mu$, and $\lambda+\mu$ and using the bilinearity and the linearity of $J$ yields $(iv)$, the equivalence of $(iv)$ and $(v)$ also follows.

Next, note that both $(vi)$ and $(vii)$ involve $\mathbb{Q}$-bilinear maps $V \times V\to\mathbb{K}$, and as composing any of them with $\operatorname{Tr}^{\mathbb{K}}_{\mathbb{Q}}$ reproduces $(\cdot,\cdot)$, we deduce that either $(vi)$ or $(vii)$ implies $(i)$. Moreover, note that by replacing $\lambda$ by $J\mu$ in the pairing from $(v)$ we obtain $\frac{(J\lambda,\mu)}{2}+\sqrt{-d}\frac{(\lambda,\mu)}{2}=\sqrt{-d}\langle\lambda,\mu\rangle$ (yielding the $\mathbb{K}$-linearity), while writing $J\mu$ instead of $\mu$ in the pairing from $(vi)$ produces $\frac{(\lambda,J\mu)}{2}-\sqrt{-d}\frac{(\lambda,\mu)}{2}=-\sqrt{-d}\langle\lambda,\mu\rangle$ (showing that this pairing is $\mathbb{K}$-conjugate-linear in the second variable). But as condition $(iv)$ implies that the pairings from $(vi)$ and $(vii)$ coincide, this joint pairing is Hermitian, yielding both $(vi)$ and $(vii)$ as desired. This proves the lemma.
\end{proof}
It is also clear from the proof of Lemma \ref{pairchar} that when the equivalent conditions are satisfied, the pairing yielding Condition $(i)$ is unique, and is given by the formulae from both conditions $(vi)$ and $(vii)$ there (which were seen to be equal).

\smallskip

When extending scalars, all the operations become linear in the extended scalars. Therefore we shall distinguish the action of $\sqrt{-d}\in\mathbb{K}$ from that of $\sqrt{-d}$ in some field of extended scalars (which can be $\mathbb{C}$, or $\mathbb{K}$ again), we shall write $\sqrt{-d}$ for the latter, and use the endomorphism $J$ for the former. We shall thus obtain the following relations between the spaces from Definition \ref{tensprod} and Lemma \ref{VQdecom} and any pairings, related as in Lemma \ref{pairchar}.
\begin{prop}
Let $\big(V,\langle\cdot,\cdot\rangle\big)$ be a Hermitian space over $\mathbb{K}$, with the associated pairing $(\cdot,\cdot):V \times V\to\mathbb{Q}$, and extend the scalars from $\mathbb{Q}$ to $\mathbb{C}$.
\begin{enumerate}[$(i)$]
\item The spaces $V_{\mathbb{C}}$ and $V_{\overline{\mathbb{C}}}$ are both isotropic with respect to $(\cdot,\cdot)$.
\item If $\zeta \in V_{\mathbb{C}}$ then $\langle\lambda,\zeta\rangle=0$ for every $\lambda \in V\otimes_{\mathbb{Q}}\mathbb{C}$, and given $\mu \in V\otimes_{\mathbb{Q}}\mathbb{C}$ we get $\langle\zeta,\mu\rangle=(\zeta,\mu)$. This also equals $\langle\alpha,\mu\rangle$, when $\alpha \in V_{\mathbb{R}}$ corresponds to $\zeta$ via the isomorphism from part $(iii)$ of Lemma \ref{VQdecom}.
\item For $\omega \in V_{\overline{\mathbb{C}}}$ and two elements $\lambda$ and $\mu$ from $V\otimes_{\mathbb{Q}}\mathbb{C}$ we have $\langle\omega,\mu\rangle=0$ and $\langle\lambda,\omega\rangle=(\lambda,\omega)$. If $\omega$ is associated with $\beta \in V_{\mathbb{R}}$ through part $(iii)$ of Lemma \ref{VQdecom}, then the latter pairing coincides with $\langle\lambda,\beta\rangle$.
\item Given $\zeta \in V_{\mathbb{C}}$ and $\omega \in V_{\overline{\mathbb{C}}}$, with $\alpha$ and $\beta$ as in parts $(ii)$ and $(iii)$, we have $\langle\omega,\zeta\rangle=0$ as well as $\langle\zeta,\omega\rangle=(\zeta,\omega)=\langle\alpha,\beta\rangle$.
\end{enumerate} \label{decompair}
\end{prop}

\begin{proof}
Let $\zeta$ and $\omega$ elements that both belong either to $V_{\mathbb{C}}$ or to $V_{\overline{\mathbb{C}}}$, and consider the pairing $(J\zeta,J\omega)$. On one hand, the action of $J$ multiplies both vectors by $\pm\sqrt{-d}$ (with the same sign), meaning, by bilinearity, that this pairing equals $-d(\zeta,\omega)$. As condition $(ii)$ of Lemma \ref{pairchar} implies (after extending scalars) that this pairing also equals $d(\zeta,\omega)$, we must have $(\zeta,\omega)=0$, and part $(i)$ follows. In part $(ii)$ we have $J\zeta=\sqrt{-d}\zeta$, and we express $\langle\lambda,\zeta\rangle$ via part $(vii)$ of Lemma \ref{pairchar} and $\langle\zeta,\mu\rangle$ using part $(vi)$ there. The asserted expressions then follow, and since $\zeta=\frac{\alpha}{2}-\sqrt{-d}\frac{J\alpha}{2d}$ via part $(ii)$ of Lemma \ref{VQdecom}, the relation with $\langle\alpha,\mu\rangle$ is now a consequence of the $\mathbb{K}$-linearity in the first variable. For part $(iii)$ we again use the same parts for writing $\langle\lambda,\zeta\rangle$ and $\langle\omega,\mu\rangle$, and use the fact that $J\omega=-\sqrt{-d}\omega$ and $\omega=\frac{\beta}{2}+\sqrt{-d}\frac{J\beta}{2d}$ (with the conjugate-linearity) to deduce the desired results. Part $(iv)$, together with the additional values $\langle\omega,\omega\rangle$ and $\langle\zeta,\beta\rangle$, is now an immediate consequences of parts $(ii)$ and $(iii)$. This proves the proposition.
\end{proof}

\begin{rmk}
In Definition \ref{tensprod} it suffices to consider the $J$-eigenspaces $V_{\mathbb{K}}$ and $V_{\overline{\mathbb{K}}}$ inside $V\otimes_{\mathbb{Q}}\mathbb{K}$ (and $V$ as the rational subspace inside). Lemmas \ref{VQdecom} and \ref{pairchar}, Corollary \ref{Konspaces}, and Proposition \ref{decompair} hold as stated, with the same formulae and proofs, also in this algebraic setting. However, as the relevant variables of Hermitian Jacobi forms below will be taken from $V\otimes_{\mathbb{Q}}\mathbb{C}$ (or its subspaces that we considered above), we worked in the setting of real and complex vector spaces throughout. \label{overK}
\end{rmk}

\begin{rmk}
The pairing $\langle\cdot,\cdot\rangle$ determines a $\mathbb{K}$-conjugate-linear map from $V$ to the dual space $V^{*}=\operatorname{Hom}_{\mathbb{K}}(V,\mathbb{K})$, by sending $\mu \in V$ to $\lambda\mapsto\langle\lambda,\mu\rangle$. The pairing is \emph{non-degenerate} when this map is injective (hence bijective, since $\dim V<\infty$), and this happens if and only if the orthogonal pairing $(\cdot,\cdot)$, on $V$ as a vector space over $\mathbb{Q}$, is non-degenerate. We then write the signature of $\big(V,\langle\cdot,\cdot\rangle\big)$ as a Hermitian space as $(b_{+},b_{-})$, and the signature of the orthogonal space $\big(V,(\cdot,\cdot)\big)$ over $\mathbb{Q}$ is thus $(2b_{+},2b_{-})$ (indeed, elements of an orthogonal basis for $V$ over $\mathbb{K}$ remain orthogonal over $\mathbb{Q}$, they are also orthogonal to all their $J$-images, and these $J$-images also form an orthogonal set, with the same signature via part $(ii)$ of Lemma \ref{pairchar}, with $d>0$). \label{mapstodual}
\end{rmk}

\smallskip

Let now $\mathcal{O}$ be an order in $\mathbb{K}$, normalize $-d$ to the discriminant of $\mathcal{O}$ (so that $\mathcal{O}$ is spanned over $\mathbb{Z}$ by 1 and $\frac{d+\sqrt{d}}{2}$), and we denote the module $\frac{1}{\sqrt{-d}}\mathcal{O}$ by $\mathcal{D}^{-1}$ (extending the inverse different from the case where $\mathcal{O}$ is the maximal order $\mathcal{O}_{\mathbb{K}}$ in $\mathbb{K}$. The objects we will work with are the following subgroups of $V$.
\begin{defn}
A subgroup $M$ of $V$ is called an \emph{$\mathcal{O}$-lattice} in $V$ if $\mathcal{O}M \subseteq M$; $M$ spans $V$ over $\mathbb{Q}$ (or equivalently over $\mathbb{K}$); $\langle M,M \rangle\subseteq\mathcal{D}^{-1}$; And the multiplication from $\mathcal{O}$ makes $M$ a projective $\mathcal{O}$-module. Noting that $|\lambda|^{2}\in\mathcal{D}^{-1}\cap\mathbb{Q}=\frac{1}{2}\mathbb{Z}$ for every $\lambda$ in an $\mathcal{O}$-lattice $M$, we say that $M$ is \emph{even} when $|\lambda|^{2}\in\mathbb{Z}$ for any $\lambda \in M$. \label{Olat}
\end{defn}

\begin{rmk}
It is clear from Definition \ref{Olat} that any $\mathcal{O}$-lattice is a torsion-free, finitely generated Abelian group with an action of $\mathcal{O}$. When $\mathcal{O}=\mathcal{O}_{\mathbb{K}}$ this already implies the projective property, but when $\mathcal{O}\subsetneq\mathcal{O}_{\mathbb{K}}$ this need not be the case, as the example where $V$ is spanned over $\mathbb{K}$ by two elements, say $u$ and $v$, and $M=\mathcal{O}x\oplus\mathcal{O}_{\mathbb{K}}y$, shows. Such an $\mathcal{O}$-module is not considered an $\mathcal{O}$-lattice in Definition \ref{Olat} (regardless of the assumption on the pairing)---see Remark \ref{notlat} below. \label{projlat}
\end{rmk}

For duals of lattices, we make the following observation.
\begin{lem}
If $M$ is an $\mathcal{O}$-lattice, then an element $y$ satisfies $\langle y,M \rangle\subseteq\mathcal{D}^{-1}$ if and only if $\langle M,y \rangle\subseteq\mathcal{D}^{-1}$ if and only if $(y,M)\subseteq\mathbb{Z}$. \label{duallat}
\end{lem}

\begin{proof}
The fact that $\frac{1}{\sqrt{-d}}$ is inverted by conjugation implies that $\mathcal{D}^{-1}$ is stable under conjugation, yielding the equivalence of the first two conditions. Next, we note that $\mathcal{D}^{-1}$ is also the set of elements $r\in\mathbb{K}$ such that $\operatorname{Tr}^{\mathbb{K}}_{\mathbb{Q}}(r\mathcal{O})\subseteq\mathbb{Z}$ (indeed, the $\mathbb{Z}$-basis $\frac{1}{\sqrt{-d}}$ and $\frac{1+\sqrt{-d}}{2}$ of $\mathcal{D}^{-1}$ is the $\mathbb{Q}$-basis of $\mathbb{K}$ that is dual to our $\mathbb{Z}$-basis for $\mathcal{O}$ via this bilinear form taking $r$ and $s$ to $\operatorname{Tr}^{\mathbb{K}}_{\mathbb{Q}}(rs)$), which means that the third condition follows from any of the two others.

For the converse implication, we use the fact that $M$ is a projective $\mathcal{O}$-module. Thus there is a module $N$ such that $M \oplus N$ is free over $\mathcal{O}$, so in particular $N$ is also torsion-free, and we view $N$ as embedded in $W=N\otimes_{\mathcal{O}}\mathbb{K}$ and $M \oplus N$ is contained in $V \oplus W$. We extend the Hermitian pairing $\langle\cdot,\cdot\rangle$ from $V$ to $V \oplus W$ by letting $W$ be isotropic and perpendicular to $V$, hence so is $N$ and the pairing (either Hermitian or orthogonal) with $M \oplus N$ yields the same values as the one with $M$.

But if $M \oplus N$ is the free module $\bigoplus_{i=1}^{r}\mathcal{O}\lambda_{i}$ then it is freely generated over $\mathbb{Z}$ by the $\lambda_{i}$'s and the elements $\frac{d+J}{2}\lambda_{i}$, $1 \leq i \leq r$. Thus $(\mu,M)=(M,\mu)\subseteq\mathbb{Z}$ if and only if $(M \oplus N,\mu)\subseteq\mathbb{Z}$ if and only if $(\lambda_{i},\mu)$ and $\big(\frac{d+J}{2}\lambda_{i},\mu\big)$ are in $\mathbb{Z}$ for every $1 \leq i \leq r$, and the fact that $\mathcal{D}^{-1}$ is an $\mathcal{O}$-module implies that $\langle M,\mu \rangle\subseteq\mathcal{D}^{-1}$ if and only if $\langle M \oplus N,\mu \rangle\subseteq\mathcal{D}^{-1}$ if and only if $\langle\lambda_{i},\mu\rangle\in\mathcal{D}^{-1}$ for any $1 \leq i \leq r$. But condition $(vi)$ of Lemma \ref{pairchar} expresses $\langle\lambda_{i},\mu\rangle$ as $\frac{(\lambda_{i},\mu)}{2}-\sqrt{-d}\frac{(J\lambda_{i},\mu)}{2d}$, which we can write as $-\frac{\sqrt{-d}}{d}\big(\frac{d+J}{2}\lambda_{i},\mu\big)+\frac{1+\sqrt{-d}}{2}\frac{(\lambda_{i},\mu)}{2}$. As this is a presentation of an element of $\mathbb{K}$ in the basis for $\mathcal{D}^{-1}$ over $\mathbb{Z}$ from above, our assumption is that the coefficients are integral, we indeed obtain that desired condition for establishing the remaining implication. This completes the proof of the lemma.
\end{proof}
We shall denote the set of elements satisfying the equivalent conditions from Lemma \ref{duallat} by $M^{*}$, and call it the \emph{dual lattice} of $M$. This equivalence can thus be phrased as the following consequence.
\begin{cor}
The dual of $M$ as an $\mathcal{O}$-lattice coincides with its dual as an orthogonal lattice over $\mathbb{Z}$. This dual, as a subgroup of $V^{*}$, is also an $\mathcal{O}$-module, which contains the image of $M \subseteq V$ under the map from Remark \ref{mapstodual}, with the quotient $M^{*}/M$ being finite and $M^{*}$ being projective over $\mathcal{O}$ in case the pairings are non-degenerate. \label{Mindual}
\end{cor}

\begin{proof}
The coincidence of duals is just the content of Lemma \ref{duallat}, and as the action of $\mathcal{O}$ on $V$ preserves $M$ and $\mathcal{D}^{-1}$ is an $\mathcal{O}$-module in $\mathbb{K}$, it follows that $M^{*}$ is also preserved under this action, making it an $\mathcal{O}$-module. The fact that it contains the image of $M$ follows from the assumption that $\langle M,M \rangle\subseteq\mathcal{D}^{-1}$ in Definition \ref{Olat}. Finally, as when the pairing is non-degenerate, $M^{*}$ coincides with the subgroup $\operatorname{Hom}_{\mathcal{O}}(M,\mathcal{D}^{-1})$ (or just $\operatorname{Hom}_{\mathbb{Z}}(M,\mathbb{Z})$) of $V^{*}$, which is finitely generated of the same rank as $M$ and its injective image, the finiteness of $M^{*}/M$ in this case follows. This also combines with the fact that $M$ is projective of finite rank over $\mathcal{O}$ and $\mathcal{D}^{-1}=\frac{1}{\sqrt{-d}}\mathcal{O}$ is free of rank 1 over $\mathcal{O}$ to show that $M^{*}$ is also projective. This proves the corollary.
\end{proof}

\begin{rmk}
The proof of Lemma \ref{duallat} used the projectivity assumption for showing that all the conditions there are equivalent. We shall use the module from Remark \ref{projlat} to exemplify what happens without it. Assume that $\lambda$ and $\mu$ there are orthogonal, with $|\lambda|^{2}$ and $|\mu|^{2}$ non-zero inside $\mathcal{D}^{-1}$ (or more precisely, its intersection with $\mathbb{Q}$, namely $\frac{1}{2}\mathbb{Z}$), then the $\mathbb{Z}$-dual of $M$ consists of vectors $a\lambda+b\mu$ with $a$ and $b$ from $\mathbb{K}$ such that $\operatorname{Tr}^{\mathbb{K}}_{\mathbb{Q}}(a)|\lambda|^{2}$ and $\operatorname{Tr}^{\mathbb{K}}_{\mathbb{Q}}(b)|\mu|^{2}$ are in $\mathbb{Z}$. This amounts to $a|\lambda|^{2}$ being in $\mathcal{D}^{-1}$, while $b|\mu|^{2}$ has to be in the usual inverse different $\mathcal{D}_{\mathbb{K}}^{-1}$ of $\mathcal{O}_{\mathbb{K}}$. Thus $M$ only has an $\mathcal{O}$-module structure (and not an $\mathcal{O}_{\mathbb{K}}$-module structure), but $M^{*}$ cannot be described in terms of the values of the Hermitian pairing with $M$ in the non-projective case, showing why the equivalences in Lemma \ref{duallat} need not hold as stated without the projective property from Definition \ref{Olat}. \label{notlat}
\end{rmk}

\section{Hermitian Jacobi--Siegel Theta Functions \label{HJTheta}}

As we will be working in general, indefinite setting, we shall be needing the symmetric space of the corresponding unitary group. Let $\big(V,\langle\cdot,\cdot\rangle\big)$ be a Hermitian space over $\mathbb{K}$ as in Definition \ref{pairdef}, which we henceforth assume to be non-degenerate of some signature $(b_{+},b_{-})$. Then the associated bilinear pairing $(\cdot,\cdot)$ is also non-degenerate, of $V_{\mathbb{R}}$ is $(2b_{+},2b_{-})$, and by considering $V_{\mathbb{R}}$ as a real orthogonal space, we obtain the associated Grassmannian
\begin{equation}
\operatorname{Gr}(V_{\mathbb{R}}):=\big\{V_{\mathbb{R}}=v_{+} \oplus v_{-}\big|v_{+}\gg0,\ v_{-}\ll0,\ v_{+} \perp v_{-}\big\}, \label{Grassdef}
\end{equation}
which is a real manifold of dimension $4b_{+}b_{-}$.

By viewing $V_{\mathbb{R}}$ as a unitary space (with $\sqrt{-d}\in\mathbb{K}\subseteq\mathbb{C}$ acting as the endomorphism $J$), the symmetric space consists of similar decompositions, but of spaces over $\mathbb{C}$ (namely $\mathbb{R}\otimes_{\mathbb{Q}}\mathbb{K}$), thus producing a complex manifold of dimension $b_{+}b_{-}$ (thus with real dimension $2b_{+}b_{-}$), which embeds into the manifold from Equation \eqref{Grassdef} (as in, e.g., \cite{[Ho2]}). In fact, being closed under this action of $\mathbb{C}$ amounts to being preserved under $J$, and conditions $(ii)$ and $(iv)$ of Lemma \ref{pairchar} immediately imply that the orthogonal complement of a $J$-invariant subspace is also $J$-invariant. Thus our complex manifold can be identified with the submanifold
\begin{equation}
\operatorname{Gr}^{J}(V_{\mathbb{R}}):=\big\{V_{\mathbb{R}}=v_{+} \oplus v_{-}\big|v_{+}\gg0,\ v_{-}\ll0,\ v_{+} \perp v_{-},\ Jv_{+} \subseteq v_{+},\ Jv_{-} \subseteq v_{-}\big\}, \label{GrassJ}
\end{equation}
of the one from Equation \eqref{Grassdef}, whose elements are characterized by $J$-invariance.

Given an element of $\operatorname{Gr}^{J}(V_{\mathbb{R}})$ from Equation \eqref{GrassJ}, any element $\lambda$ in $V$ (or in $V_{\mathbb{R}}$, or of the two summands $V_{\mathbb{C}}$ and $V_{\overline{\mathbb{C}}}$ from Definition \ref{tensprod}, or of their direct sum $V\otimes_{\mathbb{Q}}\mathbb{C}$---see part $(i)$ of Lemma \ref{VQdecom}) decomposes as the sum of its orthogonal projections $\lambda_{v_{+}}$ and $\lambda_{v_{-}}$ onto the spaces $v_{+}$ and $v_{-}$ respectively. This decomposition is the same regardless of whether we consider $v$ as a unitary decomposition, or as an orthogonal one as in Equations \eqref{Grassdef} and \eqref{GrassJ}. We thus have that $|\lambda|^{2}=\frac{\lambda^{2}}{2}$ is the sum of $|\lambda_{v_{+}}|^{2}=\lambda_{v_{+}}^{2}/2\geq0$ and of $|\lambda_{v_{-}}|^{2}=\lambda_{v_{-}}^{2}/2\leq0$, with $\langle\lambda_{v_{+}},\lambda_{v_{-}}\rangle=(\lambda_{v_{+}},\lambda_{v_{-}})=0$.

\smallskip

Let now $M$ be an even $\mathcal{O}$-lattice in $V$ as in Definition \ref{Olat}, with dual $M^{*}$ as in Lemma \ref{duallat} and Corollary \ref{Mindual}, and we shall henceforth denote the quotient $M^{*}/M$ by $D_{M}$. We recall that a matrix $A=\binom{a\ \ b}{c\ \ d}\in\operatorname{SL}_{2}(\mathbb{R})$ acts on the upper half-plane $\mathcal{H}:=\{\tau=x+iy\in\mathbb{C}|y>0\}$ by the formula $A\tau=\frac{a\tau+b}{c\tau+d}$, with the factor of automorphy $j(A,\tau):=c\tau+d$, and we use the notation $T:=\binom{1\ \ 1}{0\ \ 1}$ and $S:=\binom{0\ \ -1}{1\ \ \ 0\ }$ for the generators of the group $\operatorname{SL}_{2}(\mathbb{Z})$, as usual. Now, on $M$ as an even lattice over $\mathbb{Z}$, and on $M^{*}$, the pairing $(\cdot,\cdot)$ and the quadratic form $\lambda\mapsto\frac{\lambda^{2}}{2}$ yield well-defined $\mathbb{Q}/\mathbb{Z}$-valued functions on the quotient $D_{M}$, for which we shall use the same notation. To $M$ one then attaches the \emph{Weil representation} of $\operatorname{SL}_{2}(\mathbb{Z})$ on $\mathbb{C}[D_{M}]$, which is defined on the generators by
\begin{equation}
\rho_{M}(T)\mathfrak{e}_{\gamma}=\mathbf{e}\big(\tfrac{\gamma^{2}}{2}\big)\mathfrak{e}_{\gamma},\qquad\rho_{M}(S)\mathfrak{e}_{\gamma}=\frac{i^{b_{-}-b_{+}}}{\sqrt{|D_{M}|}}\sum_{\delta \in D_{M}}\mathbf{e}\big(-(\gamma,\delta)\big)\mathfrak{e}_{\delta}, \label{Weildef}
\end{equation}
where we have introduced the shorthand $\mathbf{e}(z):=e^{2\pi iz}$ for any $z$ from $\mathbb{C}$ or from $\mathbb{C}/\mathbb{Z}$ (for the root of unity in $\rho_{M}(S)$, recall that the signature of $M$ over $\mathbb{Z}$ is $(2b_{+},2b_{-})$, and the fact that $M$ has even rank over $\mathbb{Z}$ implies that this is indeed a representation of $\operatorname{SL}_{2}(\mathbb{Z})$, and no metaplectic double covers have to be considered).

\smallskip

Hermitian theta functions and Jacobi forms play a role in the theory of Hermitian modular forms, where the natural groups acting are the unitary groups of the following type.
\begin{defn}
For any $g\geq1$, the unitary group $\operatorname{U}(g,g)$ is the group of $2g\times2g$ matrices $A$ in the $\sqrt{-d}$-extension of the ring of definition which satisfy $A\binom{0\ \ -I_{g}}{I_{g}\ \ \ 0\ }A^{*}=\binom{0\ \ -I_{g}}{I_{g}\ \ \ 0\ }$, where $A^{*}$ is the conjugate transpose $\overline{A}^{t}$ of $A$. In particular, for $\operatorname{U}(g,g)(\mathbb{R})$ we consider matrices in $\operatorname{GL}_{2g}(\mathbb{C})$ satisfying this equality, the group $\operatorname{U}(g,g)(\mathbb{Q})$ contains matrices from $\operatorname{GL}_{2g}(\mathbb{K})$ with this property, and for $\operatorname{U}(g,g)(\mathbb{Z})$, the entries of the matrices have to come from $\mathcal{O}_{\mathbb{K}}$. The group $\operatorname{SU}(g,g)$ is intersection with the corresponding $\operatorname{SL}_{2g}$ group. \label{unitgrp}
\end{defn}

The Hermitian modular forms containing Hermitian Jacobi forms in their Fourier--Jacobi expansion involve the groups $\operatorname{U}(g,g)$ and $\operatorname{SU}(g,g)$ from Definition \ref{unitgrp} with $g\geq2$. However, we will require the groups $\operatorname{U}(1,1)$ and $\operatorname{SU}(1,1)$, especially over $\mathbb{Z}$. We now make the following identifications of groups, already alluded to in \cite{[Ho1]} and others.
\begin{lem}
We have $\operatorname{SU}(1,1)(\mathbb{Z})=\operatorname{SL}_{2}(\mathbb{Z})$, and $\operatorname{U}(1,1)(\mathbb{Z})=\mu_{\mathbb{K}}\cdot\operatorname{SL}_{2}(\mathbb{Z})$, where $\mu_{\mathbb{K}}$ is the group of roots of unity in $\mathbb{K}$, viewed as scalar $2\times2$ matrices. \label{SU11SL2}
\end{lem}

\begin{proof}
If $A=\binom{a\ \ b}{c\ \ d}\in\operatorname{U}(1,1)(\mathbb{R})$ then the complex numbers $a\overline{b}$ and $c\overline{d}$ are real. Since $A^{*}$ also lies in that group, we deduce that $\overline{a}c$ and $\overline{b}d$ are real as well. It follows that $a$, $b$, $c$, and $d$ lie in a single real line inside $\mathbb{C}$, and we can thus write $A$ as the product of a scalar from $S^{1}$ times a real matrix. Moreover, the $\operatorname{U}(1,1)$-condition also implies that $a\overline{d}-b\overline{c}=1$, and as the scalar multiplier from $S^{1}$ does not affect the latter equality, we deduce that the real matrix is from $\operatorname{SL}_{2}(\mathbb{R})$. As one easily verifies that matrices from $\operatorname{SL}_{2}(\mathbb{R})$ and scalar matrices from $S^{1}$ both lie in $\operatorname{U}(1,1)$ (via Definition \ref{unitgrp}), it follows that $\operatorname{U}(1,1)(\mathbb{R})=S^{1}\cdot\operatorname{SL}_{2}(\mathbb{R})$ (this is not a direct product, as $-I$ lies in the intersection of these groups).

Now, the determinant of an element from $S^{1}\cdot\operatorname{SL}_{2}(\mathbb{R})$ is just the square of the scalar from $S^{1}$, yielding that $\operatorname{SU}(1,1)(\mathbb{R})=\operatorname{SL}_{2}(\mathbb{R})$. In particular, by intersecting with $\operatorname{SL}_{2}(\mathcal{O}_{\mathbb{K}})$, the first assertion follows. For the second one, we saw that if $A=\binom{a\ \ b}{c\ \ d}\in\operatorname{U}(1,1)(\mathbb{Z})$ then $a$, $b$, $c$, and $d$ from $\mathcal{O}_{\mathbb{K}}$ all lie in one real line, hence they are all in the same additive subgroup of $\mathcal{O}_{\mathbb{K}}$, and if $\alpha$ generates this subgroup then they all lie in the ideal generated by $\alpha$. But then the determinant $ad-bc$, which is invertible in $\mathcal{O}_{\mathbb{K}}$ for a matrix from $\operatorname{GL}_{2}(\mathcal{O}_{\mathbb{K}})$, must be a multiple of $\alpha^{2}$, meaning that $\alpha$ has to be a unit in $\mathcal{O}_{\mathbb{K}}$ hence a root of unity. But then this $\alpha$ lies in $S^{1}$, and we obtain that our matrix is in $\mu_{\mathbb{K}}\cdot\operatorname{SL}_{2}(\mathbb{Z})$, as desired. This proves the lemma.
\end{proof}
Note that over $\mathbb{Q}$ we get $\operatorname{SU}(1,1)(\mathbb{Q})=\operatorname{SL}_{2}(\mathbb{Q})$ as well (in fact, $\operatorname{SU}(1,1)$ is the same as $\operatorname{SL}_{2}$ as algebraic groups over $\mathbb{Z}$), but $\operatorname{U}(1,1)(\mathbb{Q})$ does not decompose as the product of $\operatorname{SL}_{2}(\mathbb{Q})$ and the group of norm 1 elements of $\mathbb{K}$.

Via Definition \ref{unitgrp} and Lemma \ref{SU11SL2}, we may view the Weil representation from Equation \eqref{Weildef} as a representation of $\operatorname{SU}(1,1)(\mathbb{Z})$. Moreover, when the discriminant of $\mathbb{K}$ is not $-3$ or $-4$ (or more precisely, when $\mathcal{O}$ is not $\mathbb{Z}\big[\frac{1+\sqrt{-3}}{2}\big]$ or $\mathbb{Z}[i]$), the group of roots of unity in $\mathcal{O}_{\mathbb{K}}$ (or in $\mathcal{O}$) is just $\pm1$, so that this representation is of $\operatorname{U}(1,1)(\mathbb{Z})$, when the order in question is $\mathcal{O}$, over which $M$ is projective. However, it will be useful for us to extend the representation to $\operatorname{U}(1,1)(\mathbb{Z})$ also for the two orders containing more roots of unity, as follows.
\begin{prop}
When $M$ is an $\mathcal{O}$-lattice, any $\xi$ in the group $\mu(\mathcal{O})$ of roots of unity in $\mathcal{O}$ acts on $D_{M}$, and by setting $\rho_{M}(\xi)\mathfrak{e}_{\gamma}=\xi^{b_{-}-b_{+}}\mathfrak{e}_{\overline{\xi}\gamma}$ we obtain an extension of the representation from Equation \eqref{Weildef} to $\operatorname{U}(1,1)(\mathbb{Z})$. \label{extWeil}
\end{prop}
We remark again that the signature of $M$ over $\mathbb{Z}$ is even, meaning that no double covers are required, and the extension from Proposition \ref{extWeil} is indeed to the group $\operatorname{U}(1,1)(\mathbb{Z})$.
\begin{proof}
As $\xi\in\mathcal{O}$ acts both on $M$ and on $M^{*}$ (by Corollary \ref{Mindual}), we obtain its action also on $D_{M}$. The fact that our formula for $\rho_{M}$ yields a representation of $\mu(\mathcal{O})$ is evident, and we first need to make sure that on the intersection $\pm I$, the formula resulting from Equation \eqref{Weildef} coincides with the one we defined. But with the even signature $2b_{+}-2b_{-}$ of $M$, the we have $\rho_{M}(-I)\mathfrak{e}_{\gamma}=(-1)^{b_{-}-b_{+}}\mathfrak{e}_{-\gamma}$, which coincides with $\rho_{M}(-1)$ for $-1\in\mu(\mathcal{O})$, and the action of $I$ (or 1) is trivial in both.

As the multipliers $\operatorname{SU}(1,1)(\mathbb{Z})=\operatorname{SL}_{2}$ and $\mu(\mathcal{O})$ commute inside $\operatorname{U}(1,1)(\mathbb{Z})$, we will obtain a representation of $\operatorname{U}(1,1)(\mathbb{Z})$ with these formulae once we show that $\rho_{M}(\xi)$ commutes with $\rho_{M}(A)$ for every $A\in\operatorname{SL}_{2}$, and it suffices to verify this for $A=T$ and $A=S$. We first note that $\frac{\gamma^{2}}{2}$ is the $\mathbb{Q}/\mathbb{Z}$-image of $\frac{\lambda^{2}}{2}=|\lambda|^{2}$ for $\lambda \in M+\gamma$, and that $|\overline{\xi}\lambda|^{2}=\langle\overline{\xi}\lambda,\overline{\xi}\lambda\rangle=|\overline{\xi}|^{2}\langle\lambda,\lambda\rangle=|\lambda|^{2}$ (with $|\overline{\xi}|^{2}=1$ as a root of unity in $\mathbb{K}$). This means that $\frac{(\overline{\xi}\gamma)^{2}}{2}=\frac{\gamma^{2}}{2}\in\mathbb{Q}/\mathbb{Z}$ for every $\gamma \in D_{M}$, and therefore $\rho_{M}(T)$ multiplies both $\mathfrak{e}_{\gamma}$ and $\xi^{b_{-}-b_{+}}\mathfrak{e}_{\overline{\xi}\gamma}=\rho_{M}(\xi)\mathfrak{e}_{\gamma}$ by the same scalar $\mathbf{e}\big(\frac{\gamma^{2}}{2}\big)$, whence it commutes with $\rho_{M}(\xi)$ on all the generators of $\mathbb{C}[D_{M}]$.

For $S$, applying Equation \eqref{Weildef} to $\rho_{M}(\xi)\mathfrak{e}_{\gamma}=\xi^{b_{-}-b_{+}}\mathfrak{e}_{\overline{\xi}\gamma}$ produces \[\frac{(i\xi)^{b_{-}-b_{+}}}{\sqrt{|D_{M}|}}\sum_{\delta \in D_{M}}\mathbf{e}\big(-(\overline{\xi}\gamma,\delta)\big)\mathfrak{e}_{\delta}=\frac{(i\xi)^{b_{-}-b_{+}}}{\sqrt{|D_{M}|}}\sum_{\delta \in D_{M}}\mathbf{e}\big(-(\overline{\xi}\gamma,\overline{\xi}\delta)\big)\mathfrak{e}_{\overline{\xi}\delta}\] for any $\gamma \in D_{M}$, where we replaced the summation index $\delta$ by $\overline{\xi}\delta$. But as the same argument from the previous paragraph shows that $\langle\overline{\xi}\lambda,\overline{\xi}\mu\rangle=\langle\lambda,\mu\rangle\in\mathbb{K}$ for $\lambda \in M+\gamma$ and $\mu \in M+\delta$, we obtain that $(\overline{\xi}\lambda,\overline{\xi}\mu)=(\lambda,\mu)\in\mathbb{Q}$ (via Definition \ref{pairdef}) and hence $(\overline{\xi}\gamma,\overline{\xi}\delta)=(\gamma,\delta)\in\mathbb{Q}/\mathbb{Z}$. But this compares the latter expression with \[\frac{i^{b_{-}-b_{+}}}{\sqrt{|D_{M}|}}\sum_{\delta \in D_{M}}\mathbf{e}\big(-(\gamma,\delta)\big)\xi^{b_{-}-b_{+}}\mathfrak{e}_{\overline{\xi}\delta}=\rho_{M}(\xi)\rho_{M}(S)\mathfrak{e}_{\gamma},\] meaning that $\rho_{M}(\xi)$ commutes with $\rho_{M}(S)$ as well. This completes the proof of the proposition.
\end{proof}
We shall thus henceforth consider the Weil representation $\rho_{M}$ to be the extension of that from Equation \eqref{Weildef} described in Proposition \ref{extWeil}, namely a representation of $\operatorname{U}(1,1)(\mathbb{Z})$ (recalling that except for two possible choices of the order $\mathcal{O}$, this is just a change of notation for the original group and representation).

\smallskip

We shall write $V_{\mathbb{C}}$, $V_{\overline{\mathbb{C}}}$, and $V_{\mathbb{R}}$ from Definition \ref{tensprod} as $M_{\mathbb{C}}$, $M_{\overline{\mathbb{C}}}$, and $M_{\mathbb{R}}$ respectively (indeed, they are $M\otimes_{\mathcal{O}}\mathbb{C}$, the same tensor product with the complex conjugate embedding, and $M\otimes_{\mathbb{Z}}\mathbb{R}$ respectively). For an element $\tau\in\mathcal{H}$, vectors $\zeta \in M_{\mathbb{C}}$ and $\omega \in M_{\overline{\mathbb{C}}}$, and an element $v$ in $\operatorname{Gr}^{J}(M_{\mathbb{R}})$ from Equation \eqref{GrassJ}, the (vector-valued) \emph{Hermitian Jacobi--Siegel theta function} associated with $M$ is defined to be
\begin{equation}
\Theta_{M}(\tau,\zeta,\omega;v):=\sum_{\gamma \in D_{M}}\textstyle{\big[\sum_{\lambda \in M+\gamma}\mathbf{e}\big(\tau|\lambda_{v_{+}}|^{2}+\overline{\tau}|\lambda_{v_{-}}|^{2}+\langle\zeta,\lambda\rangle+\langle\lambda,\omega\rangle\big)\big]\mathfrak{e}_{\gamma}}, \label{JacTheta}
\end{equation}
where the coefficient multiplying $\mathfrak{e}_{\gamma}$ in Equation \eqref{JacTheta} is the (scalar-valued) Hermitian Jacobi--Siegel theta function $\theta_{M+\gamma}(\tau,\zeta,\omega;v)$.

When $M$ is considered as an orthogonal lattice over $\mathbb{Z}$, we can identify the element $v\in\operatorname{Gr}^{J}(M_{\mathbb{R}})$ with its image in $\operatorname{Gr}(M_{\mathbb{R}})$ from Equation \eqref{Grassdef}, and the pair $\zeta$ and $\omega$ with the vector $\zeta+\omega \in M\otimes_{\mathbb{Z}}\mathbb{C}$ as in part $(i)$ of Lemma \ref{VQdecom}. Then Equation (7) of \cite{[Ze3]}, with these variables, defines the (orthogonal) Siegel--Jacobi theta function
\begin{equation}
\widetilde{\Theta}_{M}(\tau,\zeta+\omega;v):=\sum_{\gamma \in D_{M}}\Bigg[\sum_{\lambda \in M+\gamma}\mathbf{e}\bigg(\tau\frac{\lambda_{v_{+}}^{2}}{2}+\overline{\tau}\frac{\lambda_{v_{-}}^{2}}{2}+(\lambda,\zeta+\omega)\bigg)\Bigg]\mathfrak{e}_{\gamma} \label{orthTheta}
\end{equation}
(we will henceforth use $\Theta$ and $\theta$ for Hermitian theta functions, and $\widetilde{\Theta}$, and possibly $\widetilde{\theta}$, for their orthogonal counterparts, to emphasize our point of view at every step).

All the properties of our $\Theta_{M}$ essentially follow from the following comparison, using the results of \cite{[Ze3]}. However, the form that these properties take can be a bit more delicate, so we will state most of them explicitly in the Hermitian setting.
\begin{lem}
The function from Equation \eqref{JacTheta} coincides with the orthogonal Siegel--Jacobi theta function from Equation \eqref{orthTheta}. \label{compThetaM}
\end{lem}

\begin{proof}
The projections $\lambda_{v_{\pm}}$ were seen to be the same in both interpretations of $v$, and the Hermitian norms $|\lambda_{v_{\pm}}|^{2}$ become $\lambda_{v_{\pm}}^{2}/2$ using the quadratic form. Moreover, parts $(ii)$ and $(iii)$ of Proposition \ref{decompair} compare the pairings $\langle\zeta,\lambda\rangle$ and $\langle\lambda,\omega\rangle$ with $(\zeta,\lambda)=(\lambda,\zeta)$ and $(\lambda,\omega)$ respectively. As $M^{*}$ is the same dual for $M$ as an $\mathcal{O}$-lattice and as a $\mathbb{Z}$-lattice (by Corollary \ref{Mindual}), so are the two meanings of the quotient $D_{M}$ and of the cosets $M+\gamma$. Thus both the set of elements $\lambda$ that show up in the two equations, and the summands associated with each $\lambda$ in them, coincide, yielding the asserted equality. This proves the lemma.
\end{proof}

\smallskip

Given any element $\sigma \in M\otimes_{\mathbb{Z}}\mathbb{C}$, recall that we have the projections onto the spaces $v_{\pm}$ associated with $v\in\operatorname{Gr}^{J}(M_{\mathbb{R}})$ (or their extensions of scalars to $\mathbb{C}$), as well as the projections onto the components $M_{\mathbb{C}}$ and $M_{\overline{\mathbb{C}}}$, as in Lemma \ref{VQdecom}. These projections commute, i.e., we have $(\sigma_{v_{\pm}})_{\mathbb{C}}=(\sigma_{\mathbb{C}})_{v_{\pm}}$ and $(\sigma_{v_{\pm}})_{\overline{\mathbb{C}}}=(\sigma_{\overline{\mathbb{C}}})_{v_{\pm}}$, and we will write these vectors as $\sigma_{v_{\pm},\mathbb{C}}$ and $\sigma_{v_{\pm},\overline{\mathbb{C}}}$ for short. The first property of $\Theta_{M}$ is the ``periodicity'' in $\zeta$ and $\omega$, as in Proposition 1.1 of \cite{[Ze3]}, but we shall need (especially for the vector-valued case later) to extend to translations in elements from $M^{*}$ as well. In order to do this, we make the following definition.
\begin{defn}
Let $\alpha$ and $\beta$ be elements of $D_{M}$. We then define the \emph{translation operator} $t_{\alpha}$ to be the linear operator on $\mathbb{C}[D_{M}]$ taking the basis vector $\mathfrak{e}_{\gamma}$ to $\mathfrak{e}_{\gamma-\alpha}$, and the \emph{character operator} $\chi_{\beta}$ to be the one multiplying each $\mathfrak{e}_{\gamma}$ by $\mathbf{e}\big(\langle\beta,\gamma\rangle+\langle\gamma,\beta\rangle\big)$. Note that the latter is just $\mathbf{e}\big((\beta,\gamma)\big)$ when $M$ is viewed as an orthogonal lattice, and both operators are defined on $D_{M}$ for any orthogonal lattice $M$ (or more precisely, on any discriminant module in the language of \cite{[Sch]} and \cite{[Ze1]}, or finite quadratic module in the terminology of \cite{[Str]}). For $\sigma$ and $\nu$ in $M^{*}$, we shall write just $t_{\sigma}$ and $\chi_{\nu}$ for the operators $t_{\sigma+M}$ and $\chi_{\nu+M}$ respectively. They are trivial if and only if $\sigma$ and $\nu$ are in $M$. \label{opersDM}
\end{defn}

\begin{rmk}
While $\alpha \mapsto t_{\alpha}$ and $\beta\mapsto\chi_{\beta}$ from Definition \ref{opersDM} are group homomorphisms from $D_{M}$ into $\operatorname{GL}_{2}\big(\mathbb{C}[D_{M}]\big)$ (and even into the subgroup of unitary operators), they do not form a homomorphism from the product, but rather from an appropriate Heisenberg group. This is related to the fact that the periodicity of theta functions and Jacobi forms is based, inside the larger, Siegel or Hermitian modular group, on the action of a Heisenberg group, rather that the Abelian self-product of the lattice (this is more clearly visible in the associated Lie group). However, we will only be using the combination $\chi_{\beta}t_{\alpha}$, sending $\mathfrak{e}_{\gamma}$ to $\mathbf{e}\big(\langle\beta,\gamma-\alpha\rangle+\langle\gamma-\alpha,\beta\rangle\big)\mathfrak{e}_{\gamma-\alpha}$ (or just $\mathbf{e}\big((\beta,\gamma-\alpha)\big)\mathfrak{e}_{\gamma-\alpha}$). In particular, if an element $X$ of $\mathbb{C}[D_{M}]$ is defined by the absolutely convergent series $\sum_{\lambda \in M^{*}}a_{\lambda}\mathfrak{e}_{\lambda+M}$ (like our theta functions), then the series $\sum_{\lambda \in M^{*}}a_{\lambda+\sigma}\mathbf{e}\big(\langle\nu,\lambda\rangle+\langle\lambda,\nu\rangle\big)\mathfrak{e}_{\lambda+M}$ (or just $\sum_{\lambda \in M^{*}}a_{\lambda+\sigma}\mathbf{e}\big((\lambda,\nu)\big)\mathfrak{e}_{\lambda+M}$ in the orthogonal notation) produces $\chi_{\nu}t_{\sigma}X$. \label{chitrels}
\end{rmk}

We can now state and prove the periodicity formula for our theta functions.
\begin{prop}
Take two elements $\sigma$ and $\nu$ from $M^{*}$, and let $\tau\in\mathcal{H}$ and $v\in\operatorname{Gr}^{J}(M_{\mathbb{R}})$ be fixed. Then, as a function of $\zeta \in M_{\mathbb{C}}$ and $\omega \in M_{\overline{\mathbb{C}}}$, the Hermitian Jacobi--Siegel theta function from Equation \eqref{JacTheta} satisfies the equality \[\Theta_{M}(\tau,\zeta+\tau\sigma_{v_{+},\mathbb{C}}+\overline{\tau}\sigma_{v_{-},\mathbb{C}}+\nu_{\mathbb{C}},\omega+\tau\sigma_{v_{+},\overline{\mathbb{C}}}+\overline{\tau}\sigma_{v_{-},\overline{\mathbb{C}}}+ \nu_{\overline{\mathbb{C}}};v)=\] \[=\mathbf{e}\big(-\tau|\sigma_{v_{+}}|^{2}-\overline{\tau}|\sigma_{v_{-}}^{2}|-\langle\zeta,\sigma\rangle-\langle\sigma,\omega\rangle\big)\chi_{\nu}t_{\sigma}\Theta_{M}(\tau,\zeta,\omega;v).\] When $\sigma$ and $\nu$ are in $M$, the latter multiplier is just $\Theta_{M}(\tau,\zeta,\omega;v)$ itself. \label{perTheta}
\end{prop}

\begin{proof}
We apply Lemma \ref{compThetaM} to work with $\widetilde{\Theta}_{M}$ from Equation \eqref{orthTheta}, and, as in the proof of Proposition 1.1 of \cite{[Ze3]}, we consider the summand associated with some $\lambda \in M^{*}$ but in which we add $\tau\sigma_{v_{+},\mathbb{C}}+\overline{\tau}\sigma_{v_{-},\mathbb{C}}+\nu_{\mathbb{C}}$ to $\zeta$ as well as $\tau\sigma_{v_{+},\overline{\mathbb{C}}}+\overline{\tau}\sigma_{v_{-},\overline{\mathbb{C}}}+\nu_{\overline{\mathbb{C}}}$ to $\omega$. Since every vector is the sum of its projections onto $M_{\mathbb{C}}$ and $M_{\overline{\mathbb{C}}}$, this amounts to adding just $\tau\sigma_{v_{+}}+\overline{\tau}\sigma_{v_{-}}+\nu$ to the sum $\zeta+\omega$, the pairing with which shows up in Equation \eqref{orthTheta}. The fact that the projections onto $v_{\pm}$ are orthogonal implies that this summand is multiplied by $\mathbf{e}\big(\tau(\lambda_{v_{+}},\sigma_{v_{+}})+\overline{\tau}(\lambda_{v_{-}},\sigma_{v_{-}})\big)$ and also by $\mathbf{e}\big(\langle\nu,\lambda\rangle+\langle\lambda,\nu\rangle\big)=\mathbf{e}\big((\lambda,\nu)\big)$.

Then this summand involves, aside from the exponent of $(\lambda,\zeta+\omega)$, also that of $\tau$ or $\overline{\tau}$ multiplied by $\frac{\lambda_{v_{\pm}}^{2}}{2}+(\lambda_{v_{\pm}},\sigma_{v_{\pm}})=\frac{(\lambda+\sigma)_{v_{\pm}}^{2}}{2}-\frac{\sigma_{v_{\pm}}^{2}}{2}$. We thus make the change of variable $\lambda\mapsto\lambda+\sigma$, with the pairing with $\lambda$ becomes that with $\lambda+\sigma$ minus that with $\sigma$, and then the multipliers that are independent of $\lambda$ reduce to the exponent of $-\tau\sigma_{v_{+}}^{2}/2-\overline{\tau}\sigma_{v_{+}}^{2}/2-(\sigma,\zeta+\omega)$. For the remaining sum over $\lambda$ we obtain via Remark \ref{chitrels}, due the translation by $\sigma$ and the multiplier from the previous paragraph, the image under $\chi_{\nu}t_{\sigma}$ of the theta function $\widetilde{\Theta}_{M}(\tau,\zeta+\omega;v)$, which becomes the desired one on the right hand side via Lemma \ref{compThetaM}, whose proof also shows that the exponent is the asserted one as well. This establishes the desired general formula, from which the special case with $\sigma$ and $\nu$ in $M$ follows directly because $\chi_{\nu}$ and $t_{\sigma}$ from Definition \ref{opersDM} are then trivial. This proves the proposition.
\end{proof}
Of course, the Proposition 1.1 of \cite{[Ze3]} can be extended, as in Proposition \ref{perTheta}, to show that the theta function from Equation (7) there satisfies such a periodicity for $\sigma$ and $\nu$ in the dual lattice, using the combination $\chi_{\nu}t_{\sigma}$ from Definition \ref{opersDM} and Remark \ref{chitrels}.

\smallskip

The second property of our Hermitian theta function $\Theta_{M}$ is the following modularity transformation.
\begin{thm}
For every matrix $A\in\operatorname{U}(1,1)(\mathbb{Z})$, and every elements $\tau\in\mathcal{H}$, $\zeta \in V_{\mathbb{C}}$, and $\omega \in V_{\overline{\mathbb{C}}}$, we have the equality \[\Theta_{M}\Big(A\tau,\tfrac{\zeta_{v_{+}}}{j(A,\tau)}+\tfrac{\zeta_{v_{-}}}{\overline{j^{c}(A,\tau)}},\tfrac{\omega_{v_{+}}}{j^{c}(A,\tau)}+\tfrac{\omega_{v_{-}}}{\overline{j(A,\tau)}};v\Big)=\] \[=j(A,\tau)^{b_{+}}\overline{j(A,\tau)}^{b_{-}}\mathbf{e}\Big(\tfrac{j_{A}'\langle\zeta_{v_{+}},\omega_{v_{+}}\rangle}{j(A,\tau)}+\tfrac{\overline{j_{A}'}\langle\zeta_{v_{-}},\omega_{v_{-}}\rangle}{\overline{j(A,\tau)}}\Big) \rho_{M}(A)\Theta_{M}(\tau,\zeta,\omega;v),\] where we recall that if $A=\binom{a\ \ b}{c\ \ d}$ then $j(A,\tau)=c\tau+d$, so that $j_{A}'$ is the derivative $c$ with respect to $\tau$ of $j_{A}(\tau)=j(A,\tau)$, and $j^{c}(A,\tau)$ stands for $\overline{c}\tau+\overline{d}$. \label{modTheta}
\end{thm}
Recall that unless the discriminant of $\mathcal{O}$ is in $\{-3,-4\}$, the group $\operatorname{U}(1,1)(\mathbb{Z})$ is $\operatorname{SU}(1,1)(\mathbb{Z})=\operatorname{SL}_{2}(\mathbb{Z})$, for elements $A$ of which we have $j^{c}(A,\tau)=j(A,\tau)$ for any $\tau\in\mathcal{H}$. However, the form given in Theorem \ref{modTheta} is the one that holds also for these exceptional cases, and is the one whose extension to Jacobi forms that are not necessarily theta functions generalizes better to Hermitian Jacobi forms that are modular with respect to $\operatorname{U}(g,g)$ for larger $g$.

\begin{proof}
Consider again the theta function $\widetilde{\Theta}_{M}$ from Equation \eqref{orthTheta}, take an element $A\in\operatorname{SL}_{2}(\mathbb{Z})$, and then for such $\tau$, $\zeta$, and $\omega$, we obtain from Theorem 1.2 of \cite{[Ze3]} (with the signature $(2b_{+},2b_{-})$, rather than $(2b_{+},b_{-})$ as in that reference) the fact that $\widetilde{\Theta}_{M}\big(A\tau,(\zeta+\omega)_{v_{+}}\big/j(A,\tau)+(\zeta+\omega)_{v_{-}}\big/\overline{j(A,\tau)};v\big)$ equals \[j(A,\tau)^{b_{+}}\overline{j(A,\tau)}^{b_{-}}\mathbf{e}\Big(\tfrac{j_{A}'(\zeta+\omega)_{v_{+}}^{2}}{2j(A,\tau)}+\tfrac{j_{A}'(\zeta+\omega)_{v_{-}}^{2}}{2\overline{j(A,\tau)}}\Big)\rho_{M}(A) \widetilde{\Theta}_{M}(\tau,\zeta+\omega;v).\] Lemma \ref{compThetaM} then compares this left hand side to the desired one and the theta function on the right hand side to that of the required one, we have the equalities $j^{c}(A,\tau)=j(A,\tau)$ and $\overline{j_{A}'}=j_{A}'$, and since the $M_{\mathbb{C}}$-part of $(\zeta+\omega)_{v_{\pm}}$ is $\zeta_{v_{\pm}}$ and its $M_{\overline{\mathbb{C}}}$-part is $\omega_{v_{\pm}}$, parts $(i)$ and $(ii)$ of Proposition \ref{decompair} shows that the argument of the exponent is also the asserted one. This establishes the case $A\in\operatorname{SL}_{2}(\mathbb{Z})$ (in fact, by using the appropriate co-cycle relations the proof for such an $A$ one could have reduced a direct proof to checking the very straightforward case $A=T$ and the case $A=S$, the latter of which amounts, as usual, to the classical calculations using Fourier transforms and the Poisson summation formula).

Assume now that $A$ is the product of the scalar $\xi\in\mu(\mathcal{O})$ and the matrix $B\in\operatorname{SL}_{2}(\mathbb{Z})$. Then $A\tau=B\tau$ (since $\xi$ is a scalar matrix), $j(A,\tau)=\xi j(B,\tau)$, and $j^{c}(A,\tau)=\overline{\xi}j(B,\tau)$, and Corollary \ref{Konspaces} shows that the action of $\xi\in\mathcal{O}$ on $\zeta$ multiplies it by $\xi$, but it multiplies $\omega$ by $\overline{\xi}=1/\xi$. This compares the left hand side with $\Theta_{M}\Big(B\tau,\tfrac{\overline{\xi}\zeta_{v_{+}}}{j(B,\tau)}+\tfrac{\overline{\xi}\zeta_{v_{-}}}{\overline{j(B,\tau)}},\tfrac{\overline{\xi}\omega_{v_{+}}}{j(B,\tau)}+ \tfrac{\overline{\xi}\omega_{v_{-}}}{\overline{j(B,\tau)}};v\Big)$, and by what we saw for $B\in\operatorname{SL}_{2}(\mathbb{Z})$ the latter equals \[j(B,\tau)^{b_{+}}\overline{j(B,\tau)}^{b_{-}}\mathbf{e}\Big(\tfrac{j_{B}'\langle\overline{\xi}\zeta_{v_{+}},\overline{\xi}\omega_{v_{+}}\rangle}{j(B,\tau)}+ \tfrac{j_{B}'\langle\overline{\xi}\zeta_{v_{-}},\overline{\xi}\omega_{v_{-}}\rangle}{\overline{j(B,\tau)}}\Big)\rho_{M}(B)\Theta_{M}(\tau,\overline{\xi}\zeta,\overline{\xi}\omega;v).\] We have $j_{A}'/j(A,\tau)=j_{B}'/j(B,\tau)$ (by canceling $\xi$) and $\overline{j_{A}'}/\overline{j(A,\tau)}=j_{B}'/\overline{j(B,\tau)}$ (after canceling $\overline{\xi}$), and as $\langle\overline{\xi}\zeta_{v_{\pm}},\overline{\xi}\omega_{v_{\pm}}\rangle=\langle\zeta_{v_{\pm}},\omega_{v_{\pm}}\rangle$ as in the proof of Proposition \ref{extWeil}, the exponent is the desired one. Moreover, the two external $j$-factors are the asserted ones times $\xi^{b_{-}-b_{+}}$, and the remaining multiplier is the $\rho_{M}(B)$-image of this root of unity times $\sum_{\gamma \in D_{M}}\theta_{M+\overline{\xi}\gamma}(\tau,\overline{\xi}\zeta,\overline{\xi}\omega;v)\mathfrak{e}_{\overline{\xi}\gamma}$. But in the summand associated with $\overline{\xi}\lambda$ for $\lambda \in M+\gamma$, the multiplication by $\overline{\xi}$ does not affect the Hermitian norms $|\lambda_{v_{\pm}}|^{2}$, and the pairings $\langle\overline{\xi}\zeta,\overline{\xi}\lambda\rangle$ and $\langle\overline{\xi}\lambda,\overline{\xi}\omega\rangle$ were seen to reproduce $\langle\zeta,\lambda\rangle$ and $\langle\lambda,\omega\rangle$ yet again. Our remaining multiplier thus equals $\rho_{M}(B)\sum_{\gamma \in D_{M}}\theta_{M+\gamma}(\tau,\zeta,\omega;v)\xi^{b_{-}-b_{+}}\mathfrak{e}_{\overline{\xi}\gamma}$, which via the formula from Proposition \ref{extWeil} equals $\rho_{M}(B)\rho_{M}(\xi)\Theta_{M}(\tau,\zeta,\omega;v)=\rho_{M}(A)\Theta_{M}(\tau,\zeta,\omega;v)$, as desired. This completes the proof of the theorem.
\end{proof}

\smallskip

We now turn to the differential properties of Hermitian theta functions. In the setting of Jacobi forms associated with orthogonal lattices, many differential operators were considered in \cite{[BRR]}, \cite{[CWR]}, \cite{[WR]}, and \cite{[RR]}. The operators that we will work with are similar to those from \cite{[Ze3]}, which were based on some constructions from \cite{[Ze2]}. We shall shorthand $\frac{\partial}{\partial z}$ with respect to any variable $z$ to $\partial_{z}$, and recall that if $W$ is a non-degenerate $r$-dimensional real quadratic space, and $\{\eta_{i}\}_{i=1}^{r}$ are the coordinates with respect to a basis for $W$, then by letting $\{\eta_{i}^{*}\}_{i=1}^{r}$ be the coordinates with respect to the dual basis for $W$ through the bilinear form, the Laplacian operator for $W$ (which is defined, up to a scalar multiple perhaps, by the action of the Casimir operator of the Lie algebra of $\operatorname{O}(W)$) is $\sum_{j=1}^{r}\partial_{\eta_{i}}\partial_{\eta_{i}^{*}}$, acting on smooth functions on $W$ (and this operator is independent of the basis chosen and is invariant under the action of $\operatorname{O}(W)$). The only extension from \cite{[Ze2]} that we shall need for this operator is by considering functions on $W\otimes_{\mathbb{R}}\mathbb{C}$, and replacing the derivative with respect to any real variable $\eta_{i}$ or $\eta_{i}^{*}$ by one with respect to an associated complex holomorphic one (this is an operator of Hodge weight $(2,0)$, and the other ones from \cite{[Ze2]}, with the other Hodge weights, will not show up here). One can show, in fact, that when $W$ comes with an endomorphism $J$ with $J^{2}=-d\operatorname{Id}_{W}$, and $W\otimes_{\mathbb{R}}\mathbb{C}$ decomposes as in Definition \ref{pairdef} and Lemma \ref{VQdecom}, then each summand in the Laplacian $\Delta_{W}$ will involve a derivative with respect to a variable from $W_{\mathbb{C}}$ times one associated with a variable from $W_{\overline{\mathbb{C}}}$.

The spaces for which we will consider these Laplacian operators are the spaces $v_{\pm}$ associated with an element of $\operatorname{Gr}^{J}(M_{\mathbb{R}})$ from Equation \eqref{GrassJ}, for which we will adopt the notation $\Delta_{v_{\pm}}^{h}$ from \cite{[Ze1]}, \cite{[Ze2]}, and others (see \cite{[Ze3]} for the relations between these operators, their real counterparts, the Laplacian of $V$ itself, and the Laplacian that was used in the theta functions defined in \cite{[Bor]}, or in \cite{[Ze5]}). Proposition 1.3 of \cite{[Ze3]} then combines with Lemma \ref{compThetaM} to produce the following result.
\begin{prop}
The theta function $\Theta_{M}$ from Equation \eqref{JacTheta} is sent to 0 by the operators $4\pi i\partial_{\tau}-\Delta_{v_{+}}^{h}$, $4\pi i\partial_{\overline{\tau}}-\Delta_{v_{-}}^{h}$, and any anti-holomorphic derivative $\partial_{\overline{\zeta}_{i}}$ or $\partial_{\overline{\omega}_{i}}$ that arises from some basis for $V_{\mathbb{C}}$ or $V_{\overline{\mathbb{C}}}$. \label{difeq}
\end{prop}
The annihilation of $\Theta_{M}$ by all the differential operators except for the first two is just the holomorphicity of this function in the variables $\zeta$ and $\omega$, as observed above.

\section{Hermitian Jacobi Forms \label{HJacForms}}

Hermitian Jacobi forms with lattice index have theta expansions, similar to the Jacobi forms in other settings. We will first make this connection, since it will give us some notions that are needed for defining them formally. In this section we will only consider scalar-valued Jacobi forms, in an analysis that is parallel to \cite{[Ze3]}, and the next section, after discussing discriminant forms in more detail, we will show how to modify these constructions for vector-valued modular forms.

The first step is the following analysis, involving Fourier expansions, that follows Lemma 2.1 of \cite{[Ze3]}.
\begin{lem}
Let $\Phi:\mathcal{H} \times M_{\mathbb{C}} \times M_{\overline{\mathbb{C}}}\to\mathbb{C}$ be a smooth function, that for fixed $\tau\in\mathcal{H}$ is holomorphic in the variables $\zeta \in M_{\mathbb{C}}$ and $\omega \in M_{\overline{\mathbb{C}}}$, and is invariant under taking $\zeta$ and $\omega$ to $\zeta+\nu_{\mathbb{C}}$ and $\omega+\nu_{\overline{\mathbb{C}}}$ with $\nu \in M$. Then there exist smooth functions $f_{\lambda}$ for $\lambda \in M^{*}$ such that \[\Phi(\tau,\zeta,\omega):=\textstyle{\sum_{\lambda \in M^{*}}f_{\lambda}(\tau)\mathbf{e}\big(\tau|\lambda_{v_{+}}|^{2}+\overline{\tau}|\lambda_{v_{-}}|^{2}+\langle\zeta,\lambda\rangle+\langle\lambda,\omega\rangle\big)}.\] Assuming that $\Phi$ satisfies the full scalar-valued relation from Proposition \ref{perTheta} (or more precisely, from Equation \eqref{perJac} in Definition \ref{Jacdef} below), for $\sigma$ and $\nu$ from $M$, the function $f_{\lambda}$ remains invariant under changing $\lambda$ by an element of $M$. \label{Fourinv}
\end{lem}

\begin{proof}
When we replace the variables $\zeta$ and $\omega$ by $\zeta+\omega \in M\otimes_{\mathbb{Z}}\mathbb{C}$, our assumption is that $\Phi$ is holomorphic and $M$-periodic in this variable. It follows that we can write $\Phi(\tau,\zeta,\omega)$ as a sum $\sum_{\lambda \in L^{*}}\tilde{f}_{\lambda}(\tau)\mathbf{e}\big((\lambda,\zeta+\omega)\big)$ for smooth functions $\tilde{f}_{\lambda}:\mathcal{H}\to\mathbb{C}$, and the proof of Lemma \ref{compThetaM} shows that the argument of the exponent also equals $\langle\zeta,\lambda\rangle+\langle\lambda,\omega\rangle$. By writing $\tilde{f}_{\lambda}$ as $\tau\mapsto\mathbf{e}\big(\tau|\lambda_{v_{+}}|^{2}+\overline{\tau}|\lambda_{v_{-}}|^{2}\big)$ times another function $f_{\lambda}$, we obtain the first assertion.

For the second one, we express both sides of the equality from Proposition \ref{perTheta} (with $\Phi$ rather than $\Theta_{M}$) using the expansion that we already have, and then on the left hand side $f_{\lambda}(\tau)$ is multiplied by the exponent of \[\tau\big(|\lambda_{v_{+}}|^{2}+\langle\sigma_{+},\lambda\rangle+\langle\lambda,\sigma_{+}\rangle\big)+\overline{\tau}\big(|\lambda_{v_{+}}|^{2}+\langle\sigma_{-},\lambda\rangle+\langle\lambda,\sigma_{-}\rangle\big) +\langle\zeta,\lambda\rangle+\langle\lambda,\omega\rangle\] (note that multiplication by $\tau$ and $\overline{\tau}$ is via the extended scalars to $\mathbb{C}$, not through the endomorphism $J$, which is why we have $\mathbb{C}$-linearity in these expressions also in the second variable of the pairing $\langle\cdot,\cdot\rangle$). On the right hand side, it is multiplied by the exponent of \[\tau|\lambda_{v_{+}}|^{2}+\overline{\tau}|\lambda_{v_{-}}|^{2}+\langle\zeta,\lambda\rangle+\langle\lambda,\omega\rangle-\tau|\sigma_{v_{+}}|^{2}-\overline{\tau}|\sigma_{v_{-}}^{2}|-\langle\zeta,\sigma\rangle- \langle\sigma,\omega\rangle,\] and as the pairings of $\lambda$ with $\sigma_{\pm}$ coincide with those of $\lambda_{\pm}$, it follows that replacing $\lambda$ by $\lambda+\sigma$ in the latter expression yields the former one. We therefore have an equality between two Fourier expansions, in which comparing the coefficients yields the equality $f_{\lambda+\sigma}(\tau)=f_{\lambda}(\tau)$ for every $\tau\in\mathcal{H}$ and $\lambda \in M^{*}$, as desired. This proves the lemma.
\end{proof}

We saw that $\Theta_{M}$ is a $\mathbb{C}[D_{M}]$-valued function, and we denote by $\overline{M}$ the lattice $M$ in which all the pairings are inverted. Then $\mathbb{C}[D_{\overline{M}}]$ is the natural dual of $\mathbb{C}[D_{M}]$, with its natural basis, which we denote by $\{\mathfrak{e}_{\gamma}^{*}\}_{\gamma \in D_{M}}$, being dual to the natural one $\{\mathfrak{e}_{\gamma}\}_{\gamma \in D_{M}}$ of $\mathbb{C}[D_{M}]$, as representation spaces of $\operatorname{SL}_{2}(\mathbb{Z})$. Moreover, since the signature of $\overline{M}$ is the opposite to that of $M$, and the action of $\mu(\mathcal{O})$ is the same on both, this duality extends to a duality of representations of $\operatorname{U}_{1,1}(\mathbb{Z})$ via Proposition \ref{extWeil}. We shall denote the (bilinear) pairing between $\mathbb{C}[D_{M}]$ and $\mathbb{C}[D_{\overline{M}}]$ by $[\cdot,\cdot]_{M}$, and then from Lemma \ref{Fourinv} we deduce the following consequence.
\begin{cor}
A smooth function $\Phi:\mathcal{H} \times M_{\mathbb{C}} \times M_{\overline{\mathbb{C}}}\to\mathbb{C}$ satisfying the transformation law from Proposition \ref{perTheta}, or Equation \eqref{perJac}, can be presented for $\tau\in\mathcal{H}$, $\zeta \in M_{\mathbb{C}}$ and $\omega \in M_{\overline{\mathbb{C}}}$, as $\Phi(\tau,\zeta,\omega)=[\Theta_{M}(\tau,\zeta,\omega;v),F(\tau)]_{M}$ for a unique smooth function $F:\mathcal{H}\to\mathbb{C}[D_{\overline{M}}]$. \label{decomtheta}
\end{cor}

\begin{proof}
Lemma \ref{Fourinv} implies that there are functions $f_{\gamma}$ for $\gamma \in D_{M}$ such that $\Phi(\tau,\zeta,\omega)$ expands as $\sum_{\lambda \in M^{*}}f_{\lambda+M}(\tau)\mathbf{e}\big(\tau|\lambda_{v_{+}}|^{2}+\overline{\tau}|\lambda_{v_{-}}|^{2}+\langle\zeta,\lambda\rangle+\langle\lambda,\omega\rangle\big)$. As $f_{\gamma}(\tau)$ multiplies the sum over $\lambda \in M+\gamma$ of these exponents, and the latter sum is $\theta_{M+\gamma}(\tau,\zeta,\omega;v)$ by Equation \eqref{JacTheta}, the desired formula follows from the definition of the pairing $[\cdot,\cdot]_{M}$. This proves the corollary.
\end{proof}

\smallskip

The main object of this paper can now be defined, as follows.
\begin{defn}
Take two integral weights $k$ and $l$, as well as a finite index subgroup $\Gamma\subseteq\operatorname{U}(1,1)(\mathbb{Z})$, and fix an $\mathcal{O}$-lattice $M$ as in Definition \ref{Olat} and an element $v$ in $\operatorname{Gr}^{J}(M_{\mathbb{R}})$ from Equation \eqref{GrassJ}. Then a \emph{Hermitian Jacobi form of weight $(k,l)$ and index $(M,v)$ with respect to $\Gamma$} is defined to be a smooth function $\Phi:\mathcal{H} \times M_{\mathbb{C}} \times M_{\overline{\mathbb{C}}}\to\mathbb{C}$, with the following properties. First, $\Phi$ is holomorphic in the variables from the two vector spaces. Second, given elements $\sigma$ and $\nu$ from $M$ we have, for any $\tau\in\mathcal{H}$, $\zeta \in M_{\mathbb{C}}$, and $\omega \in M_{\overline{\mathbb{C}}}$, the equality
\[\Phi(\tau,\zeta+\tau\sigma_{v_{+},\mathbb{C}}+\overline{\tau}\sigma_{v_{-},\mathbb{C}}+\nu_{\mathbb{C}},\omega+\tau\sigma_{v_{+},\overline{\mathbb{C}}}+\overline{\tau}\sigma_{v_{-},\overline{\mathbb{C}}}+ \nu_{\overline{\mathbb{C}}})=\]
\begin{equation}
=\mathbf{e}\big(-\tau|\sigma_{v_{+}}|^{2}-\overline{\tau}|\sigma_{v_{-}}^{2}|-\langle\zeta,\sigma\rangle-\langle\sigma,\omega\rangle\big)\Phi(\tau,\zeta,\omega). \label{perJac}
\end{equation}
Third, given such $\tau$, $\zeta$, and $\omega$, as well as a matrix $A\in\Gamma$, we have
\[\Phi\Big(A\tau,\tfrac{\zeta_{v_{+}}}{j(A,\tau)}+\tfrac{\zeta_{v_{-}}}{\overline{j^{c}(A,\tau)}},\tfrac{\omega_{v_{+}}}{j^{c}(A,\tau)}+\tfrac{\omega_{v_{-}}}{\overline{j(A,\tau)}}\Big)=\] \begin{equation}
=j(A,\tau)^{k}\overline{j(A,\tau)}^{l}\mathbf{e}\Big(\tfrac{j_{A}'\langle\zeta_{v_{+}},\omega_{v_{+}}\rangle}{j(A,\tau)}+\tfrac{\overline{j_{A}'}\langle\zeta_{v_{-}},\omega_{v_{-}}\rangle}{\overline{j(A,\tau)}}\Big) \Phi(\tau,\zeta,\omega), \label{modJac}
\end{equation}
with $j(A,\tau)$ as usual and $j^{c}(A,\tau)$ as defined in Theorem \ref{modTheta}. And fourth, we recall that Equation \eqref{perJac} produces, via Lemma \ref{Fourinv}, the smooth functions $f_{\lambda}:\mathcal{H}\to\mathbb{C}$, and we demand that their growth towards each cusp of $\Gamma$ is linear exponential. Given a Hermitian Jacobi form $\Phi$, we call it \emph{pseudo-holomorphic} if the operator $4\pi i\partial_{\overline{\tau}}-\Delta_{v_{-}}^{h}$ from Proposition \ref{difeq} takes $\Phi$ to 0. \label{Jacdef}
\end{defn}
We remark that by working with the other operator $4\pi i\partial_{\tau}-\Delta_{v_{+}}^{h}$ from Proposition \ref{difeq}, we can consider skew-holomorphic Hermitian Jacobi forms, which can be described analogously to \cite{[Ze3]}, thus extending the notion from \cite{[Sk]} and \cite{[Hay]} (among others). In addition, Equations \eqref{perJac} and \eqref{modJac} can be seen as invariance under the slash operators arising from a Hermitian version of a Heisenberg group as in Remark 2.4 of \cite{[Ze3]}, or from a semi-direct product involving this Heisenberg group and the group $\Gamma$ (or, in the real case, all of $\operatorname{U}(1,1)(\mathbb{R})$ or of $\operatorname{SU}(1,1)(\mathbb{R})$). One way of considering groups that contain non-integral entries involves characteristics as in Section 4 of \cite{[Ze3]} or in Theorem \ref{propgen} below, but we shall mainly be using the operators from Definition \ref{opersDM} and their consequences.

\smallskip

We recall that for $k$ and $l$ in $\mathbb{Z}$, a \emph{modular form of weight $(k,l)$ and representation $\rho_{\overline{M}}$ with respect to $\Gamma$} is a function $F:\mathcal{H}\to\mathbb{C}[D_{\overline{M}}]$ that satisfies, for every $\tau\in\mathcal{H}$ and $A\in\Gamma$ the functional equation
\begin{equation}
F(A\tau)=j(A,\tau)^{k}\overline{j(A,\tau)}^{l}\rho_{\overline{M}}(A)F(\tau), \label{modeq}
\end{equation}
and that is smooth and grows at most linearly exponentially towards any cusp of $\Gamma$. A weight $(k,0)$ will be written simply as weight $k$, as usual, with $F$ being \emph{weakly holomorphic} in case it is holomorphic on $\mathcal{H}$ with this growth condition, \emph{holomorphic} if in addition it is bounded at the cusps of $\Gamma$, and \emph{cuspidal} when in even decreases towards them. Using this notion we can now state the main relation involving Hermitian Jacobi forms.
\begin{thm}
Let $\Phi$ be a Hermitian Jacobi form of weight $(k,l)$ and index $(M,v)$ with respect to $\Gamma$, as in Definition \ref{Jacdef}, and let $F$ be the function associated with $\Phi$ via Corollary \ref{decomtheta}. Then $F$ is a modular form of weight $(k-b_{+},l-b_{-})$ and representation $\rho_{\overline{M}}$ with respect to $\Gamma$. If $F$ is such a modular form, then the function defined by $\Phi_{M,v}^{F}(\tau,\zeta,\omega):=[\Theta_{M}(\tau,\zeta,\omega;v),F(\tau)]_{M}$ is a Hermitian Jacobi form with these parameters. \label{main}
\end{thm}

\begin{proof}
It is clear from the construction that the pairing with $\Theta_{M}$ in one direction, and the expansion from Lemma \ref{Fourinv} and Corollary \ref{decomtheta} in the other one, are inverse mappings between spaces of functions. We hence need to show that pairing a modular form with $\Theta_{M}$ yields a Hermitian Jacobi form, and that the function arising from the expansion of a Hermitian Jacobi form is a vector-valued modular form.

So let $\Phi$ be as in Definition \ref{Jacdef}, and let $F:\mathcal{H}\to\mathbb{C}[D_{\overline{M}}]$ be the associated function from Corollary \ref{decomtheta}. The assumption on the growth is part of Definition \ref{Jacdef}, so we have to check Equation \eqref{modeq}. Take $\tau\in\mathcal{H}$, $\zeta \in M_{\mathbb{C}}$, $\omega \in M_{\overline{\mathbb{C}}}$, and $A\in\Gamma$, and apply Corollary \ref{decomtheta} to the two sides in Equation \eqref{modJac}. On the right hand side we get
\[j(A,\tau)^{k}\overline{j(A,\tau)}^{l}\mathbf{e}\Big(\tfrac{j_{A}'\langle\zeta_{v_{+}},\omega_{v_{+}}\rangle}{j(A,\tau)}+\tfrac{\overline{j_{A}'}\langle\zeta_{v_{-}},\omega_{v_{-}}\rangle}{\overline{j(A,\tau)}}\Big) [\Theta_{M}(\tau,\zeta,\omega;v),F(\tau)]_{M},\] and for the left hand side we apply Theorem \ref{modTheta} to the theta function in the pairing and obtain, by bilinearity, the expression
\[j(A,\tau)^{b_{+}}\overline{j(A,\tau)}^{b_{-}}\mathbf{e}\Big(\tfrac{j_{A}'\langle\zeta_{v_{+}},\omega_{v_{+}}\rangle}{j(A,\tau)}+\tfrac{\overline{j_{A}'}\langle\zeta_{v_{-}},\omega_{v_{-}}\rangle}{\overline{j(A,\tau)}}\Big) [\rho_{M}(A)\Theta_{M}(\tau,\zeta,\omega;v),F(\tau)]_{M}.\] As the pairing is $\operatorname{U}(1,1)(\mathbb{Z})$-equivariant, we can replace $\rho_{M}(A)$ on one side by $\rho_{\overline{M}}(A)^{-1}$, so by comparing the two expressions we obtain that the $[\cdot,\cdot]_{M}$-pairing of $\Theta_{M}(\tau,\zeta,\omega;v)$ with $j(A,\tau)^{k-b_{+}}\overline{j(A,\tau)}^{l-b_{-}}F(\tau)$ and with $\rho_{\overline{M}}(A)^{-1}F(A\tau)$, both functions of $\tau$ alone, coincide as functions $\mathcal{H} \times M_{\mathbb{C}} \times M_{\overline{\mathbb{C}}}$. But since the proof of Corollary \ref{decomtheta} shows that the pairing with $\Theta_{M}$ determines the function on $\mathcal{H}$ (this is just the uniqueness of Fourier expansions), we obtain that $F$ satisfies the modularity condition from Equation \eqref{modeq}, with the desired weights.

Let now $F$ be a modular form of weight $(k-b_{+},l-b_{-})$ and representation $\rho_{\overline{M}}$ with respect to $\Gamma$, and set $\Phi_{M,v}^{F}$ by the formula above. The fact that $F$ does not depend on $\zeta$ and $\omega$ implies, via Proposition \ref{perTheta} and the bilinearity of the pairing $[\cdot,\cdot]_{M}$, that Equation \eqref{perJac} holds for $\Phi=\Phi_{L,v}^{F}$, and that $\Phi_{L,v}^{F}$ is holomorphic in $\zeta$ and $\omega$. Take now $A\in\Gamma$, and apply Theorem \ref{modTheta} and Equation \eqref{modeq} to transform the left hand side of Equation \eqref{modJac}, with $\Phi=\Phi_{M,v}^{F}$, into \[\mathbf{e}\Big(\!\tfrac{j_{A}'\langle\zeta_{v_{+}},\omega_{v_{+}}\rangle}{j(A,\tau)}+\tfrac{\overline{j_{A}'}\langle\zeta_{v_{-}},\omega_{v_{-}}\rangle}{\overline{j(A,\tau)}}\!\Big)j(A,\tau)^{k}\overline{j(A,\tau)}^{l} [\rho_{M}(A)\Theta_{M}(\tau,\zeta,\omega;v),\!\rho_{\overline{M}}(A)F(\tau)]_{M}\!.\] But the equivariance of the pairing allows us to omit the Weil representations in the latter formula, yielding the desired right hand side by the definition of $\Phi_{M,v}^{F}$. The fact that the function whose pairing with $\Theta_{M}$ yields $\Phi_{M,v}^{F}$ is $F$ by definition, for which the growth condition holds by assumption, complete the verification that $\Phi_{M,v}^{F}$ has all the properties required in Definition \ref{Jacdef}. This proves the theorem.
\end{proof}

\begin{rmk}
Consider, in Equation \eqref{modJac}, the case where $A\in\Gamma$ is a power $T^{N}$ of $T$ (for an appropriate $N$, this power will indeed by in $\Gamma$), where $j$ and $j^{c}$ equal 1 and $j'$ and its complex conjugate vanish. This implies, by Definition \ref{Jacdef}, that if $\Phi$ is a Hermitian Jacobi form, then it is $N$-periodic in $\tau$ (for fixed $\zeta$ and $\omega$), which turns the Fourier expansion from Lemma \ref{Fourinv} into the form \[\textstyle{\sum_{m\in\frac{1}{N}\mathbb{Z}}\sum_{\lambda \in M^{*}}c_{m,\lambda}(y)\mathbf{e}\big(mx+\langle\zeta,\lambda\rangle+\langle\lambda,\omega\rangle\big)e^{-2\pi(|\lambda_{v_{+}}|^{2}-|\lambda_{v_{-}}|^{2})y}}.\] Equation \eqref{perJac} would then imply the equality $c_{m,\lambda}=c_{m+\langle\sigma,\lambda\rangle+\langle\lambda,\sigma\rangle+|\sigma|^{2},\lambda+\sigma}$ of functions of $y$ for any $\sigma \in M$ and $\lambda \in M^{*}$, resembling more the relations showing up in \cite{[Hav1]}, \cite{[Hav2]}, and others (the change in $m$ can also be written as $(\lambda,\sigma)+\frac{\sigma^{2}}{2}$ using the associated orthogonal pairing, as in the orthogonal setting). By setting $a_{n,\lambda}$ to be $c_{n+|\lambda|^{2},\lambda}$ (as functions of $y$ again) for any $n\in\mathbb{Q}$ for which the first index lies in $\frac{1}{N}\mathbb{Z}$, we deduce that $a_{n,\lambda}$ depends on $\lambda$ only modulo $M$, allowing us to write this function as $a_{n,\gamma}$ where $\gamma$ is now in $D_{M}$. Since $c_{m,\lambda}$ is now given by $a_{m-|\lambda|^{2},\lambda+M}$, we can replace $m$ by $n=m-|\lambda|^{2}$ in the previous expansion, which will take it to the form \[\sum_{n\in\mathbb{Q}}\sum_{\gamma \in D_{M}}a_{n,\lambda+L}(y)\mathbf{e}(nx)\sum_{\lambda \in M+\gamma}\mathbf{e}\big(|\lambda|^{2}x+\langle\zeta,\lambda\rangle+\langle\lambda,\omega\rangle\big)e^{-2\pi(|\lambda_{v_{+}}|^{2}-|\lambda_{v_{-}}|^{2})y}.\] But the internal sum over $\lambda$ is just $\theta_{M+\gamma}(\tau,\zeta,\omega;v)$ via Equation \eqref{JacTheta}, and multiplying each element of the first sum by $\mathfrak{e}_{\gamma}^{*}$ produces the Fourier expansion of the function $F$ defined in Corollary \ref{decomtheta} (namely the modular form $F$ for which we have $\Phi=\Phi_{M,v}^{F}$ in Theorem \ref{main}). But working with $F(\tau)$ itself directly makes the argument easier, as we do not have to consider convergence issues for such series. \label{JacFour}
\end{rmk}

\smallskip

In general, Theorem \ref{main} involve non-holomorphic expressions in $\tau$, and indeed, Jacobi forms of indefinite lattice index will never be holomorphic (though see Theorem \ref{holJac} below for the positive definite case). The holomorphicity question for vector-valued modular forms shows up on the Hermitian Jacobi forms side via the following analogue of Proposition 2.7 of \cite{[Ze3]}.
\begin{prop}
Let $F$ be a modular form of weight $(k-b_{+},l-b_{-})$ and representation $\rho_{\overline{M}}$ with respect to $\Gamma$, with the associated Hermitian Jacobi form $\Phi_{M,v}^{F}$ from Theorem \ref{main}. The modular form $F$ is then weakly holomorphic if and only if $\Phi_{M,v}^{F}$ is pseudo-holomorphic, a situation which can occur only when the weight of $\Phi_{M,v}^{F}$ is $(k,b_{-})$. \label{pshol}
\end{prop}

\begin{proof}
Apply the operator $4\pi i\partial_{\overline{\tau}}-\Delta_{v_{-}}^{h}$ from Definition \ref{Jacdef} to $\Phi_{M,v}^{F}(\tau,\zeta,\omega)$, considered as the pairing of $\Theta_{M}(\tau,\zeta,\omega;v)$ and $F(\tau)$, with the latter independent of $\zeta$ and $\omega$, and using Leibnitz' rule for $\partial_{\overline{\tau}}$, produces the expression \[\big[(4\pi i\partial_{\tau}-\Delta_{v_{-}}^{h})\Theta_{M}(\tau,\zeta,\omega;v),F(\tau)\big]_{M}+4\pi i\big[\Theta_{M}(\tau,\zeta,\omega;v),\partial_{\overline{\tau}}F(\tau)\big]_{M}.\] Proposition \ref{difeq} yields the vanishing of the first term, and as a $\mathbb{C}[D_{\overline{M}}]$-valued function of $\tau$ is determined by its pairing with $\Theta_{M}$, we deduce from Definition \ref{Jacdef} that $\Phi_{M,v}^{F}$ is pseudo-holomorphic precisely when $\partial_{\overline{\tau}}F=0$. As the modular form $F$ is smooth by assumption and grows at most linearly exponentially at the cusps, the latter condition holds if and only if $F$ is weakly holomorphic. This proves the first assertion, from which the second one immediately follows as weakly holomorphic modular forms have weights $(k,0)$, with $l=0$. This proves the proposition.
\end{proof}

\begin{rmk}
We replace $a_{n,\gamma}(y)$ from Remark \ref{JacFour} by $\tilde{a}_{n,\gamma}(y)e^{-2\pi ny}$, to present the Fourier expansion of $F$ there using the holomorphic exponentials $\mathbf{e}(n\tau)$. Then $c_{m,\lambda}(y)=a_{m-|\lambda|^{2},\lambda+M}(y)$ becomes $\tilde{c}_{m,\lambda}(y)e^{-2\pi my+2\pi|\lambda|^{2}y}$, with the former multiplier being $\tilde{a}_{m-|\lambda|^{2},\lambda+M}(y)$. The full expansion $\Phi_{L,v}^{F}(\tau,\zeta)$ then takes the form $\sum_{m\in\frac{1}{N}\mathbb{Z}}\sum_{\lambda \in M^{*}}\tilde{c}_{m,\lambda}(y)\mathbf{e}\big(m\tau+\langle\zeta,\lambda\rangle+\langle\lambda,\omega\rangle\big)e^{4\pi|\lambda_{v_{-}}|^{2}y}$, using the fact that $|\lambda|^{2}=|\lambda_{v_{+}}|^{2}+|\lambda_{v_{-}}|^{2}$. Noting that $F$ is weakly holomorphic precisely when $\tilde{c}_{m,\lambda}$ is a constant for every $m$ and $\lambda$, which happens if and only if $4\pi i\partial_{\overline{\tau}}$ acts only on the exponentials in $y$. As $\Delta_{v_{-}}^{h}$ operates only on the terms with $\zeta$ and $\omega$, the idea behind Proposition \ref{pshol} is clearer using these conventions. \label{holFour}
\end{rmk}
Remark \ref{holFour} is particularly valuable in the case where $M$ is positive definite, as seen in Equation \eqref{Fourposdef} below.

\smallskip

Section 3 of \cite{[Ze3]} contains a lot of relations that are satisfied between constructions on lattices, modular forms, and Jacobi forms. They all extend, \emph{mutatis mutandis}, to the case of $\mathcal{O}$-lattices and Hermitian Jacobi forms. Here are the important details to note when changing the setting to the Hermitian one:
\begin{itemize}
\item The isomorphism from Lemma 3.1 there extends to the representation of $\operatorname{U}(1,1)(\mathbb{Z})$ constructed in Proposition \ref{extWeil}, by adding conjugation to all the entries and hence inverting the scalar matrices from $\mu(\mathcal{O})$. This can then be used to construct skew-holomorphic Hermitian Jacobi forms, via the other operator from Proposition \ref{difeq}.
\item The over-lattices of an $\mathcal{O}$-lattice, for which one considers the arrow operators (these are introduced here in Equation \eqref{arrowdef} below), can be an even projective $\mathcal{O}$-lattice containing $M$, or obtained by extending scalars to a larger order in $\mathbb{K}$, or combining these operations. Lemma \ref{duallat} and Corollary \ref{Mindual} transfer it to orthogonal over-lattices over $\mathbb{Z}$.
\item Orthogonal direct sums of $\mathcal{O}$-lattices yields an $\mathcal{O}$-lattice if and only if both lattices are over the same order $\mathcal{O}$ in $\mathbb{K}$ (this was already exemplified in Remarks \ref{projlat} and \ref{notlat}). The module structure of Hermitian Jacobi forms over the ring of scalar-valued modular forms holds equally well.
\item A primitive sub-lattice of an $\mathcal{O}$-lattice is an $\mathcal{O}$-lattice, since it is a direct summand of a projective $\mathcal{O}$-module.
\item The theta contraction from \cite{[Ma]}, or the restriction from \cite{[Ze5]}, extend to the unitary setting via embeddings like that from \cite{[Ho2]}, namely via the embedding of the space from Equation \eqref{GrassJ} into that from Equation \eqref{Grassdef}.
\item The isomorphism $\iota$ from Proposition 3.9 of \cite{[Ze3]} should respect the action of $\mathcal{O}$, or equivalently of the endomorphism $J$. Then $v$ and $w$ there are from the joint $\operatorname{Gr}^{J}$-space.
\end{itemize}
In addition, as in Remark 2.9 of \cite{[Ze3]}, one can fix the modular form $F$ and investigate the properties of $(\tau,\zeta,\omega,v)\mapsto\Phi_{M,v}^{F}(\tau,\zeta,\omega)$ as a function of variables from the product $\mathcal{H} \times M_{\mathbb{C}} \times M_{\overline{\mathbb{C}}}\times\operatorname{Gr}^{J}(V_{\mathbb{R}})$ (in particular, it is invariant under the diagonal action on the last three variables of the stable unitary group of $M$, acting trivially on $D_{M}$).

\smallskip

The theta functions from Section 4 of \cite{[Bor]} may also involve characteristics, and Section 4 of \cite{[Ze3]} presents the basic properties of Jacobi theta functions with characteristics. Since such theta functions are, in fact, used implicitly in \cite{[Xi]} and more explicitly in \cite{[BRZ]} (at least in the holomorphic, positive definite setting), we define them here and state their functional equations. With $M$, $\tau$, $\zeta$, $\omega$, and $v$ as above, we take two elements $\alpha$ and $\beta$ in $M_{\mathbb{R}}$, and define the \emph{Hermitian Jacobi--Siegel theta function with characteristics $\binom{\alpha}{\beta}$} to be the expression $\Theta_{M}\big(\tau,\zeta,\omega;\binom{\alpha}{\beta};v\big)=\sum_{\gamma \in D_{M}}\theta_{M+\gamma}\big(\tau,\zeta,\omega;\binom{\alpha}{\beta};v\big)\mathfrak{e}_{\gamma}$ in which the scalar-valued component is the sum over $\lambda \in M+\gamma$ of
\begin{equation}
\mathbf{e}\big(\tau|(\lambda+\beta)_{v_{+}}|^{2}+\overline{\tau}|(\lambda+\beta)_{v_{-}}|^{2}+\langle\zeta,\lambda+\beta\rangle+\langle\lambda+\beta,\omega\rangle\big)-\big\langle\lambda+ \tfrac{\beta}{2},\alpha\big\rangle-\big\langle\alpha,\lambda+\tfrac{\beta}{2}\big\rangle\big]. \label{charTheta}
\end{equation}
Combining the proofs of Theorem 4.1 from \cite{[Ze3]} with Lemma \ref{compThetaM} (whose proof relates the theta function from Equation \eqref{charTheta} to the one with characteristics defined in Equation (16) of \cite{[Ze3]}, with the variable $\zeta+\omega \in M\otimes_{\mathbb{Z}}\mathbb{C}$) as in the proofs of Proposition \ref{perTheta} and Theorem \ref{modTheta}, as well as Proposition \ref{difeq} for the last assertion, yields the following result.
\begin{thm}
The expression \[\Theta_{M}\big(\tau,\zeta+\tau\sigma_{v_{+},\mathbb{C}}+\overline{\tau}\sigma_{v_{-},\mathbb{C}}+\nu_{\mathbb{C}},\omega+\tau\sigma_{v_{+},\overline{\mathbb{C}}}+\overline{\tau}\sigma_{v_{-},\overline{\mathbb{C}}}+ \nu_{\overline{\mathbb{C}}};\textstyle{\binom{\alpha}{\beta}};v\big)\] is equal, when $\sigma$ and $\nu$ are elements of $M$, to \[\mathbf{e}\big(-\tau|\sigma_{v_{+}}|^{2}-\overline{\tau}|\sigma_{v_{-}}^{2}|-\langle\zeta,\sigma\rangle-\langle\sigma,\omega\rangle+\langle\nu,\beta\rangle+\langle\beta,\nu\rangle+\langle\sigma,\alpha\rangle+ \langle\alpha,\sigma\rangle\big)\] times $\Theta_{M}\big(\tau,\zeta,\omega;\binom{\alpha}{\beta};v\big)$. When we view an element $A\in\operatorname{U}(1,1)(\mathbb{Z})$ as a $2\times2$ matrix with entries in $\mathcal{O}$, which can thus act on elements of $M$ hence of $M_{\mathbb{R}}$, we have the equality \[\Theta_{M}\Big(A\tau,\tfrac{\zeta_{v_{+}}}{j(A,\tau)}+\tfrac{\zeta_{v_{-}}}{\overline{j^{c}(A,\tau)}},\tfrac{\omega_{v_{+}}}{j^{c}(A,\tau)}+\tfrac{\omega_{v_{-}}}{\overline{j(A,\tau)}};\textstyle{\binom{\alpha}{\beta}};v \Big)=\] \[=j(A,\tau)^{b_{+}}\overline{j(A,\tau)}^{b_{-}}\mathbf{e}\Big(\tfrac{j_{A}'\langle\zeta_{v_{+}},\omega_{v_{+}}\rangle}{j(A,\tau)}+\tfrac{\overline{j_{A}'}\langle\zeta_{v_{-}},\omega_{v_{-}}\rangle}{\overline{j(A,\tau)}} \Big)\rho_{M}(A)\Theta_{M}(\tau,\zeta,\omega;\textstyle{\binom{\alpha}{\beta}};v).\] The theta functions from Equation \eqref{charTheta} are also annihilated by all the differential operators mentioned in Proposition \ref{difeq}. \label{propgen}
\end{thm}
Note that while the proof of Theorem 4.1 of \cite{[Ze3]} suffices (via Lemma \ref{compThetaM}) for establishing the second equality in Theorem \ref{propgen} when the matrix $A$ is in $\operatorname{SU}(1,1)(\mathbb{Z})=\operatorname{SL}_{2}(\mathbb{Z})$, the considerations from the proof of Theorem \ref{modTheta} apply here too (with $\beta$ also multiplied by $\overline{\xi}$, as is $\alpha$ for preserving the pairings involved) for extending it to the full representation of $\operatorname{SU}(1,1)(\mathbb{Z})$, as defined in Proposition \ref{extWeil}. As in Proposition 4.2 of that reference, the functions on $\mathcal{H} \times M_{\mathbb{C}} \times M_{\overline{\mathbb{C}}}$ that satisfy the first equality from Theorem \ref{propgen} and are holomorphic in $\zeta$ and $\omega$ are $[\cdot,\cdot]_{M}$-pairings of the theta function from Equation \eqref{charTheta} with smooth $\mathbb{C}[D_{\overline{M}}]$-valued functions (that are holomorphic if and only if the pairing is pseudo-holomorphic), but the pairings are no longer Jacobi forms when $\binom{\alpha}{\beta}\neq\binom{0}{0}$. The similarity with the theta functions with characteristics as in \cite{[FZ]} and others, and the relation with the arrow operator as in Proposition 4.4 there, are valid here equally well.

\section{Vector--Valued Jacobi Forms \label{VVJF}}

The vector-valued modular and Jacobi forms (either Hermitian and orthogonal) will be based on Weil representations associated with discriminant forms, that have to be appropriately related to $D_{M}$ for our lattice $M$.

\smallskip

We thus formally define the objects with which we shall work.
\begin{defn}
A \emph{discriminant form} is finite Abelian group $D$ with a $\mathbb{Q}/\mathbb{Z}$-valued quadratic form, denoted by $\gamma\mapsto\frac{\gamma^{2}}{2}$, such that the associated bilinear pairing $(\cdot,\cdot):D \times D\to\mathbb{Q}/\mathbb{Z}$ is non-degenerate in the sense that the homomorphism that it defines from $D$ to $\operatorname{Hom}(D,\mathbb{Q}/\mathbb{Z})$ is injective hence bijective. For a subgroup $H \subseteq D$, we denote by $H^{\perp}$ the set of elements of $D$ pairing trivially with $H$. A subgroup $H \subseteq D$ is called \emph{isotropic} if $\frac{\gamma^{2}}{2}=0$ for all $\gamma \in H$. \label{discform}
\end{defn}

If $H$ is an isotropic subgroup of $D$ as in Definition \ref{discform}, then $H$ pairs trivially with itself, namely we have $H \subseteq H^{\perp}$, and the quotient $\Delta:=H^{\perp}/H$ is non-degenerate with the induced bilinear pairing. However, not every $H$ that is contained in $H^{\perp}$ is isotropic, and only the quotients arising from isotropic subgroups $H$ carry well-defined $\mathbb{Q}/\mathbb{Z}$-valued quadratic forms, and are thus discriminant forms themselves.
\begin{rmk}
Note that discriminant forms like $D=D_{M}$ for an $\mathcal{O}$-lattice $M$ carry additional structure, namely an endomorphism induced by $J$ that squares to $-d\operatorname{Id}_{D}$ and satisfies condition $(iii)$ of Lemma \ref{pairchar}, whence also the conditions $(ii)$, $(iv)$, and $(v)$ there for the pairings. This structure carries over to $\Delta$ in case $J$ preserves $H$, and commutes with operations like the arrow operators defined below. However, this will play no role in what follows. \label{JonD}
\end{rmk}

For such $D$, $H$, and $\Delta$ there are two arrow operators, which we denote by $\uparrow^{D}_{\Delta}=\uparrow_{H}:\mathbb{C}[\Delta]\to\mathbb{C}[D]$ and $\downarrow^{D}_{\Delta}=\downarrow_{H}:\mathbb{C}[D]\to\mathbb{C}[\Delta]$, and their actions on the natural bases is given by
\begin{equation}
\uparrow_{H}\mathfrak{e}_{\gamma}:=\sum_{\delta \in H^{\perp},\ \delta+H=\gamma}\mathfrak{e}_{\delta}\qquad\mathrm{and}\qquad\downarrow_{H}\mathfrak{e}_{\delta}:=\begin{cases} \mathfrak{e}_{\delta+H}, & \mathrm{when\ }\delta \in H^{\perp}, \\  0, & \mathrm{if\ }\delta \not\in H^{\perp}. \end{cases} \label{arrowdef}
\end{equation}
As is shown in \cite{[Ma]} and considered further in \cite{[Ze5]} (but has already been known before), both maps are homomorphisms of representations of $\operatorname{SL}_{2}(\mathbb{Z})$ (or of its metaplectic cover in case the signature is odd) when the spaces are considered as the representation spaces of the corresponding Weil representations (a fact that was used in \cite{[Ze4]} and others for investigating reducibility and related questions), and they are dual to one another when considering pairings like $[\cdot,\cdot]_{M}$ from above.

\smallskip

One of the most natural settings in which these operators show up is when $M$ is a primitive sublattice of another lattice $N$, of higher rank. Then if $M^{\perp}$ stands for the orthogonal complement of $M$ in $N$, then $N$ contains $M \oplus M^{\perp}$, the discriminant form of this direct sum is $D_{M} \oplus D_{M^{\perp}}$, and the isotropic subgroup $N/(M \oplus M^{\perp})$ of this direct sum intersects each of the defining summands trivially. Since the lattice $M$ (and not only $D_{M}$) is part of the data, but the lattices $M^{\perp}$ and $N$ are not, we make the following definition.
\begin{defn}
Let $D$ be a any discriminant form. Then a subgroup $H$ of $D_{M} \oplus D$ is called \emph{horizontal} if $H \cap D_{M}=H \cap D=\{0\}$ inside $D_{M} \oplus D$. \label{horizdef}
\end{defn}
We shall use, for elements $\alpha \in D_{M} \oplus D$, the notations $\alpha_{M}$ and $\alpha_{D}$ for the images of $\alpha$ under the natural projections onto $D_{M}$ and $D$ respectively. For horizontal isotropic subgroups we now prove the following simple lemma, which is mainly used for introducing our notation.
\begin{lem}
The following assertions hold:
\begin{enumerate}[$(i)$]
\item Let $H_{M} \subseteq D_{M}$ be any subgroup, and assume that we have an injective homomorphism $\iota:H_{M} \to D$. Then $H:=\{\gamma_{M}+\iota(\gamma_{M})|\gamma \in H_{M}\}$ is a horizontal subgroup of $D_{M} \oplus D$, and every horizontal subgroup is obtained in this way.
\item An element $\alpha \in D_{M} \oplus D$ lies in $H^{\perp}$ for $H$ from part $(i)$ if and only if the equality $(\alpha_{M},\gamma_{M})=-\big(\alpha_{D},\iota(\gamma_{M})\big)$ holds for every $\gamma \in H$.
\item The group $H$ is isotropic if and only if the equality $\frac{\iota(\gamma_{M})^{2}}{2}=-\frac{\gamma_{M}^{2}}{2}$ holds for any $\gamma_{M} \in H_{M}$.
\end{enumerate} \label{horizdecom}
\end{lem}

\begin{proof}
For every non-zero element $\gamma \in H$, we know that $\gamma_{M}\neq0$ as well as $\gamma_{D}=\iota(\gamma_{M})\neq0$ (as $\iota$ is injective), yielding the horizontality from Definition \ref{horizdef}. Conversely, let $H \subseteq D_{M} \oplus D$ be horizontal, and let $H_{M}$ denote the image of $H$ under the projection to $D_{M}$. The fact that $H$ intersects $D$ trivially (by Definition \ref{horizdef}) implies the projection map from $H$ to $H_{M}$ is injective hence bijective, and we can thus invert it. The map $\iota$ is then the composition of this inverted map with the projection onto $D$, and as this projection is also injective when restricted to $H$ (Definition \ref{horizdef} implies that since $H$ intersects $D_{M}$ trivially as well), we deduce that $\iota$ is also injective. This proves part $(i)$.

Part $(ii)$ then follows from the evaluation of the pairing of $\alpha=\alpha_{M}+\alpha_{D}$ and $\gamma=\gamma_{M}+\iota(\gamma_{M}) \in H$ and the orthogonality of $D_{M}$ and $D$ inside their direct sum, and this orthogonality is also applied in the direct and straightforward calculation that is required for part $(iii)$. This proves the lemma.
\end{proof}
Using Lemma \ref{horizdecom} and our notations for the projections, whenever a horizontal subgroup $H$ of $D_{M} \oplus D$ is given, we will use $H_{M}$ and $H_{D}$ for its images in $D_{M}$ and $D$ respectively, and the elements $\gamma \in H$, $\gamma_{M}$, and $\gamma_{D}=\iota(\gamma_{M})$ will be related as in the proof of that lemma.

\smallskip

For considering the quotient $\Delta_{H}:=H^{\perp}/H$ that arises from a horizontal isotropic subgroup $H$ of $D_{M} \oplus D$, we shall need the following description.
\begin{lem}
Let $H \subseteq D_{M} \oplus D$ be a horizontal subgroup, and fix a set $R \subseteq D_{M}$ of representatives for the quotient $D_{M}/H_{M}$.
\begin{enumerate}[$(i)$]
\item Any element $\alpha \in D_{M} \oplus D$ can be uniquely written as a sum $\beta+\gamma$, with $\gamma \in H$ and $\beta_{M} \in R$.
\item When $H$ is isotropic, the condition from part $(ii)$ of Lemma \ref{horizdecom} holds for $\alpha$ if and only if it holds for $\beta$.
\item In the isotropic case, the set $H^{\perp}_{R}$ of $\beta \in D_{M} \oplus D$ for which $\beta_{M} \in R$ and the condition from part $(ii)$ of Lemma \ref{horizdecom} is satisfied map bijectively onto $\Delta_{H}$ in the natural projection. We denote the image of such $\beta$ by $\overline{\beta}$.
\item The element $\mathfrak{e}_{\alpha}$ associated with $\alpha \in D_{M} \oplus D$, decomposed as in part $(i)$, is sent by $\downarrow_{H}$ to $\mathfrak{e}_{\overline{\beta}}$ in case the condition from part $(ii)$ of Lemma \ref{horizdecom} holds for $\alpha$ (or equivalently $\beta$), and to 0 otherwise.
\end{enumerate} \label{arrowhoriz}
\end{lem}

\begin{proof}
Part $(i)$ is a direct consequence from the fact that adding an element $\gamma \in H$ changes the $D_{M}$-coordinate by $\gamma_{M} \in H_{M}$, and the assumption on $R$. As $H$ is isotropic, elements $\gamma \in H$ lie in $H^{\perp}$ and hence satisfy the condition in question, and part $(ii)$ follows. Part $(iii)$ is deduced via parts $(i)$ and $(ii)$ from the definition of $\Delta_{H}$ as a quotient, and so is part $(iv)$ in view of the definition of $\downarrow_{H}$ in Equation \eqref{arrowdef}. This proves the lemma.
\end{proof}

\begin{rmk}
In case $H_{M}$, or equivalently $H_{D}$, is non-degenerate, the papers \cite{[Ma]} and \cite{[Ze5]} give explicit formulae involving the arrow operator $\downarrow_{H}$ using the resulting orthogonal decomposition of $D_{M}$ and $D$. This corresponds to taking the set $R$ in Lemma \ref{arrowhoriz} to be $H_{M}^{\perp}$ (it is a set of representatives for $D_{M}/H_{M}$ precisely when $H_{M}$ is non-degenerate), and then there is a corresponding set for $D/H_{D}$ (which we shall not need). We will, however, work more generally, without assuming anything on $H_{M}$. \label{HMgen}
\end{rmk}

\smallskip

Considering $D_{M} \oplus D$ as a discriminant form, Definition \ref{opersDM} produces the corresponding operators with any index in $D_{M} \oplus D$. In particular, the operators having index from $D_{M} \subseteq D_{M} \oplus D$ act on $\mathbb{C}[D_{M} \oplus D]$, which may be seen as an extension of the action of the operators from Definition \ref{opersDM} (for $D_{M}$) to an action on this larger space. We will need an induced action of the operators from Definition \ref{opersDM} on $\mathbb{C}[\Delta_{H}]$.
\begin{lem}
The operators associated with elements of $M^{*}$ whose $D_{M}$-images are in $H_{M}^{\perp}$ induce a well-defined action on $\mathbb{C}[\Delta_{H}]$. \label{indchit}
\end{lem}

\begin{proof}
The subgroup $H_{M}^{\perp}$ of $D_{M}$, viewed as embedded into $D_{M} \oplus D$, is contained in $H^{\perp}$, and clearly intersects the horizontal subgroup $H$ trivially. It therefore admits an isomorphic copy also in $\Delta_{H}$. Definition \ref{opersDM} then yields operators on $\mathbb{C}[\Delta_{H}]$ for every element of $\Delta_{H}$, and the desired operators are those arising from this image of $H_{M}^{\perp}$ inside $\Delta_{H}$. This proves the lemma.
\end{proof}
Using the induced operators from Lemma \ref{indchit}, we can now consider pairings with theta functions. Since we work with the Hermitian ones, we assume that $M$ is a Hermitian $\mathcal{O}$-lattice, producing all the objects defined above. However, recall that we represent the basis for $\mathbb{C}[\Delta_{H}]$ through the set $H^{\perp}_{R}$ from Lemma \ref{arrowhoriz}, which is based on a set of representatives $R \subseteq D_{M}$ for $D_{M}/H_{M}$. As this set need not respect the addition from $H_{M}^{\perp}$, we shall be using the following notation and its properties.
\begin{lem}
Let a set $R \subseteq D_{M}$ of representatives for $D_{M}/H_{M}$ be given.
\begin{enumerate}[$(i)$]
\item For $\delta \in D_{M}$ and $\sigma \in M^{*}$, there exists a unique pair of elements $r_{\sigma}(\delta) \in R$ and $h_{\sigma}(\delta) \in H$ such that $\delta+(\sigma+M) \in D_{M}$ equals $r_{\sigma}(\delta)+h_{\sigma}(\delta)_{M}$.
\item We have the equalities $r_{-\sigma}\big(r_{\sigma}(\delta)\big)=\delta$ and $h_{-\sigma}\big(r_{\sigma}(\delta)\big)=-h_{\sigma}(\delta)$.
\end{enumerate} \label{addRsubgrp}
\end{lem}

\begin{proof}
Part $(i)$ follows directly from the fact that $R$ is a set of representatives for $D_{M}/H_{M}$ and the projection from $H$ onto $H_{M}$ is bijective. For part $(ii)$ we just note that the defining equation from part $(i)$ yields the equality of $r_{\sigma}(\delta)-(\sigma+M)$ and $\delta-h_{\sigma}(\delta)_{M}$, and we apply the uniqueness from part $(i)$ with $\delta$ replaced by $r_{\sigma}(\delta)$ and $\sigma$ inverted. This proves the lemma.
\end{proof}

We can now state and prove the following result.
\begin{prop}
Let $F:\mathcal{H}\to\mathbb{C}[D]$ be any smooth function. Then the function $\Phi_{M,v}^{F,H}:\mathcal{H} \times M_{\mathbb{C}} \times M_{\overline{\mathbb{C}}}\to\mathbb{C}[\Delta_{H}]$ taking $(\tau,\zeta,\omega)$ to $\downarrow_{H}\big(\Theta_{M}(\tau,\zeta,\omega;v) \otimes F(\tau)\big)$ is holomorphic in the variables $\zeta$ and $\omega$ and satisfies, for every $\sigma$ and $\nu$ from $M^{*}$ with $\sigma+M$ and $\nu+M$ in $H_{M}^{\perp}$, the equality \[\Phi_{M,v}^{F,H}(\tau,\zeta+\tau\sigma_{v_{+},\mathbb{C}}+\overline{\tau}\sigma_{v_{-},\mathbb{C}}+\nu_{\mathbb{C}},\omega+\tau\sigma_{v_{+},\overline{\mathbb{C}}}+\overline{\tau}\sigma_{v_{-},\overline{\mathbb{C}}}+ \nu_{\overline{\mathbb{C}}})=\] \[=\mathbf{e}\big(-\tau|\sigma_{v_{+}}|^{2}-\overline{\tau}|\sigma_{v_{-}}^{2}|-\langle\zeta,\sigma\rangle-\langle\sigma,\omega\rangle\big)\chi_{\nu}t_{\sigma}\Phi_{M,v}^{F,H}(\tau,\zeta,\omega).\] \label{pairper}
\end{prop}

\begin{proof}
By writing $F=\sum_{\delta \in D}f_{\delta}\mathfrak{e}_{\delta}$, the tensor product $\Theta_{M} \otimes F$ (with the appropriate values) is, under the natural identification of $\mathbb{C}[D_{M}]\otimes_{\mathbb{C}}\mathbb{C}[D]$ with $\mathbb{C}[D_{M} \oplus D]$, the $\mathbb{C}[D_{M} \oplus D]$-valued function given by $\sum_{\alpha \in D_{M} \oplus D}\theta_{\alpha_{M}}f_{\alpha_{D}}\mathfrak{e}_{\alpha}$. Then its $\downarrow_{H}$-image $\Phi_{M,v}^{F,H}$ is given, by the parts of Lemma \ref{arrowhoriz}, by
\begin{equation}
\Phi_{M,v}^{F,H}(\tau,\zeta,\omega)=\textstyle{\sum_{\beta \in H^{\perp}_{R}}\big[\sum_{\gamma \in H}\theta_{\beta_{M}+\gamma_{M}}(\tau,\zeta,\omega;v)f_{\beta_{D}+\gamma_{D}}(\tau)\big]\mathfrak{e}_{\overline{\beta}}}. \label{downarrow}
\end{equation}
For the periodicity formula, Proposition \ref{perTheta} and its proof show that changing $\zeta$ and $\omega$ by the parameters arising from $\sigma$ and $\nu$ yields the asserted exponent times the expression \[\textstyle{\sum_{\beta \in H^{\perp}_{R}}\big[\sum_{\gamma \in H}\mathbf{e}\big((\beta_{M}+\gamma_{M},\nu+M)\big)\theta_{\beta_{M}+\gamma_{M}+(\sigma+M)}(\tau,\zeta,\omega;v)f_{\beta_{D}+\gamma_{D}}(\tau)\big]\mathfrak{e}_{\overline{\beta}}}.\] But if $\nu+M \in H_{M}^{\perp}$ then the pairing in this exponent reduces to $(\beta_{M},\nu+M)$, or equivalently $(\beta,\nu+M)$ (or also $(\overline{\beta},\nu+M)$), independently of $\gamma$. Moreover, if $\sigma+M \in H_{M}^{\perp}$ then we write $\beta_{M}+\gamma_{M}+(\sigma+M)$ as $r_{\sigma}(\beta_{M})+\gamma_{M}+h_{\sigma}(\beta_{M})_{M}$ via part $(i)$ of Lemma \ref{addRsubgrp}, and thus $\beta_{D}+\gamma_{D}$ becomes the sum of $\beta_{D}-h_{\sigma}(\beta_{M})_{D}$ and the corresponding element $\gamma_{D}+h_{\sigma}(\beta_{M})_{D}$ of $H_{D}$. By taking $\tilde{\beta}$ to be with $\tilde{\beta}_{M}=r_{\sigma}(\beta_{M})$ and $\tilde{\beta}_{D}=\beta_{D}-h_{\sigma}(\beta_{M})_{D}$, and $\tilde{\gamma}:=\gamma+h_{\sigma}(\beta_{M})$, we deduce from part $(ii)$ of that Lemma \ref{addRsubgrp} that $\beta_{M}=r_{-\sigma}(\tilde{\beta}_{M})$ and $\beta_{D}=\tilde{\beta}_{D}-h_{-\sigma}(\tilde{\beta}_{M})_{D}$. We can thus take the sum to be over $\tilde{\beta} \in H^{\perp}_{R}$ and $\tilde{\gamma} \in H$, with the theta function $\theta_{\tilde{\beta}_{M}+\tilde{\gamma}_{M}}$ and the coefficient $f_{\tilde{\beta}_{D}+\tilde{\gamma}_{D}}$, and the associated vector $\mathfrak{e}_{\overline{\beta}}$ being the one associated with the $\Delta_{H}$-image of $\tilde{\beta}-(\sigma+M)$. As this combines with the effect of the exponents $(\overline{\beta},\nu+M)$ to yield the effect of $\chi_{\nu}t_{\sigma}$ from Lemma \ref{indchit} on $\Phi_{M,v}^{F,H}(\tau,\zeta,\omega)$, the desired formula follows. This proves the proposition.
\end{proof}

\smallskip

\begin{rmk}
When $\sigma$ and $\nu$ are in $M$, the relation from Proposition \ref{pairper} amounts to the scalar-valued components of $\Phi=\Phi_{M,v}^{F,H}=\sum_{\beta\in\Delta_{H}}\Phi_{\beta}\mathfrak{e}_{\beta}$ satisfying the one from Equation \eqref{perJac}. Therefore we can, by Corollary \ref{decomtheta}, write every smooth function $\Phi:\mathcal{H} \times M_{\mathbb{C}} \times M_{\overline{\mathbb{C}}}\to\mathbb{C}[\Delta_{H}]$ that is holomorphic in the variables from $M_{\mathbb{C}}$ and from $M_{\overline{\mathbb{C}}}$ and satisfies this relation (given also in Equation \eqref{perHMperp} below) as $\Phi(\tau,\zeta,\omega)=\sum_{\beta\in\Delta_{H}}[\Theta_{M}(\tau,\zeta,\omega;v),F_{\beta}(\tau)]_{M}\mathfrak{e}_{\beta}$ for smooth functions $F_{\beta}:\mathcal{H}\to\mathbb{C}[D_{\overline{M}}]$, which are determined by $\Phi$. \label{expcomp}
\end{rmk}

We can now define our vector-valued Hermitian Jacobi forms.
\begin{defn}
Let an $\mathcal{O}$-lattice $M$ from Definition \ref{Olat}, another discriminant form $D$, a horizontal isotropic subgroup $H \subseteq D_{M} \oplus D$ with image $H_{M} \subseteq D_{M}$, weights $k$ and $l$ from $\mathbb{Z}$, a subgroup $\Gamma\subseteq\operatorname{U}(1,1)(\mathbb{Z})$ of finite index, and an element $v\operatorname{Gr}^{J}(M_{\mathbb{R}})$ from Equation \eqref{GrassJ} be given. We then set $\Delta_{H}:=H^{\perp}/H$, and define a \emph{Hermitian Jacobi form of weight $(k,l)$, index $(M,H,v)$, and representation $\rho_{\Delta_{H}}$ with respect to $\Gamma$} be a function $\Phi:\mathcal{H} \times M_{\mathbb{C}} \times M_{\overline{\mathbb{C}}}\to\mathbb{C}[\Delta_{H}]$, that is smooth and has the following properties. The first one is the usual holomorphicity in the variables from $M_{\mathbb{C}}$ and $M_{\overline{\mathbb{C}}}$. For the second one, recall from Lemma \ref{indchit} the operators on $\mathbb{C}[\Delta_{H}]$ that are induced from those from Definition \ref{opersDM}, and then for $\tau\in\mathcal{H}$, $\zeta \in M_{\mathbb{C}}$, $\omega \in M_{\overline{\mathbb{C}}}$, and $\sigma$ and $\nu$ from $M^{*}$ with $D_{M}$-images in $H_{M}^{\perp}$, the equality
\[\Phi(\tau,\zeta+\tau\sigma_{v_{+},\mathbb{C}}+\overline{\tau}\sigma_{v_{-},\mathbb{C}}+\nu_{\mathbb{C}},\omega+\tau\sigma_{v_{+},\overline{\mathbb{C}}}+\overline{\tau}\sigma_{v_{-},\overline{\mathbb{C}}}+ \nu_{\overline{\mathbb{C}}})=\]
\begin{equation}
=\mathbf{e}\big(-\tau|\sigma_{v_{+}}|^{2}-\overline{\tau}|\sigma_{v_{-}}^{2}|-\langle\zeta,\sigma\rangle-\langle\sigma,\omega\rangle\big)\chi_{\nu}t_{\sigma}\Phi(\tau,\zeta,\omega)
\label{perHMperp}
\end{equation}
has to be satisfied. For the third one, let $j(A,\tau)$ and $j^{c}(A,\tau)$ be as in Theorem \ref{modTheta}, and then for such $\tau$, $\zeta$, and $\omega$, and for $A\in\Gamma$, we require the equality \[\Phi\Big(A\tau,\tfrac{\zeta_{v_{+}}}{j(A,\tau)}+\tfrac{\zeta_{v_{-}}}{\overline{j^{c}(A,\tau)}},\tfrac{\omega_{v_{+}}}{j^{c}(A,\tau)}+\tfrac{\omega_{v_{-}}}{\overline{j(A,\tau)}}\Big)=\] \begin{equation}
=j(A,\tau)^{k}\overline{j(A,\tau)}^{l}\mathbf{e}\Big(\tfrac{j_{A}'\langle\zeta_{v_{+}},\omega_{v_{+}}\rangle}{j(A,\tau)}+\tfrac{\overline{j_{A}'}\langle\zeta_{v_{-}},\omega_{v_{-}}\rangle}{\overline{j(A,\tau)}}\Big) \rho_{\Delta_{H}}(A)\Phi(\tau,\zeta,\omega). \label{VVmodJac}
\end{equation}
The fourth one is the requirement that for each $\beta\in\Delta_{H}$, the components of the associated function $F_{\beta}:\mathcal{H}\to\mathbb{C}[D_{\overline{M}}]$ from Remark \ref{expcomp} grow at most linearly exponentially towards any cusp of $\Gamma$. The condition for $\Phi$ to be \emph{pseudo-holomorphic} is as in Definition \ref{Jacdef}. \label{VVdef}
\end{defn}

\smallskip

We can now prove our main result for vector-valued Hermitian Jacobi forms.
\begin{thm}
Any smooth function $\Phi:\mathcal{H} \times M_{\mathbb{C}} \times M_{\overline{\mathbb{C}}}\to\mathbb{C}[\Delta_{H}]$ that satisfied Equation \eqref{perHMperp} and the holomorphicity condition from Definition \ref{VVdef} is of the form $\Phi_{M,v}^{F,H}$ from Proposition \ref{pairper} for a smooth function $F:\mathcal{H}\to\mathbb{C}[D]$, which is then unique. In this case $\Phi$ is a Hermitian Jacobi form of weight $(k,l)$, index $(M,H,v)$, and representation $\rho_{\Delta_{H}}$ with respect to a finite index subgroup $\Gamma\subseteq\operatorname{U}(1,1)(\mathbb{Z})$ if and only $F$ is a modular form of weight $(k-b_{+},l-b_{-})$ and representation $\rho_{D}$ with respect to $\Gamma$, as in Equation \eqref{modeq}. Moreover, $F$ is weakly holomorphic if and only if $\Phi$ is pseudo-holomorphic, and then $l=b_{-}$. \label{VVmain}
\end{thm}

\begin{proof}
We shall be using the representing set $H^{\perp}_{R}$ from Lemma \ref{arrowhoriz} for expressing the elements of $\Delta_{H}$, where $R \subseteq D_{M}$ is some set of representatives for $D_{M}/H_{M}$. Then if $\overline{\beta}\in\Delta_{H}$ is the image of $\beta \in H^{\perp}_{R}$ then Remark \ref{expcomp} allows us, via Corollary \ref{decomtheta}, to write the component $\Phi_{\overline{\beta}}$ as the $[\cdot,\cdot]_{M}$-pairing of $\Theta_{M}$ with a unique smooth function $F_{\beta}:\mathcal{H}\to\mathbb{C}[D_{\overline{M}}]$. By expanding this pairing, we obtain that we can write \[\textstyle{\Phi(\tau,\zeta,\omega)=\sum_{\beta \in H^{\perp}_{R}}\big[\sum_{\alpha \in D_{M}}\theta_{\alpha}(\tau,\zeta,\omega;v)\tilde{f}_{\beta,\alpha}(\tau)\big]\mathfrak{e}_{\overline{\beta}}},\] where $\{\tilde{f}_{\beta,\alpha}\}_{\alpha \in D_{M}}$ are the components of $F_{\beta}$ for any $\beta \in H^{\perp}_{R}$.

We now consider $\sigma=0$ but any $\nu \in M^{*}$ with $\nu+M \in H_{M}^{\perp}$ in Equation \eqref{perHMperp}, and compare it with the corresponding property of $\theta_{\alpha}$ from Proposition \ref{perTheta}. The action of $\chi_{\nu}$ in Equation \eqref{perHMperp} implies that the summand associated with $\beta \in H^{\perp}_{R}$ and $\alpha \in D_{M}$ has to be multiplied by $\mathbf{e}\big((\nu+M,\beta_{M})\big)$, while via Proposition \ref{perTheta} it has to be multiplied by $\mathbf{e}\big((\nu+M,\alpha)\big)$. As both sides have to be equal and the theta expansions are unique, we deduce that this summand can be non-zero, i.e., we have $f_{\beta,\alpha}\neq0$, only when $\mathbf{e}\big((\nu+M,\beta_{M})\big)=\mathbf{e}\big((\nu+M,\alpha)\big)$ for every such $\nu$. But the pairing with the elements $\nu+M \in H_{M}^{\perp}$ determine the element of $D_{M}$ up to translations by $H_{M}$, meaning that $\alpha$ has to be of the form $\beta_{M}+\gamma_{M}$ for $\gamma \in H$. By writing each $\beta \in H^{\perp}_{R}$ as the sum of $\beta_{M} \in R$ and $\beta_{D} \in D$ (which satisfy the equality from part $(ii)$ of Lemma \ref{horizdecom}), it follows that our expansion becomes
\begin{equation}
\textstyle{\Phi(\tau,\zeta,\omega)=\sum_{\beta \in H^{\perp}_{R}}\big[\sum_{\gamma \in H}\theta_{\beta_{M}+\gamma_{M}}(\tau,\zeta,\omega;v)\hat{f}_{\beta_{M}+\beta_{D},\gamma}(\tau)\big]\mathfrak{e}_{\overline{\beta_{M}+\beta_{D}}}},
\label{expnusigma}
\end{equation}
where $\hat{f}_{\beta,\gamma}$, for $\beta=\beta_{M}+\beta_{D}$, is the function $\tilde{f}_{\beta,\beta_{M}+\gamma_{M}}$ from above.

We now turn to $\sigma \in M^{*}$ in Equation \eqref{perHMperp} and Proposition \ref{perTheta}, and observe that the exponents showing up in both of them are the same. We thus take $\sigma$ with $\sigma+M \in H_{M}^{\perp}$ (for preserving $H^{\perp}$ via the proof of Lemma \ref{indchit} and for making Equation \eqref{perHMperp} applicable), replace $\zeta$ and $\omega$ in Equation \eqref{expnusigma} by the resulting variables from the left hand sides of Equation \eqref{perHMperp} and Proposition \ref{perTheta}, and apply them. After canceling the common exponent, changing the summation indices on the left hand side, and expanding the resulting right hand side via Equation \eqref{expnusigma}, we obtain the equality \[t_{\sigma}\sum_{\substack{\tilde{\beta} \in H^{\perp}_{R} \\ \tilde{\gamma} \in H}}\theta_{\tilde{\beta}_{M}+\tilde{\gamma}_{M}}\hat{f}_{\tilde{\beta}_{M}+\tilde{\beta}_{D},\tilde{\gamma}}\mathfrak{e}_{\overline{\tilde{\beta}_{M}+\tilde{\beta}_{D}}}=\sum_{\substack{\beta \in H^{\perp}_{R} \\ \gamma \in H}}\theta_{\beta_{M}+\gamma_{M}+(\sigma+M)}\hat{f}_{\beta_{M}+\beta_{D},\gamma}\mathfrak{e}_{\overline{\beta_{M}+\beta_{D}}},\] with the usual variables $\tau$, $\zeta$, $\omega$, and $v$ suppressed.

Now, as in the proof of Proposition \ref{pairper}, we replace $\tilde{\beta}_{M}$ on the right hand side by $\tilde{\beta}_{M}=r_{\sigma}(\beta_{M})$ from part $(i)$ of Lemma \ref{addRsubgrp}, set $\tilde{\gamma}:=\gamma+h_{\sigma}(\beta_{M})$, and define $\tilde{\beta}_{D}=\beta_{D}-h_{\sigma}(\beta_{M})_{D}$, so that the theta function there becomes $\theta_{\tilde{\beta}_{M}+\tilde{\gamma}_{M}}$ like on the left hand side. Moreover, via part $(ii)$ of that lemma we deduce that the vector $\mathfrak{e}_{\overline{\beta_{M}+\beta_{D}}}$ from the right hand side is the one obtained by subtracting $\sigma+M$ from the index of $\mathfrak{e}_{\overline{\tilde{\beta}_{M}+\tilde{\beta}_{D}}}$, namely it equals the $t_{\sigma}$-image of the latter vector. Since the theta expansion in Remark \ref{expcomp} and Corollary \ref{decomtheta} is unique, we obtain an equality in the coefficients from both sides as functions of $\tau$, which by our substitutions amounts, for any $\beta \in H^{\perp}_{R}$, $\gamma \in H$, and $\sigma \in M^{*}$ with $\sigma+M \in H_{M}^{\perp}$, to the equality
\begin{equation}
\hat{f}_{\beta_{M}+\beta_{D},\gamma}=\hat{f}_{r_{\sigma}(\beta_{M})+\beta_{D}-h_{\sigma}(\beta_{M})_{D},\gamma+h_{\sigma}(\beta_{M})}. \label{funceq}
\end{equation}

We now construct the smooth $\mathbb{C}[D]$-valued function $F=\sum_{\alpha_{D} \in D}f_{\alpha_{D}}$ as follows. Since the pairing on $D_{M}$ is non-degenerate, the map from $D_{M}$ to the quotient $\operatorname{Hom}(H_{M},\mathbb{Q}/\mathbb{Z})$ or $\operatorname{Hom}(D_{M},\mathbb{Q}/\mathbb{Z})$ is surjective, and in particular there exists $\alpha_{M} \in D_{M}$ such that the equality $(\alpha_{M},\gamma_{M})=-\big(\alpha_{D},\iota(\gamma_{M})\big)$ from part $(ii)$ of Lemma \ref{horizdecom} holds for any $\gamma \in H$, so that $\alpha:=\alpha_{M}+\alpha_{D} \in D_{M} \oplus D$ lies in $H^{\perp}$. We thus choose such an $\alpha_{M}$, and note that any other choice is obtained by subtracting some element of $H_{M}^{\perp}$. We express $\alpha$ as $\beta+\gamma$ as in part $(i)$ of Lemma \ref{arrowhoriz}, namely with $\beta_{M} \in R$ and $\gamma \in H$, and we set $f_{\alpha_{D}}$ to be $\hat{f}_{\beta,\gamma}$. We observe that if we replace $\alpha_{M}$ by $\alpha_{M}-(\sigma+M)$ for some $\sigma \in M^{*}$ with $H_{M}^{\perp}$ (which will produce all the other choices for $\alpha_{M}$) then $\beta_{M}$ will be taken to $r_{\sigma}(\beta)_{M}$, $\gamma$ will be sent to $\gamma+h_{\sigma}(\beta_{M})$, and $\beta_{D}$ will thus become $\beta_{D}-h_{\sigma}(\beta_{M})_{D}$. As the function arising from this choice equals the one from our initial choice by Equation \eqref{funceq}, we deduce that $f_{\alpha_{D}}$ is well-defined.

Finally, the construction implies that given $\beta \in H^{\perp}_{R}$ and $\gamma \in H$ is $f_{\alpha_{D}}$, the function $\hat{f}_{\beta,\gamma}$ is $f_{\alpha_{D}}$ for $\alpha:=\beta+\gamma$, namely $f_{\beta_{D}+\gamma_{D}}$. By substituting this equality into the right hand side of Equation \eqref{expnusigma} and comparing with Equation \eqref{downarrow}, we deduce that our $\Phi$ equals $\Phi_{M,v}^{F,H}$ with $F=\sum_{\alpha_{D} \in D}f_{\alpha_{D}}$, as desired. The fact that $F$ is a modular form if and only if $\Phi=\Phi_{M,v}^{F,H}$ is a Hermitian Jacobi form is established using the uniqueness of our theta expansion like in the proof of Theorem \ref{main}, and the relation between the holomorphicity of $F$ and the pseudo-holomorphicity of $\Phi$ is deduced as in Proposition \ref{pshol}, since the operators from Proposition \ref{difeq} act in the same manner on scalar-valued functions and on vector-valued functions with constant bases. This completes the proof of the theorem.
\end{proof}

\begin{rmk}
The proofs of Proposition \ref{pairper} and Theorem \ref{VVmain} can be significantly simplified in the case where $H_{M}$ is non-degenerate as a subgroup of $D_{M}$, where formulae like those mentioned in Remark \ref{HMgen} can also be obtained. Indeed, in this case we can take $R$ to be just the orthogonal complement $H_{M}^{\perp}$, where in Lemma \ref{addRsubgrp} we have $r_{\sigma}(\delta)=\delta+(\sigma+M)$ and $h_{\sigma}(\delta)=0$ for every $\delta \in R$ and $\sigma$ with $D_{M}$-image in $H_{M}^{\perp}$, and the expressions in these proofs become much more natural. This case also has the advantage that $D$ can be constructed directly from $\Delta$, $D_{M}$, and $H_{M}^{\perp}$, a construction that does not seem easy, or even possible, in other situations. In particular, the case where $H=\{0\}$ produces Jacobi forms that are just tensor products $\Theta_{M} \otimes F$ for modular forms $F$, and if $D=D_{\overline{M}}$ and $H$ is the natural diagonal subgroup, with $H^{\perp}=H$, Proposition \ref{pairper} and Theorem \ref{VVmain} reproduce Corollary \ref{decomtheta} and Theorem \ref{main}, because the pairing $[\Theta_{M},F]$ is the same as $\downarrow_{H}(\Theta_{M} \otimes F)$ for this $H$. However, in applications, like those from \cite{[BRZ]} in the holomorphic setting to be expressed explicitly below, this construction has to be carried out with horizontal subgroups with arbitrary projections, so we needed to work in this generality. We also note that in applications in the Hermitian setting, typically $D$ and the isotropic subgroup $H$ are assumed to carry the $J$-structure mentioned in Remark \ref{JonD}, but this does not affect any part of the construction or the results. \label{noconstD}
\end{rmk}

\smallskip

All of the constructions involving discriminant forms only used their structure involving the $\mathbb{Q}/\mathbb{Z}$-valued quadratic form, which can be obtained from the natural construction of the discriminant form associated with an even orthogonal lattice over $\mathbb{Z}$. Since \cite{[Ze3]} only considered scalar-valued Jacobi forms, we give the definition of vector-valued Jacobi forms in this setting as well, and state the corresponding theorem.

So let $L$ be an even lattice over $\mathbb{Z}$, with quadratic form $\lambda\mapsto\frac{\lambda^{2}}{2}$ and induced pairing $(\cdot,\cdot)$, of some signature $(b_{+},b_{-})$, with discriminant $D_{L}$. Given $v$ in the associated Grassmannian from Equation (1) of \cite{[Ze3]} (or equivalently Equation \eqref{GrassJ} here), and elements $\tau\in\mathcal{H}$ and $\zeta \in L_{\mathbb{C}}$, let
\begin{equation}
\Theta_{L}(\tau,\zeta;v):=\textstyle{\sum_{\gamma \in D_{L}}\big[\sum_{\lambda \in L+\gamma}\mathbf{e}\big(\tau\lambda_{v_{+}}^{2}/2+\overline{\tau}\lambda_{v_{-}}^{2}/2+(\lambda,\zeta)\big)\big]\mathfrak{e}_{\gamma}} \label{Thetaorth}
\end{equation}
be the associated theta function from Equation (7) of \cite{[Ze3]}. It satisfies not only the periodicity property from Proposition 1.1 there, but the extended one given in Proposition \ref{perTheta}, using the operators from Definition \ref{opersDM}. We write $\operatorname{Mp}_{2}(\mathbb{Z})$ for the double cover of $\operatorname{SL}_{2}(\mathbb{Z})$ that is used for modular forms of half-integral weight, with $\rho_{L}$ the representation of $\operatorname{Mp}_{2}(\mathbb{Z})$ on $\mathbb{C}[D_{L}]$ as defined in Equation (5) of \cite{[Ze3]} (or Equation \eqref{Weildef} here).

Let $D$ be another discriminant form, take a horizontal isotropic subgroup $H$ of $D_{L} \oplus D$, and let $H_{L}$ and $H_{D}$ be the projections of $H$ as usual, with $\Delta_{H}:=H^{\perp}/H$ the associated quotient as above. Then Lemma \ref{indchit} holds here equally well, when the elements $\nu$ and $\sigma$ from $L^{*}$ are such that their images in $D_{L}$ are in $H_{L}^{\perp}$, and the vector-valued Jacobi forms in this setting are the following functions.
\begin{defn}
Consider a finite index subgroup $\Gamma$ of $\operatorname{Mp}_{2}(\mathbb{Z})$, two weights $k$ and $l$ from $\frac{1}{2}\mathbb{Z}$. Then a \emph{Jacobi form of weight $(k,l)$, index $(L,H,v)$, and representation $\rho_{\Delta_{H}}$ with respect to $\Gamma$} is a smooth function $\Phi:\mathcal{H} \times L_{\mathbb{C}}\to\mathbb{C}$ for which the following four conditions are satisfied: First, we require the equation
\begin{equation}
\Phi(\tau,\zeta+\tau\sigma_{v_{+}}+\overline{\tau}\sigma_{v_{-}}+\nu)=\mathbf{e}\big(-\tau\sigma_{v_{+}}^{2}/2-\overline{\tau}\sigma_{v_{-}}^{2}/2-(\sigma,\zeta)\big)\chi_{\nu}t_{\sigma}\Phi(\tau,\zeta) \label{perorth}
\end{equation}
for any $\tau\in\mathcal{H}$, $\zeta \in L_{\mathbb{C}}$, and elements $\sigma$ and $\nu$ in $L^{*}$ for which $\sigma+L$ and $\nu+L$ are in $H_{L}^{\perp}$; Second, the we demand the equality
\begin{equation}
\Phi\Big(A\tau,\tfrac{\zeta_{v_{+}}}{j(A,\tau)}+\tfrac{\zeta_{v_{-}}}{\overline{j(A,\tau)}}\Big)= j(A,\tau)^{k}\overline{j(A,\tau)}^{l}\mathbf{e}\bigg(\tfrac{j_{A}'\zeta_{v_{+}}^{2}}{2j(A,\tau)}+\tfrac{j_{A}'\zeta_{v_{-}}^{2}}{2\overline{j(A,\tau)}}\bigg)\rho_{\Delta_{H}}(A)\Phi(\tau,\zeta) \label{modorth}
\end{equation}
to hold for any $\tau$ and $\zeta$ and for any $A\in\Gamma$, with the half-integral powers of $j$ determined by the metaplectic data of $A$ if necessary; Third, for any fixed $\tau\in\mathcal{H}$, the function $\zeta\mapsto\Phi(\tau,\zeta)$ is holomorphic on $L_{\mathbb{C}}$; And fourth, the functions deduced from the theta expansions of all the coefficients of $\Phi$ via Corollary 2.2 of \cite{[Ze3]} are bounded, towards any cusp of $\Gamma$, by the exponential of a linear function of the height there. We saw that $\Phi$ is \emph{pseudo-holomorphic} when it is annihilated by the operator $4\pi i\partial_{\overline{\tau}}-\Delta_{v_{-}}^{h}$ from Proposition 1.3 of \cite{[Ze3]} (or Proposition \ref{difeq} here), and \emph{skew-holomorphic} if it is annihilated by the operator $4\pi i\partial_{\tau}-\Delta_{v_{+}}^{h}$ defined there. \label{orthVV}
\end{defn}

The main result concerning the Jacobi forms from Definition \ref{orthVV} is the following analogue of Proposition \ref{pairper} and Definition \ref{VVdef}.
\begin{thm}
Given a smooth function $\Phi:\mathcal{H} \times L_{\mathbb{C}}\to\mathbb{C}[\Delta_{H}]$, Equation \eqref{perorth} and the holomorphicity in Definition \ref{orthVV} are equivalent to the existence of a smooth function $F:\mathcal{H}\to\mathbb{C}[D]$, which is then determined by $\Phi$, such that $\Phi(\tau,\zeta)=\downarrow_{H}\big(\Theta_{L}(\tau,\zeta;v) \otimes F(\tau)\big)$ for every $\tau\in\mathcal{H}$ and $\zeta \in L_{\mathbb{C}}$, where $\Theta_{L}$ is the function from Equation \eqref{Thetaorth}. In this case Equation \eqref{modorth} for $\Phi$ is equivalent to $F$ satisfying Equation \eqref{modeq} with the weight $\big(k-\frac{b_{+}}{2},l-\frac{b_{-}}{2}\big)$ and the representation $\rho_{D}$, and $\Phi$ is a Jacobi form of weight $(k,l)$ and index $(L,H,v)$, and representation $\rho_{\Delta_{H}}$ with respect to $\Gamma$ if and only if $F$ is modular of weight $\big(k-\frac{b_{+}}{2},l-\frac{b_{-}}{2}\big)$ and representation $\rho_{D}$. The Jacobi form is pseudo-holomorphic if and only if $F$ is weakly holomorphic, a case which can only occur when $l=\frac{b_{-}}{2}$. \label{VVorth}
\end{thm}
The proof of Theorem \ref{VVorth} goes just like that of Theorem \ref{VVmain}, with the orthogonal notions. The maps from Lemma 3.1 and Corollary 3.2 of \cite{[Ze3]} also apply to the discriminants $D$ and $\Delta_{H}$, yielding an analogous result to Theorem \ref{VVorth}, in which the holomorphicity of $F$ is equivalent to the corresponding Jacobi form (in this construction) being skew-holomorphic.

\section{Relations with Arrow Operators \label{RelsArrow}}

Recall from Proposition 3.4 of \cite{[Ze3]} that if the modular form $F$ producing a scalar-valued Jacobi form $\Phi$ is the $\uparrow$-operator image of another modular form $G$, then $\Phi$ can also be obtained from $G$ using the theta function of the associated over-lattice (this works equally well for Hermitian $\mathcal{O}$-lattices, as mentioned above). We shall need similar relations in the vector-valued setting. We shall work with an orthogonal $\mathbb{Z}$-lattice $L$, but as above, all the results in this section hold for any Hermitian $\mathcal{O}$-lattice $M$ as well, provided that all the subgroups involved in the calculations, and hence all the resulting quotients, are closed under the action of $\mathcal{O}$.

\smallskip

Our calculations will be based on a certain combination of arrow operators. For defining it, we shall need some notation. Let $A$ be any discriminant form, and take two isotropic subgroups $H$ and $I$ of $A$. Then the subgroup $I+(H \cap I^{\perp})$ is also isotropic (as so are $I$ and $H \cap I^{\perp}$, and they are perpendicular by definition), and as for subgroups $Y$ and $Z$ of $A$ one has $(Y^{\perp})^{\perp}=Y$, $(Y+Z)^{\perp}=Y^{\perp} \cap Z^{\perp}$, and $(Y \cap Z)^{\perp}=Y^{\perp}+Z^{\perp}$, the subgroup perpendicular to ours is $I^{\perp}\cap(H^{\perp}+I)$, which equals $(I^{\perp} \cap H^{\perp})+I$ since $I \subseteq I^{\perp}$.

We shall be using the following simple observation below.
\begin{lem}
The quotients $|H|/|H \cap I^{\perp}|$ and $|I|/|I \cap H^{\perp}|$ have the same size. This size is $|I|$ when $I \cap H^{\perp}=\{0\}$. \label{perfpair}
\end{lem}

\begin{proof}
We restrict the pairing on $A$ to a pairing from $H \times I$ to $\mathbb{Q}/\mathbb{Z}$. Since this pairing is trivial when the first element lies in $H \cap I^{\perp}$ or when second one in is $I \cap H^{\perp}$, we obtain an induced pairing on the product of $H/(H \cap I^{\perp})$ and $I/(I \cap H^{\perp})$, the sizes of which are the asserted quotients. Moreover, these intersections are precisely the subgroups of elements for which the pairing with the other group is trivial, which implies that in the quotients, the pairing is now perfect, and embeds each group with the dual of the other. As the dual group of a finite group has the same size of the group itself, these embeddings yield two opposite inequalities between the sizes of these groups, whence the desired equality (and indeed, for finite groups such embeddings are isomorphisms when the pairing is perfect). The second assertion is an immediate special case of the first. This proves the lemma.
\end{proof}
We shall also use the simplest case of Lemma \ref{perfpair} in which both sizes are 1. This amounts to the equivalence of the statements $I \cap H^{\perp}=I$, $I \subseteq H^{\perp}$, the pairing between $H$ and $I$ is trivial, $H \subseteq I^{\perp}$, and $H \cap I^{\perp}=H$.

\smallskip

Both $H$ and $I+(H \cap I^{\perp})$ contain $H \cap I^{\perp}$, with the quotient $\overline{I}$ resulting from the second inclusion being isomorphic to $I/(I \cap H)$ (since $I \subseteq I^{\perp}$ again). We set $\overline{A}:=(H^{\perp}+I)/(H \cap I^{\perp})$, $B:=H^{\perp}/H$, and $\overline{B}:=[(I^{\perp} \cap H^{\perp})+I]/[I+(H \cap I^{\perp})]$, and then $B$ and $\overline{B}$ can also be obtained from $\overline{A}$ using the isotropic subgroups $H/(H \cap I^{\perp})$ and $\overline{I}$ respectively. We shall need the following operator.
\begin{lem}
The subgroup $[H+(I^{\perp} \cap H^{\perp})]/H \subseteq B$ admits a natural surjective map $p$ onto $\overline{B}$, which preserves the quadratic values in $\mathbb{Q}/\mathbb{Z}$ and whose kernel is $[H+(I \cap H^{\perp})]/H$. The combination $\downarrow_{H/(H \cap I^{\perp})}\uparrow_{\overline{I}}:\mathbb{C}[\overline{B}]\to\mathbb{C}[B]$ is the injective map taking $\mathfrak{e}_{\alpha}$ for $\alpha\in\overline{B}$ to $\sum_{\beta\in[H+(I^{\perp} \cap H^{\perp})]/H,\ p\beta=\alpha}\mathfrak{e}_{\beta}$. \label{combarrow}
\end{lem}

\begin{proof}
The Isomorphism Theorem from basic Group Theory identifies our subgroup of $B$ with the subgroup $(I^{\perp} \cap H^{\perp})/(H \cap I^{\perp})\subseteq\overline{A}$. As $I^{\perp} \cap H^{\perp}$ is perpendicular to both $I$ and $H \cap I^{\perp}$, our subgroup of $\overline{A}$ is contained in $\overline{I}^{\perp}$, and we define $p$ to be the composition of the isomorphism from above with the projection of $\overline{I}^{\perp}$ onto $\overline{B}=\overline{I}^{\perp}/\overline{I}$. The image of this map is the image, modulo the subgroup $I+(H \cap I^{\perp})$ defining $\overline{B}$, of the sum of $I^{\perp} \cap H^{\perp}$ from the numerator of the subgroup of $\overline{A}$ and the numerator $I+(H \cap I^{\perp})$ of $\overline{I}$, which together yield the numerator $(I^{\perp} \cap H^{\perp})+I$ of $\overline{A}$, proving that $p$ is indeed surjective. We now note that the kernel of the projection from $(I^{\perp} \cap H^{\perp})/(H \cap I^{\perp})$ onto $\overline{B}$ is its intersection with $\overline{I}$, which is thus $[(I \cap H^{\perp})+(H \cap I^{\perp})]/(H \cap I^{\perp})$ (because $H \cap I^{\perp} \subseteq I^{\perp} \cap H^{\perp}$), and transferring this kernel by the isomorphism from Group Theory indeed produces the asserted kernel.

It remains to establish the formula for $\downarrow_{H/(H \cap I^{\perp})}\uparrow_{\overline{I}}\mathfrak{e}_{\alpha}$ for any $\alpha\in\overline{B}$, since the surjectivity of $p$ and the fact that any $\beta$ in the subgroup of $B$ only shows up in the image of $\downarrow_{H/(H \cap I^{\perp})}\uparrow_{\overline{I}}\mathfrak{e}_{p\beta}$ immediately imply the injectivity of our combination of arrow operators. But indeed, we have $\uparrow_{\overline{I}}\mathfrak{e}_{\alpha}=\sum_{\delta\in\overline{I}^{\perp},\ \delta+\overline{I}=\alpha}\mathfrak{e}_{\delta}$, where $\overline{I}^{\perp}\subseteq\overline{A}$ consists of elements that must be perpendicular to the image of $I$ in $\overline{A}$ and is thus contained in $I^{\perp}/(H \cap I^{\perp})$. Applying $\downarrow_{H/(H \cap I^{\perp})}$ to this sum restricts the set of $\delta$'s associated with each such $\alpha$ to those in $H^{\perp}/(H \cap I^{\perp})$ and takes every such $\delta$ to its $B$-image $\delta+H/(H \cap I^{\perp})$. But these are thus precisely the set of $\delta$ from the subgroup of $(I^{\perp} \cap H^{\perp})/(H \cap I^{\perp})$ of $\overline{A}$ that are sent to $\alpha$ by our surjective map, and $\delta\mapsto\delta+H$ is precisely the inversion of the isomorphism from the previous paragraph, yielding the asserted formula. The fact that $p$ preserves the quadratic values follows either from the naturality of the maps in the previous paragraph, or from the fact that the arrow operators from Equation \eqref{arrowdef} preserve the action of $T$ via the Weil representations. This proves the lemma.
\end{proof}

We deduce the following consequence.
\begin{cor}
Let $A$, $H$, $I$, $\overline{A}$, $B$, and $\overline{B}$ be as above.
\begin{enumerate}[$(i)$]
\item If $I \cap H=\{0\}$ and $I \subseteq H^{\perp}$ then $B$ equals $\overline{A}$ and contains a isotropic subgroup $I_{B}$ that is isomorphic to $I$, such that $\overline{B}$ is the associated quotient and the operator from Lemma \ref{combarrow} is just $\uparrow_{I_{B}}$.
\item When $I \cap H^{\perp}=\{0\}$, there is an injection $\iota_{\overline{B}}:\overline{B} \to B$ of discriminant forms, such that the combination from Lemma \ref{combarrow} just takes $\mathfrak{e}_{\alpha}$ to $\mathfrak{e}_{\iota_{\overline{B}}\alpha}$.
\end{enumerate} \label{injop}
\end{cor}

\begin{proof}
If $I \subseteq H^{\perp}$ then $H \subseteq I^{\perp}$, so that the group $H \cap I^{\perp}$ used for defining $\overline{A}$ equals $H$ and thus $\overline{A}=B$. Then $B$ contains the isotropic group $I_{B}:=\overline{I}$ from above, which was seen to be isomorphic to $I/(I \cap H)$ and thus to $I$ under our assumption, and we saw that $\overline{B}$ was obtained from $\overline{A}=B$ using the isotropic subgroup $\overline{I}=I_{B}$. Since $\uparrow_{\overline{I}}$ clearly equals $\uparrow_{I_{B}}$ and the operator $\downarrow_{H/(H \cap I^{\perp})}=\downarrow_{H/H}$ it trivial, this establishes part $(i)$.

For part $(ii)$ note that our assumption makes the kernel of the map $p$ from Lemma \ref{combarrow}, which is isomorphic to $(I \cap H^{\perp})/(I \cap H)$ (as $H \subseteq H^{\perp}$), trivial, so that $p$ becomes an isomorphism in this case, which was seen to preserve the $\mathbb{Q}/\mathbb{Z}$-quadratic structure as well. By defining $\iota_{\overline{B}}$ to be the composition of the inverse isomorphism with the inclusion of the subgroup from Lemma \ref{combarrow} into $B$, the assertion follows directly from the formula given in that lemma. This proves the corollary.
\end{proof}
Under the condition from part $(ii)$ of Corollary \ref{injop}, we shall allow ourselves the slight abuse of notation by using $\iota_{\overline{B}}$ also for the map from $\mathbb{C}[\overline{B}]$ to $\mathbb{C}[B]$ taking $\mathfrak{e}_{\alpha}$ to $\mathfrak{e}_{\iota_{\overline{B}}\alpha}$, namely the operator from Lemma \ref{combarrow} reduces to just $\iota_{\overline{B}}$ in this case.

\smallskip

In our case the discriminant form $A$ is $D_{L} \oplus D$, and the group $H$ is a horizontal one as in Definition \ref{horizdef} (so that $B$ is $\Delta_{H}$ from above), thus yielding the isomorphic images $H_{L} \subseteq D_{L}$ and $H_{D} \subseteq D$, inside which $H_{L}^{\perp}$ and $H_{D}^{\perp}$ are the respective perpendicular group. We shall only consider isotropic subgroups $I$ that are contained in $D$, for which $I \cap H=\{0\}$ by horizontality, and then we have $I \cap H^{\perp}=I \cap H_{D}^{\perp}$ and $H \cap I^{\perp}$ maps isomorphically onto $H_{D} \cap I^{\perp}$ in restriction of the projection from $A$ onto $D$ (we write just $(H \cap I^{\perp})_{L}$ for its image in $D_{L}$). In this setting we have many additional groups, as follows.
\begin{defn}
Write $I_{D}$ for the subgroup $I \cap H_{D} \subseteq H_{D}$, let $I_{H}$ and $I_{L}$ be its isomorphic images in $H$ and in $H_{L}$ respectively, and set $\overline{H}:=(H \cap I^{\perp})/I_{H}$. Denote by $\Lambda$ the over-lattice of $L$ for which $\Lambda/L=I_{L}$, let $D_{\Lambda}$ be the discriminant group $I_{L}^{\perp}/I_{L}$ for $I_{L}^{\perp} \subseteq D_{L}$, and with $I^{\perp} \subseteq D$, denote $I^{\perp}/I$ by $\overline{D}$. \label{subgps}
\end{defn}

We shall need the following relation arising from Definition \ref{subgps}.
\begin{lem}
The group $\overline{H}$ can be identified with a horizontal isotropic subgroup of $D_{\Lambda}\oplus\overline{D}$. Its projections are $\overline{H}_{\Lambda}=(H \cap I^{\perp})_{L}/I_{L}$ and $\overline{H}_{\overline{D}}=(H_{D} \cap I^{\perp})/I_{D}$. \label{Hbarhoriz}
\end{lem}

\begin{proof}
Any element of $(H \cap I^{\perp})_{L}$ is $\gamma_{L}$ for some $\gamma \in H \cap I^{\perp}$, and we write an element of $I_{L}$ as $\delta_{L}$ with $\delta \in I_{H}$ via the isomorphism $H \to H_{L}$ and the notation from Definition \ref{subgps}. We write $\gamma_{D}$ and $\delta_{D}$ for their images in $H_{D} \cap I^{\perp}$ and $I_{D} \subseteq I$, meaning that $(\gamma_{D},\delta_{D})=0$, and since $(\gamma,\delta)=0$ as well (since $\gamma$ and $\delta$ are both in the isotropic subgroup $H$), and the latter pairing is the sum of the pairings of the projections (as such is the pairings of elements of $D_{L} \oplus D$), we deduce that $(\gamma_{L},\delta_{L})=0$ as well. This proves that $(H \cap I^{\perp})_{L} \subseteq I_{L}^{\perp}$, and thus has an image in $D_{\Lambda}$. Similarly, $H_{D} \cap I^{\perp}$ is contained in $I^{\perp}$ and thus has an image in $\overline{D}$, meaning that we have a natural map from $H \cap I^{\perp}$ to $D_{\Lambda}\oplus\overline{D}$. The fact that the kernel of the map to $D_{\Lambda}$ is the inverse image $I_{H}$ of $I_{L}$ (by which we divided to obtain $D_{\Lambda}$) in $H$ and the kernel of the map to $\overline{D}$ is the inverse image $I_{H}$ of $I \cap H_{D}=I_{D}$ in $H$ implies that this map factors through an injective map from $\overline{H}$, as defined in Definition \ref{subgps}, to $D_{\Lambda}\oplus\overline{D}$, and that the image of $\overline{H}$ is indeed horizontal as in Definition \ref{horizdef}. The expressions for the images of $\overline{H}$ also follow from this description. This proves the lemma.
\end{proof}
We shall thus henceforth identify $\overline{H}$ with its image in $D_{\Lambda}\oplus\overline{D}$, and with $\overline{H}^{\perp} \subseteq D_{\Lambda}\oplus\overline{D}$, and as above we denote the discriminant group $\overline{H}^{\perp}/\overline{H}$ resulting from Lemma \ref{Hbarhoriz} by $\overline{\Delta}_{\overline{H}}$.

\begin{rmk}
While $H_{L}$, and even its subgroup $(H \cap I^{\perp})_{L}$, contain $I_{L}$ (which allowed us to obtain the formula for $\overline{H}_{\Lambda}$ in Lemma \ref{Hbarhoriz}), it is not necessarily true for the complement $H_{L}^{\perp} \subseteq D_{L}$ showing up in the periodicity property for the index $(L,H)$. This group is contained in $I_{L}^{\perp}$, and its image in $D_{\Lambda}$ is $(H_{L}^{\perp}+I_{L})/I_{L} \cong H_{L}^{\perp}/(I_{L} \cap H_{L}^{\perp})$. The same argument from the proof of Lemma \ref{Hbarhoriz} shows that an element of $I_{L} \subseteq H_{L}$ is perpendicular to all of $H_{L}$ if and only if its image in $I_{D}=I \cap H_{D}$ is perpendicular to all of $H_{D}$, so that the size of this image, which is clearly contained in $\overline{H}_{\Lambda}^{\perp} \subseteq D_{\Lambda}$, is $|H_{L}^{\perp}|/|I \cap H_{D} \cap H_{D}^{\perp}|$, which we compare with the size of $\overline{H}_{\Lambda}^{\perp}$ in Remark \ref{dimreps} below. \label{sizeinperp}
\end{rmk}

\smallskip

We will need an analogue of Lemma \ref{arrowhoriz} that will respect the relation with $I$ as well. Recall that for any subgroup $I \subseteq D$, the elements of $D_{L} \oplus D$ that are perpendicular to $I$ are precisely those that lie in $D_{L} \oplus I^{\perp}$.
\begin{lem}
Given a horizontal subgroup $H \subseteq D_{L} \oplus D$ and a subgroup $I \subseteq D$, let $R_{I} \subseteq D_{L}$ be a set of representatives for $D_{L}/(H \cap I^{\perp})_{L}$.
\begin{enumerate}[$(i)$]
\item Any element of $D_{L} \oplus D$ such that its $H$-coset intersects $D_{L} \oplus I^{\perp}$ has a unique presentation as $\beta+\gamma$, in which $\gamma$ lies in $H$ and $\beta$ is with $\beta_{L} \in R_{I}$ and $\beta_{D} \in I^{\perp}$.
\item If $H$ is isotropic, then the map $\beta\mapsto\overline{\beta}:=\beta+H$ is a bijection from the set $H^{\perp}_{R_{I},I}$ consisting of $\beta\in(D_{L} \oplus I^{\perp}) \cap H^{\perp}$ with $\beta_{L} \in R_{I}$ onto its image in $\Delta_{H}$, and adding $\gamma \in H$ to $\beta$ does not affect the $H^{\perp}$ condition.
\item In case $I$ is also isotropic, the values of the pairings of $\beta_{L} \in R_{I}$ with $I_{L}$ depend only on its image in $D_{L}/(H \cap I^{\perp})_{L} \oplus I^{\perp}$ (and not on the choice of the set of representatives $R_{I}$). Moreover, if $\beta$ lies in $H^{\perp}_{R_{I},I}$ then $\beta_{L} \in R_{I}$ is perpendicular to $I_{L}$.
\item In the latter case, the set $\overline{R}:=\{\beta_{L}+I_{L}|\beta_{L} \in R_{I} \cap I_{L}^{\perp}\} \subseteq D_{\Lambda}$ is a set of representatives for $D_{\Lambda}/\overline{H}_{\Lambda}$ which is in bijection with $R_{I} \cap I_{L}^{\perp}$. The map taking $\beta \in H^{\perp}_{R_{I},I}$ to $\beta+I_{L}+I \in D_{\Lambda}\oplus\overline{D}$ is surjection from $H^{\perp}_{R_{I},I}$ onto the set $\overline{H}^{\perp}_{\overline{R}}$ associated via Lemma \ref{arrowhoriz} with $\overline{H} \subseteq D_{\Lambda}\oplus\overline{D}$ and $\overline{R}$, such that the map $p$ from Lemma \ref{combarrow} takes $\beta+H$ to $(\beta+I_{L}+I)+\overline{H}$.
\end{enumerate} \label{repswithI}
\end{lem}

\begin{proof}
For part $(i)$, the assumption on the intersection allows us to write our element as $\tilde{\beta}+\tilde{\gamma}$ with $\tilde{\gamma} \in H$ and $\tilde{\beta} \in D_{L} \oplus I^{\perp}$. Any other possibility of expressing our element as such a sum is obtained by changing $\tilde{\beta}$ and $\tilde{\gamma}$ by the same element, which must be in $H$ but also in $I^{\perp}$, hence part $(i)$ follows from $R_{I}$ being a set of representatives modulo the isomorphic image $(H \cap I^{\perp})_{L}$ of the intersection $H \cap I^{\perp}$ in $D_{L}$. The second assertion in part $(ii)$ is immediate (as for part $(ii)$ of Lemma \ref{arrowhoriz}), and the first one follows directly from part $(i)$ (with the image being in $\Delta_{H}$ by the $H^{\perp}$ condition). We now recall from the proof of Lemma \ref{Hbarhoriz} that when $I$ and $H$ are isotropic, the defining group $(H \cap I^{\perp})_{L}$ for $R_{I}$ is perpendicular to $I_{L}$ (i.e., we have $(H \cap I^{\perp})_{L} \subseteq I_{L}^{\perp}$), yielding the first assertion in part $(iii)$. The fact that for $\beta \in H^{\perp}_{R_{I},I}$ and $\gamma \in I_{H}$ we have $(\beta,\gamma)=0$ (since $\beta \in H^{\perp}$ and $\gamma \in H$) and also $(\beta_{D},\gamma_{D})=0$ (because $\beta_{D} \in I^{\perp}$ and $\gamma_{D} \in I_{D} \subseteq I$) proves the second assertion there as well. Intersecting with $I_{L}^{\perp}$ thus gives a set of representatives for the quotient $I_{L}^{\perp}/(H \cap I^{\perp})_{L}$, and as this quotient is isomorphic to $D_{\Lambda}/\overline{H}_{\Lambda}$ by Lemma \ref{Hbarhoriz}, the first assertion in part $(iv)$ follows. Noting that $H^{\perp}_{R_{I},I}$ is contained in $I^{\perp} \cap H^{\perp}$, this set can also serve as a set of representatives for $(I^{\perp} \cap H^{\perp})/(H \cap I^{\perp})$ from the proof of Lemma \ref{combarrow}, which are all perpendicular to $I_{L} \oplus I$ and, in the quotient modulo the latter group, to $\overline{H}$. Thus projecting modulo this direct sum, as in the construction of the map $p$ from Lemma \ref{combarrow}, yields the set of $\alpha \in D_{\Lambda}\oplus\overline{D}$ that are perpendicular to $\overline{H}$ and satisfy $\alpha_{\Lambda}\in\overline{R}$, which were seen in Lemma \ref{arrowhoriz} to represent $\overline{\Delta}_{\overline{H}}$, as desired. This proves the lemma.
\end{proof}

\smallskip

We can now state our extension of Proposition 3.4 of \cite{[Ze3]}, in which we recall that the Grassmannian associated with $L$ is the same as the one arising from any over-lattice $\Lambda$ of $L$ (a fact that clearly holds also when $L$ and $\Lambda$ are replaced by an $\mathcal{O}$-lattice $M$ and an appropriate over-lattice, with Hermitian structures inducing the orthogonal ones, and the Grassmannian is the Hermitian one from Equation \eqref{GrassJ}), and that $\overline{I}$ stands for the group $[I+(H \cap I^{\perp})]/(H \cap I^{\perp})$, which is isomorphic to $I$ since $I \cap H=\{0\}$.
\begin{prop}
Assume that the smooth function $F:\mathcal{H}\to\mathbb{C}[D]$ from Proposition \ref{pairper} is $\uparrow_{I}G$ for a smooth function $G:\mathcal{H}\to\mathbb{C}[\overline{D}]$. Then the associated function $\Phi_{L,v}^{F,H}=\downarrow_{H}(\Theta_{L} \otimes F)$ equals $\downarrow_{H/(H \cap I^{\perp})}\uparrow_{\overline{I}}\Phi_{\Lambda,v}^{G,\overline{H}}$, where $\Phi_{\Lambda,v}^{G,\overline{H}}$ is similarly defined to be $\downarrow_{\overline{H}}(\Theta_{\Lambda} \otimes G)$ and the combination of the arrow operators is the one from Lemma \ref{combarrow}. \label{FuparrowG}
\end{prop}
As the formulation and result of Proposition \ref{FuparrowG} may be a bit complicated in general, we present four simpler and clearer cases in Corollary \ref{spcases} below, from which the general case can, in fact, be constructed via Remark \ref{compoarrow} below.

\begin{proof}
Recall from the proof of Proposition \ref{pairper} that by taking any set of representatives $\tilde{R}$ inside $H^{\perp} \subseteq D_{L} \oplus D$ for $\Delta_{H}$, the parts of Lemma \ref{arrowhoriz} express $\Phi_{L,v}^{F,H}$ as $\sum_{\beta\in\tilde{R}}\big[\sum_{\gamma \in H}\theta_{\beta_{M}+\gamma_{L}}f_{\beta_{D}+\gamma_{D}}\big]\mathfrak{e}_{\overline{\beta}}$, with $\overline{\beta}=\beta+H$ for every $\beta\in\tilde{R}$ (Equation \eqref{downarrow} does this with the set $H^{\perp}_{R}$, but we will soon need a different one). The assumption that $F=\uparrow_{I}G$ means, via the definition in Equation \eqref{arrowdef}, that $f_{\beta_{D}+\gamma_{D}}$ equals $g_{\beta_{D}+\gamma_{D}+I}$ when $\beta_{D}+\gamma_{D} \in I^{\perp}$, and 0 otherwise. Therefore the coefficient $\mathfrak{e}_{\overline{\beta}}$ in $\Phi_{L,v}^{F,H}$ may be non-zero only if the coset $\beta+H$ intersects $D_{L} \oplus I^{\perp}$, so that by Lemma \ref{repswithI} we can take, for the corresponding subset of $\Delta_{H}$, the set $H^{\perp}_{R_{I},I}$ as the associated part of our $\tilde{R}$, where $R_{I} \subseteq D_{L}$ is any fixed set of representatives for $D_{L}/(H \cap I^{\perp})_{L}$. Moreover, as for $\beta \in H^{\perp}_{R_{I},I}$ and $\gamma \in H$ we have $\beta_{D}+\gamma_{D} \in I^{\perp}$ precisely when $\gamma \in H \cap I^{\perp}$, the formula that we get is
\begin{equation}
\Phi_{L,v}^{F,H}=\textstyle{\sum_{\beta \in H^{\perp}_{R_{I},I}}\big[\sum_{\gamma \in H \cap I^{\perp}}\theta_{\beta_{L}+\gamma_{L}}g_{\beta_{D}+\gamma_{D}+I}\big]\mathfrak{e}_{\beta+H}}. \label{Phiuparrow}
\end{equation}

Now, part $(iii)$ of Lemma \ref{repswithI} shows that if $\beta \in H^{\perp}_{R_{I},I}$ then $\beta_{L} \in I_{L}^{\perp}$, and recalling from the proof of Lemma \ref{Hbarhoriz} that $\gamma_{L}$ is also in $I_{L}^{\perp}$ for any $\gamma \in H \cap I^{\perp}$, we deduce that Equation \eqref{Phiuparrow} contains only theta functions with indices from $I_{L}^{\perp}$. Moreover, changing $\gamma$ by an element of $I_{H} \subseteq H \cap I^{\perp}$ leaves the index $\beta_{D}+\gamma_{D}+I$ of $g$ invariant, and since $\Theta_{\Lambda}$ equals $\downarrow_{I_{L}}\Theta_{L}$ (see, e.g., Lemma 1.5 of \cite{[Ze5]}), the sum of the resulting theta functions produces the corresponding component of $\Theta_{\Lambda}$. We can thus take the internal sum to be over $\overline{H}$, and part $(iv)$ of Lemma \ref{repswithI} shows that if we write $p(\beta+H)\in\overline{\Delta}_{\overline{H}}$, for $\beta \in H^{\perp}_{R_{I},I}$, as $\alpha+\overline{H}$ for $\alpha\in\overline{H}^{\perp}_{\overline{R}}$ (which decomposes as $\alpha_{\Lambda}+\alpha_{\overline{D}}$), then the multiplier of $\mathfrak{e}_{\overline{\beta}}$ in Equation \eqref{Phiuparrow} is $\sum_{\gamma\in\overline{H}}\theta_{\alpha_{\Lambda}+\gamma_{\Lambda}}g_{\alpha_{\overline{D}}+\gamma_{\overline{D}}}$. But as this is the coefficient of $\mathfrak{e}_{\alpha+\overline{H}}$ in Equation \eqref{downarrow} for $\Phi_{\Lambda,v}^{G,\overline{H}}$, and Lemma \ref{combarrow} combines with part $(iv)$ of Lemma \ref{repswithI} to show that $\downarrow_{H/(H \cap I^{\perp})}\uparrow_{\overline{I}}\mathfrak{e}_{\alpha+\overline{H}}$ equals $\sum_{\beta \in H^{\perp}_{R_{I},I},\ p(\beta+H)=\alpha+\overline{H}}\mathfrak{e}_{\beta+H}$, the sum from Equation \eqref{Phiuparrow} indeed equals $\downarrow_{H/(H \cap I^{\perp})}\uparrow_{\overline{I}}\Phi_{\Lambda,v}^{G,\overline{H}}$, as desired. This proves the proposition.
\end{proof}
Recall from Theorem \ref{VVorth} that $F$ is a modular form if and only if $\Phi_{L,v}^{F,H}$ is a Jacobi form as in Definition \ref{Jacdef}, and similarly $G$ is a modular form if and only if $\Phi_{\Lambda,v}^{G,\overline{H}}$ is a Jacobi form. This corresponds, via Proposition \ref{FuparrowG}, to the fact that the arrow operators take modular forms to modular forms and Jacobi forms to Jacobi forms.

\begin{rmk}
Note that the dimension of the space $\mathbb{C}[\Delta_{H}]$ associated with $\Phi_{L,v}^{F,H}$ is $|D_{L}|\cdot|D|/|H|^{2}$, while that of $\mathbb{C}[\overline{\Delta}_{\overline{H}}]$ and $\Phi_{\Lambda,v}^{G,\overline{H}}$ is $|D_{\Lambda}|\cdot|\overline{D}|/|\overline{H}|^{2}$. Moreover, we have $|D_{\Lambda}|=|D_{L}|/|I_{L}|^{2}$, $|\overline{D}|=|D|/|I|^{2}$, and $|\overline{H}|=|H \cap I^{\perp}|/|I_{H}|$, with $|I_{H}|=|I_{L}|$. As applying $\uparrow_{\overline{I}}$ to $\mathbb{C}[\overline{\Delta}_{\overline{H}}]$ multiplies the dimension of the space by $|\overline{I}|^{2}=|I|^{2}$, and $\downarrow_{H/(H \cap I^{\perp})}$ divides the dimension by $|H|^{2}/|H \cap I^{\perp}|^{2}$, the dimensions indeed match. In addition, the scalar-valued components of $\Phi_{L,v}^{F,H}$ satisfy the periodicity relations from Proposition 1.1 of \cite{[Ze3]} for every $\sigma$ and $\nu$ in the larger lattice $\Lambda$ (as so do those of $\Phi_{\Lambda,v}^{G,\overline{H}}$), and not only $L$, and the more general periodicity property holds for $\sigma$ and $\nu$ in $\overline{H}_{\Lambda}^{\perp} \subseteq D_{\Lambda}$. As $|\overline{H}_{\Lambda}|=|\overline{H}|$, its size is $|D_{\Lambda}|/|\overline{H}|$, which is the size $|D_{L}|/|H|$ of $H_{L}^{\perp} \subseteq D_{L}$ multiplied by the quotient $|H|/|H \cap I^{\perp}|\cdot|I_{H}|$, which also equals $|I|/|I \cap H^{\perp}|\cdot|I_{H}|$ by Lemma \ref{perfpair}. Comparing with the size of the image of $H_{L}^{\perp}$ in $\overline{H}_{\Lambda}^{\perp}$ given in Remark \ref{sizeinperp}, we deduce that this subgroup has index $|I|\cdot|I_{D} \cap H^{\perp}|/|I \cap H^{\perp}|\cdot|I_{H}|$ in $\overline{H}_{\Lambda}^{\perp}$. For more explicit descriptions in some special cases, see Remark \ref{propspcs} below. \label{dimreps}
\end{rmk}

\smallskip

We now illustrate Proposition \ref{FuparrowG} by presenting some interesting special cases, which will also be useful later. We recall that $I \subseteq D$ intersects the horizontal subgroup $H$ trivially, so in case $I$ and $H$ are perpendicular to one another, part $(i)$ of Corollary \ref{injop} produces a subgroup $I_{\Delta}\subseteq\Delta_{H}$ that is an isomorphic image of $I$. In case $I$ intersects $H^{\perp}$ trivially, we denote the embedding from part $(ii)$ of that corollary by $\iota_{\overline{\Delta}}:\overline{\Delta}_{\overline{H}}\to\Delta_{H}$, using again the same notation for the resulting map from $\mathbb{C}[\overline{\Delta}_{\overline{H}}]$ to $\mathbb{C}[\Delta_{H}]$.
\begin{cor}
Assume that $F=\uparrow_{I}G$ as in Proposition \ref{FuparrowG}.
\begin{enumerate}[$(i)$]
\item If $I \subseteq H_{D} \cap H_{D}^{\perp}$ then $\overline{H}$ equals $H/I_{H} \subseteq D_{\Lambda}\oplus\overline{D}$ with $I_{H} \cong I$, and we have $\Phi_{L,v}^{F,H}=\uparrow_{I_{\Delta}}\Phi_{\Lambda,v}^{G,H/I_{H}}$.
\item When $I \subseteq H_{D}^{\perp}$ but $I \cap H_{D}=\{0\}$, the group $H$ has an isomorphic image in $D_{L}\oplus\overline{D}$, and we get $\Phi_{L,v}^{F,H}=\uparrow_{I_{\Delta}}\Phi_{L,v}^{G,H}$.
\item In case $I \subseteq H_{D}$ and $I \cap H_{D}^{\perp}=\{0\}$, the subgroup $I_{H}$ is again isomorphic to $I$, and we obtain $\Phi_{L,v}^{F,H}=\iota_{\overline{\Delta}}\Phi_{\Lambda,v}^{G,\overline{H}}$.
\item Assuming that $I \cap H_{D}=I \cap H_{D}^{\perp}=\{0\}$, there is an isomorphic image of $H \cap I^{\perp}$ inside $D_{L}\oplus\overline{D}$, which produces the equality $\Phi_{L,v}^{F,H}=\iota_{\overline{\Delta}}\Phi_{L,v}^{G,H \cap I^{\perp}}$.
\end{enumerate} \label{spcases}
\end{cor}

\begin{proof}
In case $(i)$ we have $I=I_{D}$ and $H \subseteq I^{\perp}$, so that $H \cap I^{\perp}=H$ and $\overline{H}=H/I_{H}$. Part $(i)$ of Corollary \ref{injop} shows that $\Delta_{H}$ contains the subgroup $I_{\Delta} \cong I$, which produces the subquotient $\overline{\Delta}_{\overline{H}}$, and the $\uparrow$ operator is trivial, so that the formula from Proposition \ref{FuparrowG} takes the asserted form. For case $(ii)$, note that $I_{D}=I_{H}=I_{L}=\{0\}$ (so that in particular $\Lambda=L$) and again $H \cap I^{\perp}=H$, implying that $\overline{H}$ is an isomorphic copy of $H$ inside $D_{L}\oplus\overline{D}$. The result thus follows from Proposition \ref{FuparrowG} via part $(i)$ of Corollary \ref{injop} again. Moving to case $(iii)$, we again have $I=I_{D}$ but the index of $H \cap I^{\perp}$ in $H$ is $|I|$ (because $I$ injects into $\operatorname{Hom}(H,\mathbb{Q}/\mathbb{Z})$, and $\operatorname{Hom}(H \cap I^{\perp},\mathbb{Q}/\mathbb{Z})$ is the resulting quotient), as and part $(ii)$ of Corollary \ref{injop} shows that the combination of operators from Lemma \ref{combarrow} reduces to the injection $\iota_{\overline{\Delta}}$, we deduce the desired formula from Proposition \ref{FuparrowG} once again. Finally, the situation from case $(iv)$ implies that $I_{D}=I_{H}=I_{L}=\{0\}$ again and $H \cap I^{\perp}$, which again has index $|I|$ in $H$, has $\overline{H}\subseteq D_{L}\oplus\overline{D}$ as its isomorphic image. Using part $(ii)$ of Corollary \ref{injop} once more and Proposition \ref{FuparrowG} yet again produces the required result. This proves the corollary.
\end{proof}

\begin{rmk}
The cardinalities $|D_{\Lambda}|$, $|\overline{D}|$, and $|\overline{H}|^{2}$ are obtained, in case $(i)$ of Corollary \ref{spcases}, from $|D_{L}|$, $|D|$, and $|H|^{2}$ by division by $|I|^{2}$, and $\uparrow_{I_{\Delta}}$ multiplies the dimension by $|I|^{2}$, exemplifying the matching from Remark \ref{dimreps}. The group $\overline{H}_{\Lambda}^{\perp}$ with respect to which the vector-valued periodicity holds in this case is of size $|H_{L}^{\perp}|/|I|$, as is the image of $H_{L}^{\perp}$ given in Remark \ref{sizeinperp}, and indeed the index in the former remark is 1. In case $(ii)$ we have the same group $D_{L}$ as well as $H$ (up to isomorphism), again with $|\overline{D}|=|D|/|I|^{2}$ and the dimension extension by $|I|^{2}$ via $\uparrow_{I_{\Delta}}$. The periodicity properties are the same, as so are $L$ and $H$, and indeed both $\overline{H}_{\Lambda}^{\perp}$ and the image of $H_{L}^{\perp}$ in it via Remark \ref{sizeinperp} have the same size as $H_{L}^{\perp}$ itself (with the index from Remark \ref{dimreps} again being 1). Turning to case $(iii)$, here division by $|I|^{2}$ produces $|D_{\Lambda}|$, $|\overline{D}|$, and $|\overline{H}|$ from $|D_{L}|$, $|D|$, and $|H|$, the sizes of $\Delta_{H}$ and $\overline{\Delta}_{\overline{H}}$ are the same, and indeed $\iota_{\overline{\Delta}}$ is an isomorphism. Also here $\overline{H}_{\Lambda}^{\perp}$ and the subgroup arising from $H_{L}^{\perp}$ are of size $|H_{L}^{\perp}|$, and the index is 1 yet again. Finally, case $(iv)$ yields $D_{L}$ again, with $|\overline{D}|$, and $|\overline{H}|^{2}$ being $|D|$ and $|H|$ divided by $|I|^{2}$ respectively, and again the relation uses the isomorphism $\iota_{\overline{\Delta}}$. Here $H_{L}^{\perp}$ injects into $\overline{H}_{\Lambda}^{\perp}=(H \cap I^{\perp})_{L}$, but the latter contains the former with the index $|I|$ of $H \cap I^{\perp}$ in $H$, which is indeed the index showing up in Remark \ref{dimreps} in this case. \label{propspcs}
\end{rmk}
The size comparisons in cases $(iii)$ and $(iv)$ of Corollary \ref{spcases}, as given in Remark \ref{propspcs}, also indicate that in this case $\iota_{\overline{\Delta}}$ is an isomorphism between $\overline{\Delta}_{\overline{H}}$ and $\Delta_{H}$, with the subgroup $H_{L}^{\perp}$ being mapped onto $\overline{H}_{\Lambda}^{\perp}$ in the former case (and onto a subgroup of index $|I|$ in it in the latter). The scalar-valued case from Proposition 3.4 of \cite{[Ze3]}, which is obtained when $D=D_{\overline{L}}$, $\overline{D}=D_{\overline{\Lambda}}$, and $H$ and $\overline{H}$ are the associated diagonal subgroups, is contained in case $(iii)$ of Corollary \ref{spcases}, including the trivial groups $H_{L}^{\perp}$ and $\overline{H}_{\Lambda}^{\perp}$ which are matched by Remark \ref{propspcs} in this case.

\begin{rmk}
If $J$ is a subgroup of $I$, then the operator $\uparrow_{I}$, taking $G$ to $F$ in Proposition \ref{FuparrowG}, decomposes as the composition $\uparrow_{J}\uparrow_{I/J}$, through the intermediate subquotient between $D$ and $\overline{D}$. The formula from that proposition thus decomposes correspondingly, a fact that can be checked directly but is an immediate consequence of the proposition itself via the uniqueness of theta expansions observed above. This can be used in order to construct the general case in Proposition \ref{FuparrowG} from the special cases given in Corollary \ref{spcases} (with the details about $H$ being as described in Remark \ref{propspcs} in each stage). \label{compoarrow}
\end{rmk}

\smallskip

The function $\Phi_{L,v}^{F,H}$ that is obtained from $\Phi_{\Lambda,v}^{G,\overline{H}}$ via the arrow operators as in Proposition \ref{FuparrowG} inherits periodicity properties that are sometimes finer than those given in Proposition \ref{pairper} (or Equation \eqref{perorth} in the orthogonal case). Indeed, in cases $(i)$ and $(iii)$ of Corollary \ref{spcases} the scalar-valued components have the stronger periodicity property concerning $\sigma$ and $\nu$ from $\Lambda$ (rather than just $L$), while case $(iv)$ there introduces a periodicity property with respect to $(H \cap I^{\perp})_{L}^{\perp}$, which can be larger than $H_{L}^{\perp}$ (as Remark \ref{propspcs} shows). We now show that these enhanced periodicity properties characterize functions that arise in this way. This complements Theorem \ref{VVorth} (as well as Theorem \ref{VVmain} in the Hermitian case), but note that here we have the image of the connecting operator as well.

We begin with the scalar-valued periodicity.
\begin{thm}
Assume that the scalar-valued components of $\Phi_{L,v}^{F,H}$ satisfy the periodicity property for elements of $\Lambda$, where $\Lambda$ is an over-lattice of $L$ such that the isotropic subgroup $I_{L}:=\Lambda/L \subseteq D_{L}$ is contained in $H_{L}$, and let $I=I_{D}$ be the corresponding subgroup of $H_{D} \subseteq D$, with $I_{H} \subseteq H$ as usual. Then $F=\uparrow_{I}G$ for a unique function $G:\mathcal{H}\to\mathbb{C}[\overline{D}]$, where $\overline{D}=I^{\perp}/I$ for $I^{\perp} \subseteq D$, and $\Phi_{L,v}^{F,H}$ is obtained from $\Phi_{\Lambda,v}^{G,\overline{H}}$ as in Proposition \ref{FuparrowG}, where $\overline{H}=(H \cap I^{\perp})/I_{H}$ as in Definition \ref{subgps}. \label{peruparrow}
\end{thm}

\begin{proof}
As adding $\nu\in\Lambda \subseteq L^{*}$ to the argument of $\Theta_{L}$ multiplies the component associated with some element $\beta_{L}+\gamma_{L} \in D_{L}$ by $\mathbf{e}\big((\nu+L,\beta_{L}+\gamma_{L})\big)$, the fact that $\Phi_{L,v}^{F,H}$ remains invariant under this operation implies that in its expansion as in Equation \eqref{downarrow} (or its orthogonal analogue), only terms $\theta_{\beta_{L}+\gamma_{L}}$ in which $\beta_{L}+\gamma_{L} \in I_{L}^{\perp}$ may show up. The same argument from the proof of Lemma \ref{Hbarhoriz} shows that if $\beta+\gamma \in H^{\perp} \subseteq I_{H}^{\perp}$ and $\beta_{L}+\gamma_{L} \in I_{L}^{\perp}$ then $\beta_{D}+\gamma_{D} \in I^{\perp}$, meaning first that $f_{\beta_{D}+\gamma_{D}}=0$ unless $\beta_{D}+\gamma_{D} \in I^{\perp}$, and second that we may present $\Phi_{L,v}^{F,H}$ using the set of representatives $H^{\perp}_{R_{I},I}$ from Lemma \ref{repswithI}, and with $\gamma$ taken only from $H \cap I^{\perp}$.

For the periodicity property associated with $\sigma\in\Lambda$, recall that its effect, except for the usual exponent, simply takes each theta function $\theta_{\beta_{L}+\gamma_{L}}$ to $\theta_{\beta_{L}+\gamma_{L}+(\sigma+L)}$, and note that in our situation, $\sigma+L \in I_{L}$ is contained in $(H \cap I^{\perp})_{L}=H_{L} \cap I_{L}^{\perp}$. When comparing it with the expression resulting from the assumption that the scalar-valued components of $\Phi_{L,v}^{F,H}$ satisfy this periodicity property, we deduce that for every $\beta \in H^{\perp}_{R_{I},I}$, the equality \[\textstyle{\sum_{\gamma \in H \cap I^{\perp}}\theta_{\beta_{L}+\gamma_{L}+(\sigma+L)}f_{\beta_{D}+\gamma_{D}}=\sum_{\gamma \in H \cap I^{\perp}}\theta_{\beta_{L}+\gamma_{L}}f_{\beta_{D}+\gamma_{D}}}\] must hold. Replacing $\gamma$ on the left hand side by $\gamma-\sigma_{H}$, where $\sigma_{H} \in I_{H} \subseteq H \cap I^{\perp}$ is the element with $(\sigma_{H})_{L}=\sigma+L$, we obtain the same theta functions on both sides, and the uniqueness of theta expansions implies that changing the index of $f$ by $(\sigma_{H})_{D}$, namely by an arbitrary element of $I$, does not change the function. We may therefore write $f_{\beta_{D}+\gamma_{D}}$, for $\beta_{D}+\gamma_{D} \in I^{\perp}$, as $g_{\beta_{D}+\gamma_{D}+I}$ for some function $G:\mathcal{H}\to\mathbb{C}[\overline{D}]$ (which is clearly unique), and obtain that $F=\uparrow_{I}G$ as desired. Our expression for $\Phi_{L,v}^{F,H}$ also becomes the one from Equation \eqref{Phiuparrow} in the proof of Proposition \ref{FuparrowG} in this case, so that $\Phi_{L,v}^{F,H}$ is indeed the image of $\Phi_{\Lambda,v}^{G,\overline{H}}$, with this $\overline{H}$, under the corresponding operator. This proves the theorem.
\end{proof}
Note that the proof of Theorem \ref{peruparrow} did not require any assumption on whether $I$ is contained in $H_{D}^{\perp}$ intersects it trivially, so that it corresponds to any mixture of cases $(i)$ and $(iii)$ in Corollary \ref{spcases}, combined via Remark \ref{compoarrow} to yield the general combination (with indeed the groups $\overline{H}$ and $\overline{H}_{\Lambda}^{\perp}$ being the ones constructed from $H$ and $H_{L}^{\perp}$, as Remark \ref{propspcs} shows).

\smallskip

We now leave the lattice $L$ and the scalar-valued periodicity properties as in Definition \ref{orthVV}, but increase the validity of Equation \eqref{perorth} to a larger subgroup. We recall that this equation is based on $\Delta_{H}$ containing the isomorphic copy $\{\delta_{L}+H|\delta_{L} \in H_{L}^{\perp} \subseteq H^{\perp} \subseteq D_{L} \oplus D\}$ of $H_{L}^{\perp}$, in which $\delta_{L}$ can be recognized as the projection onto $D_{L}$ of a particular lift $\delta$ of its $\Delta_{H}$-image to $H^{\perp} \subseteq D_{L} \oplus D$. Our way of extending $H_{L}^{\perp}$ is as follows.
\begin{lem}
Assume that there is a subgroup $K_{L} \subseteq D_{L}$, which contains $H_{L}^{\perp}$, and which admits a lift into $H^{\perp} \subseteq D_{L} \oplus D$ whose composition with the projection onto $D_{L}$ is $\operatorname{Id}_{K_{L}}$, such that the composition with the projection onto $\Delta_{H}$ identifies $K_{L}$ with a subgroup of $\Delta_{H}$, with the $\mathbb{Q}/\mathbb{Z}$-quadratic structure, which contains the image of $H_{L}^{\perp}$. Then $K_{L}/H_{L}^{\perp}$ embeds into $D$ as an isotropic subgroup $I$ which intersects both $H_{D}$ and $H_{D}^{\perp}$ trivially, and $K_{L}$ can be reproduced from $I$ and $H$ as the group $(H \cap I^{\perp})_{L}^{\perp}$. \label{extHperp}
\end{lem}

\begin{proof}
The lift $\delta$ of $\delta_{L} \in K_{L}$ to $H^{\perp} \subseteq D_{L} \oplus D$ must be of the form $\delta_{L}+\delta_{D}$ with $\delta_{D}=\tilde{\kappa}\delta_{L}$ for a group homomorphism $\tilde{\kappa}:K_{L} \to D$. The fact that an element $\delta_{L} \in H_{L}^{\perp} \subseteq K_{L}$ lifts to $\delta_{L}+0$ for $0 \in D$ implies that $H_{L}^{\perp}\subseteq\ker\tilde{\kappa}$, so that $\tilde{\kappa}$ induces a map $\kappa:K_{L}/H_{L}^{\perp}$ into $D$. As the lift respects the $\mathbb{Q}/\mathbb{Z}$-quadratic structure, and $D_{L} \oplus D$ is an orthogonal direct summand, we deduce that $\kappa(K_{L}/H_{L}^{\perp})$, which will be our subgroup $I$, is isotropic.

Assume now that $\delta_{D}=\tilde{\kappa}\delta_{L}$ lies in $H_{D}^{\perp}$ for some $\delta_{L} \in K_{L}$, with $\delta$ being the lift $\delta_{L}+\delta_{D} \in D_{L} \oplus D$, and take any $\gamma \in H$. The fact that $(\delta,\gamma)=0$ for $\delta \in H^{\perp}$ and $\gamma \in H$, and that $(\delta_{D},\gamma_{D})=0$ since $\delta_{D} \in H_{D}^{\perp}$ and $\gamma_{D} \in H_{D}$, implies that $(\delta_{L},\gamma_{L})=0$, and as $\gamma \in H$ is arbitrary, and with it so is $\gamma_{L} \in H_{L}$, we deduce that $\delta_{L} \in H_{L}^{\perp}$. In particular, with $\delta_{D}=0$ this shows that $\ker\tilde{\kappa}=H_{L}^{\perp}$ and $\kappa$ is injective (so that $I \cong K_{L}/H_{L}^{\perp}$), but also that $I \cap H_{D}^{\perp}=\{0\}$.

Consider now the case that $\delta_{D}=\tilde{\kappa}\delta_{L}$ is in $H_{D}$, and write $\delta_{H}$ for the corresponding element of $H$, which we write as $\delta_{H,L}+\delta_{D}$ with $\delta_{H,L}$ the $H_{L}$-image of $\delta_{H}$. Then the $\Delta_{H}$-image of $\delta$, hence of $\delta_{L} \in K_{L}$, coincides with that of $\delta-\delta_{H}$, which is the sum of $\delta_{L}-\delta_{H,L} \in D_{L}$ and $0 \in D$ and lies in $H^{\perp}$, so that $\delta_{L}-\delta_{H,L} \in H_{L}^{\perp}$ and is taken to the same element of $\Delta_{H}$. As the map from $K_{L}$ to $\Delta_{H}$ is injective by hypothesis, we obtain that $\delta_{L}=\delta_{L}-\delta_{H,L} \in H_{L}^{\perp}$, and again $\delta_{D}$ vanishes, which proves that $I \cap H_{D}=\{0\}$ as well. Altogether, $\kappa$ embeds $K_{L}/H_{L}^{\perp}$ into $D$ as the isotropic subgroup $I$, with the desired properties.

Now, Lemma \ref{perfpair} combines with the fact that $I \cap H^{\perp}=\{0\}$ to show that $H$ contains $H \cap I^{\perp}$ with index $|I|$, hence the same relation holds between their isomorphs $H_{L}$ and $(H \cap I^{\perp})_{L}$, which implies that $(H \cap I^{\perp})_{L}^{\perp}$
contains $H_{L}^{\perp}$ with the same index inside the discriminant for $D_{L}$. As $K_{L}$ also contains $H_{L}^{\perp}$ with the same index (indeed, $I$ is isomorphic to the quotient), it has the same size as $(H \cap I^{\perp})_{L}^{\perp}$, and it suffices to prove that $K_{L}$ is contained in the latter subgroup. But indeed, for any $\delta_{L} \in K_{L}$, embedded as $\delta_{L}+\tilde{\kappa}\delta_{L} \in H^{\perp} \subseteq D_{L} \oplus D$, its pairing with $\gamma \in H$ was seen to vanish, and if $\gamma$ lies in $H \cap I^{\perp}$ then $(\gamma_{D},\tilde{\kappa}\delta_{L})=0$ as well, yielding the orthogonality of $\delta_{L}$ to any $\gamma_{L}$ for $\gamma$ in that intersection. This completes the proof of the lemma.
\end{proof}

Recalling from the proof of Lemma \ref{indchit} that the operators on $\mathbb{C}[\Delta_{H}]$ that arise from $H_{M}^{\perp}$ (or $H_{L}^{\perp}$ in our setting) are obtained via Definition \ref{opersDM} through the embedding of this group into $\Delta_{H}$ considered also in Lemma \ref{extHperp}, we can thus state and prove the following result.
\begin{thm}
Let $K_{L} \subseteq D_{L}$ be as in Lemma \ref{extHperp}, with $I \subseteq D$ the corresponding isotropic subgroup, and set $\overline{D}:=I^{\perp}/I$ as usual. Assume that the function $\Phi_{L,v}^{F,H}$ associated with $F:\mathcal{H}\to\mathbb{C}[\Delta_{H}]$ via Proposition \ref{pairper} satisfies the periodicity from Equation \eqref{perorth}, in which $\sigma$ and $\nu$ from $L^{*}$ can have arbitrary $D_{L}$-images in $K_{L}$, and where $t_{\sigma}$ and $\chi_{\nu}$ are defined on $\Phi_{L,v}^{F,H}$ via this embedding of $K_{L}$ into $\Delta_{H}$. Then there is a unique function $G:\mathcal{H}\to\mathbb{C}[\overline{D}]$ such that $F=\uparrow_{I}G$, and $\Phi_{L,v}^{F,H}$ is $\iota_{\overline{\Delta}}\Phi_{L,v}^{G,H \cap I^{\perp}}$ as in case $(iv)$ of Corollary \ref{spcases}. \label{largerHperp}
\end{thm}

\begin{proof}
We begin by considering an element $\nu \in L^{*}$ with $\nu+L \in K_{L}$, and comparing the resulting periodicity relations for the theta components $\theta_{\beta_{L}+\gamma_{L}}$ and using $\chi_{\nu}$ for this $\nu$ by assumption. We recall that for the simple case of $\nu+L \in H_{L}^{\perp}$, the action of $t_{\nu}$ on the basic vector $\mathfrak{e}_{\beta}=\mathfrak{e}_{\beta_{L}}\otimes\mathfrak{e}_{\beta_{D}}$ of $\mathbb{C}[D_{L} \oplus D]$ multiplies it by $\mathbf{e}\big((\nu+L,\beta_{L})\big)$, and the pairing is that $\nu+L \in H_{L}^{\perp} \subseteq D_{L}$ as contained in $D_{L} \oplus D$, with $\beta=\beta_{L}+\beta_{D}$. Our assumption means that if $\nu+L$ is a more general element of $K_{L}$, which is embedded in $H^{\perp} \subseteq D_{L} \oplus D$, then the action of $\chi_{\nu}$ multiplies this vector $\mathfrak{e}_{\beta}$ by $\mathbf{e}\big((\nu+L,\beta)\big)$ (which for $\beta \in H^{\perp}$ can be evaluated using the image of $\nu+L \in K_{L}$ in $\Delta_{H}$), with the pairing being the sum of $(\nu+L,\beta_{L})$ and $\big(\tilde{\kappa}(\nu+L),\beta_{D}\big)$, with $\tilde{\kappa}:K_{L} \to D$ the map defined in the proof of Lemma \ref{extHperp}.

So take now $\beta \in H^{\perp} \subseteq D_{L} \oplus D$, and consider the vector $\mathfrak{e}_{\beta+H}\in\mathbb{C}[\Delta_{H}]$ and the coefficient $\sum_{\gamma \in H}\theta_{\beta_{L}+\gamma_{L}}f_{\beta_{D}+\gamma_{D}}$ multiplying it in $\Phi_{L,v}^{F,H}$. The translation of the argument by our $\nu$ multiplies the theta function by the exponent of $(\nu+L,\beta_{L}+\gamma_{L})$, and when we compare it with the action of $\chi_{\nu}$ on this vector, as evaluated above, we deduce that only summands associated with $\gamma \in H$ for which $(\nu+L,\gamma_{L})$ equals $\big(\tilde{\kappa}(\nu+L),\beta_{D}\big)$ can occur, where $\tilde{\kappa}(\nu+L)$ or equivalently $\kappa[(\nu+L)+H_{L}^{\perp}]$ for $(\nu+L)+H_{L}^{\perp} \in K_{L}/H_{L}^{\perp}$, runs over $I$ as $\nu$ varies over the elements of $L^{*}$ with $D_{L}$-images in $K_{L}$.

We now observe that as $I \cap H^{\perp}=\{0\}$ by Lemma \ref{extHperp}, the proof of Lemma \ref{perfpair} produces a perfect pairing from $I\times[H/(H \cap I^{\perp})]$ to $\mathbb{Q}/\mathbb{Z}$, so that for every such $\beta$ we have an element $\alpha \in H$ such that $(\delta_{D},\alpha_{D})=(\delta_{D},\beta_{D})$ for every $\delta_{D} \in I$. We can thus represent $\beta+H\in\Delta_{H}$ by $\beta+\delta \in D_{L} \oplus I^{\perp}$, noting that for an element of $D_{L} \oplus I^{\perp}$ the set of $\gamma$ from the previous paragraph becomes the subgroup $H \cap I^{\perp}$ (via the argument from the proof of Lemma \ref{Hbarhoriz} once again). We thus take the set of representatives $H^{\perp}_{R_{I},I}$ from Lemma \ref{repswithI}, and deduce that for each $\beta$ in that set, the coefficient in front of $\mathfrak{e}_{\beta+H}$ in $\Phi_{L,v}^{F,H}$ is $\sum_{\gamma \in H \cap I^{\perp}}\theta_{\beta_{L}+\gamma_{L}}f_{\beta_{D}+\gamma_{D}}$, recalling that all the indices $\beta_{D}+\gamma_{D}$ of $f$ that show up here lie in $I^{\perp}$.

We now turn to the periodicity property arising from $\sigma \in L^{*}$ with $\sigma+L \in K_{L}$ (which is indeed the group $(H \cap I^{\perp})_{L}^{\perp}$ by Lemma \ref{extHperp}), and to the action of $t_{\sigma}$ for such $\sigma$ as defined here. The argument from the proofs of Proposition \ref{pairper} and Theorem \ref{VVmain} combines with the set of representatives in our setting to yield the equality \[t_{\sigma}\sum_{\substack{\tilde{\beta} \in H^{\perp}_{R_{I},I} \\ \tilde{\gamma} \in H \cap I^{\perp}}}\theta_{\tilde{\beta}_{L}+\tilde{\gamma}_{L}}f_{\tilde{\beta}_{D}+\tilde{\gamma}_{D}}\mathfrak{e}_{\tilde{\beta}_{L}+\tilde{\beta}_{D}+H}=\sum_{\substack{\beta \in H^{\perp}_{R_{I},I} \\ \gamma \in H \cap I^{\perp}}}\theta_{\beta_{L}+\gamma_{L}+(\sigma+M)}f_{\beta_{D}+\gamma_{D}}\mathfrak{e}_{\beta_{L}+\beta_{D}+H}\] and we employ once again the notation from Lemma \ref{addRsubgrp}, for the set of representatives $R_{I}$ for $D_{L}/(H \cap I^{\perp})_{L}$ (the perpendicularity condition from part $(iii)$ of Lemma \ref{repswithI} it trivial, since $I_{L}=\{0\}$ in our case).

We make again the substitution $\tilde{\beta}_{L}=r_{\sigma}(\beta_{L})$ and $\tilde{\gamma}:=\gamma+h_{\sigma}(\beta_{L})$ from that lemma with these parameters, but note that $t_{\sigma}$ subtracts the image of $\sigma+L \in K_{L}$ from $\beta$, and that image is the sum of $\sigma+L \in D_{L}$ and $\tilde{\kappa}(\sigma+L) \in I \subseteq I^{\perp} \subseteq D$. We subtract $h_{-\sigma}(\tilde{\beta}_{L}) \in H \cap I^{\perp} \subseteq H$ to remain in the same coset, and the representing element of $H^{\perp}$ is the sum of $h_{-\sigma}(\tilde{\beta}_{L}) \in R_{I} \subseteq D_{L}$ and the element $\tilde{\beta}_{D}-h_{-\sigma}(\tilde{\beta}_{L})_{D}-\tilde{\kappa}(\sigma+L)$ from $I^{\perp} \subseteq D$. Part $(ii)$ of Lemma \ref{addRsubgrp} implies that the first expression is $\beta_{L}$ back again, and we set $\tilde{\beta}_{D}$ to be $\beta_{D}+h_{-\sigma}(\tilde{\beta}_{L})_{D}+\tilde{\kappa}(\sigma+L)$, or equivalently $\beta_{D}-h_{\sigma}(\beta_{L})_{D}+\tilde{\kappa}(\sigma+L)$ (by the same lemma), so that we get the same vector $\mathfrak{e}_{\beta_{L}+\beta_{D}+H}$ on both sides. Using the uniqueness of theta expansions once again, and the fact that $\tilde{\beta}_{D}+\tilde{\gamma}_{D}$ equals $\beta_{D}+\gamma_{D}$ plus the arbitrary element $\tilde{\kappa}(\sigma+L) \in I$, we deduce that $f_{\beta_{D}+\gamma_{D}}$ remains invariant under changing the index from $D$, which lies in $I^{\perp}$, by an element of $I$.

We can thus write $f_{\beta_{D}+\gamma_{D}}$ for an index from $I^{\perp}$ as $g_{\beta_{D}+\gamma_{D}+I}$ for function $G:\mathcal{H}\to\mathbb{C}[\overline{D}]$, that is unique as usual, and obtain that $F=\uparrow_{I}G$. We also observe that if $\beta=\beta_{L}+\beta_{D} \in H^{\perp}_{R_{I},I}$ and $\alpha+(H \cap I^{\perp})=p(\beta+H)$ for $\alpha=\alpha_{L}+\alpha_{\overline{D}} \in H^{\perp}_{R_{I}} \subseteq D_{L}\oplus\overline{D}$ then $\alpha_{L}=\beta_{L} \in R_{I}$ and $\alpha_{\overline{D}}=\beta_{D}+I$, so that the theta function is $\theta_{\alpha_{L}+\gamma_{L}}$ and it multiplies $g_{\alpha_{\overline{D}}+\gamma_{\overline{D}}}$. Moreover, $\beta_{D}$ is the unique element of $I^{\perp}$ with $\overline{D}$-image $\alpha_{\overline{D}}$ and such that $\beta_{L}+\beta_{D}$, with $\beta_{L}=\alpha_{L}$, lies in $H^{\perp}$, which defines $\beta+H$ for this $\beta$ to be $\iota_{\overline{\Delta}}\big(\alpha+(H \cap I^{\perp})\big)$. Altogether we deduce that $\Phi_{L,v}^{F,H}$ is indeed $\iota_{\overline{\Delta}}\Phi_{L,v}^{G,H \cap I^{\perp}}$, like Proposition \ref{FuparrowG} and case $(iv)$ of Corollary \ref{spcases} imply. This completes the proof of the theorem.
\end{proof}
Like with Proposition \ref{FuparrowG}, Theorems \ref{peruparrow} and \ref{largerHperp} only deal with the periodicity properties, and the uniqueness in the constructions imply that $\Phi_{L,v}^{F,H}$ is a Jacobi form of some weight with respect to some group $\Gamma$ if and only if the corresponding functions arising from some $G$ have the same property (with the appropriate representations), as indeed $F$ is modular with the associated parameters if and only if $G$ is. However, as we saw, the modularity property plays no role in these results. We remark again that when $L$ is an $\mathcal{O}$-lattice $M$, with the quadratic structure arising from the Hermitian one on $M$, then all of the results of this section extend to this setting, provided that all the groups involved in the constructions are preserved under the action of $\mathcal{O}$.

\begin{rmk}
In some sense, expressing a function as $\Phi_{L,v}^{F,H}$, or as $\Phi_{M,v}^{F,H}$ in the Hermitian setting, implies that it satisfies at least the periodicity properties arising from $L$ for the scalar-valued components and from the inverse image of $H_{L}^{\perp}$ inside $D_{L}$ with the action. Theorems \ref{peruparrow} and \ref{largerHperp} can be seen as saying that the precise periodicity properties (the scalar-valued and the vector-valued ones respectively) that are satisfied by $\Phi_{L,v}^{F,H}$ (or $\Phi_{M,v}^{F,H}$) detect what are the most natural lattice $\Lambda \subseteq L$ and group $\overline{H}_{\Lambda}^{\perp}$ (containing the image of $H_{L}^{\perp}$) that really define this function. This will be important for defining vector-valued Jacobi forms with matrix index below. Note that case $(ii)$ of Corollary \ref{spcases} does not affect the periodicity properties at all, and simply presents an interplay between the $\uparrow$ operators on $G$ and on $\Phi$. \label{maxlat}
\end{rmk}

\section{Holomorphic Hermitian Jacobi Forms and Jacobi Forms of Matrix Index \label{HolJForms}}

The main motivation and most applications for considering Jacobi forms with lattice index are clearer and more tangible when the lattice is positive definite. Since (as far as I know) no reference gathers their properties in this particularly important special case, we write the simpler expressions and properties that are obtained from our general ones in this case.

So let $M$ be an $\mathcal{O}$-lattice, which we now assume to be positive definite of some rank $b_{+}$. Then the sets from Equations \eqref{Grassdef} and \eqref{GrassJ} consist of one point, so that the variable $v$ can be omitted. In addition, the projections onto the negative definite space $v_{-}=\{0\}$ vanish (as does the Laplacian operator $\Delta_{v_{-}}^{h}$), and projecting onto the positive definite space $v_{+}=M_{\mathbb{R}}$ leaves every vector invariant (with $\Delta_{v_{+}}^{h}$ being just $\Delta_{M_{\mathbb{R}}}^{h}$). Therefore the definition in Equation \eqref{JacTheta} and the properties from Propositions \ref{perTheta} and \ref{difeq} and Theorem \ref{modTheta} combine to the following statement, in which we recall again the operators from Definition \ref{opersDM}.
\begin{thm}
The theta function \[\Theta_{M}(\tau,\zeta,\omega):=\textstyle{\sum_{\lambda \in M^{*}}\mathbf{e}\big(\tau|\lambda}|^{2}+\langle\zeta,\lambda\rangle+\langle\lambda,\omega\rangle\big)\mathfrak{e}_{\lambda+M}\] is a holomorphic function from $\mathcal{H} \times M_{\mathbb{C}} \times M_{\overline{\mathbb{C}}}$ to $\mathbb{C}[D_{M}]$ that is a solution to the holomorphic heat equation and satisfies the periodicity equation \[\Theta_{M}(\tau,\zeta+\tau\sigma_{\mathbb{C}}+\nu_{\mathbb{C}},\omega+\tau\sigma_{\overline{\mathbb{C}}}+\nu_{\overline{\mathbb{C}}})=\mathbf{e}\big(-\tau|\sigma|^{2}-\langle\zeta,\sigma\rangle-\langle\sigma,\omega\rangle\big) \chi_{\nu}t_{\sigma}\Theta_{M}(\tau,\zeta,\omega)\] for any $\sigma$ and $\nu$ from $M^{*}$, as well as, for every $A\in\operatorname{U}(1,1)(\mathbb{Z})$, the equality
\[\Theta_{M}\Big(A\tau,\tfrac{\zeta}{j(A,\tau)},\tfrac{\omega}{j^{c}(A,\tau)}\Big)=j(A,\tau)^{b_{+}}\mathbf{e}\Big(\tfrac{j_{A}'\langle\zeta,\omega\rangle}{j(A,\tau)}\Big)\rho_{M}(A)\Theta_{M}(\tau,\zeta,\omega).\] \label{holTheta}
\end{thm}

\smallskip

In this setting, from Definition \ref{Jacdef} we deduce that a \emph{Hermitian Jacobi form of weight $k\in\mathbb{Z}$ and index $M$ with respect to $\Gamma\subseteq\operatorname{U}(1,1)(\mathbb{Z})$} (of finite index) is a function $\Phi:\mathcal{H} \times M_{\mathbb{C}} \times M_{\overline{\mathbb{C}}}\to\mathbb{C}$, holomorphic in all the variables except perhaps the first one (in which it is smooth), for which the equality \[\Phi(\tau,\zeta+\tau\sigma_{\mathbb{C}}+\nu_{\mathbb{C}},\omega+\tau\sigma_{\overline{\mathbb{C}}}+\nu_{\overline{\mathbb{C}}})=\mathbf{e}\big(-\tau|\sigma|^{2}-\langle\zeta,\sigma\rangle-\langle\sigma,\omega\rangle\big) \Phi(\tau,\zeta,\omega)\] holds for every two elements $\sigma$ and $\nu$ from $M$, as does the equality
\[\Phi\Big(A\tau,\tfrac{\zeta}{j(A,\tau)},\tfrac{\omega}{j^{c}(A,\tau)}\Big)=j(A,\tau)^{k}\mathbf{e}\Big(\tfrac{j_{A}'\langle\zeta,\omega\rangle}{j(A,\tau)}\Big)\Phi(\tau,\zeta,\omega)\] for every $A\in\Gamma$, plus some growth conditions. Via Remarks \ref{JacFour} and \ref{holFour} we see that every such Jacobi form has a Fourier expansion of the sort
\begin{equation}
\Phi(\tau,\zeta,\omega)=\textstyle{\sum_{m\in\frac{1}{N}\mathbb{Z}}\sum_{\lambda \in M^{*}}\tilde{c}_{m,\lambda}(y)\mathbf{e}\big(m\tau+\langle\zeta,\lambda\rangle+\langle\lambda,\omega\rangle\big)} \label{Fourposdef}
\end{equation}
(for some $N$ depending on $\Gamma$), with the coefficients $\tilde{c}_{m,\lambda}(y)$ being constant if and only of $\Phi$ is holomorphic also in $\tau\in\mathcal{H}$, a situation where the growth condition amounts to $\tilde{c}_{m,\lambda}=0$ unless $m\gg-\infty$ (such functions are called \emph{weakly holomorphic Hermitian Jacobi forms}). The Fourier coefficients satisfy the relation $\tilde{c}_{m,\lambda}=\tilde{c}_{m+\langle\sigma,\lambda\rangle+\langle\lambda,\sigma\rangle+|\sigma|^{2},\lambda+\sigma}$ for every $\sigma \in M$ (regardless of holomorphicity), and we say that a weakly holomorphic Hermitian Jacobi form is a \emph{weak Hermitian Jacobi form}, or a \emph{holomorphic Hermitian Jacobi form}, or a \emph{Hermitian Jacobi cusp form}, if a non-zero Fourier coefficient $\tilde{c}_{m,\lambda}\neq0$ can occur only if $m\geq0$ (resp. $m\geq|\lambda|^{2}$, resp. $m>|\lambda|^{2}$). In the coordinates used when the index is a matrix, the Fourier expansion takes the form presented in Remark \ref{FourS} below.

Theorem \ref{main} and Proposition \ref{pshol}, with some additional details involving holomorphicity, now read as follows.
\begin{thm}
The Hermitian Jacobi forms of weight $k\in\mathbb{Z}$ and index $M$ with respect to $\Gamma$ are precisely the functions on $\mathcal{H} \times M_{\mathbb{C}} \times M_{\overline{\mathbb{C}}}\to\mathbb{C}$ that are obtained as pairings $\Phi_{M}^{F}(\tau,\zeta,\omega):=[\Theta_{M}(\tau,\zeta,\omega),F(\tau)]_{M}$ of $\Theta_{M}$ with modular forms $F$ of weight $k-b_{+}$ and representation $\rho_{\overline{M}}$ with respect to $\Gamma$, as in Equation \eqref{modeq}. The Hermitian Jacobi form $\Phi_{M}^{F}$ and the associated modular form $F$ are weakly holomorphic, holomorphic, and cuspidal together. \label{holJac}
\end{thm}
As with other types of Jacobi forms, the weakly holomorphic modular form $F$ that is associated with a weak Hermitian Jacobi form $\Phi:=\Phi_{M}^{F}$ via Theorem \ref{holJac} need not be holomorphic, but rather its components multiplying basis vectors $\mathfrak{e}_{\gamma}^{*}$ for anisotropic $\gamma \in D_{M}$ are allowed to have poles of very small order at the cusps.

\smallskip

We shall also be needing the vector-valued objects, for which we assume that $D$, $H$, and $\Delta_{H}$ are given as in Definition \ref{VVdef}, and $k$ and $\Gamma$ are as above. Then a \emph{Hermitian Jacobi form of weight $k$, index $(M,H)$, and representation $\rho_{\Delta_{H}}$ with respect to $\Gamma$} is defined to be a smooth function $\Phi:\mathcal{H} \times M_{\mathbb{C}} \times M_{\overline{\mathbb{C}}}\to\mathbb{C}[\Delta_{H}]$, which is holomorphic when the first variable is fixed, and which satisfies the equality \[\Phi(\tau,\zeta+\tau\sigma_{\mathbb{C}}+\nu_{\mathbb{C}},\omega+\tau\sigma_{\overline{\mathbb{C}}}+\nu_{\overline{\mathbb{C}}})=\mathbf{e}\big(-\tau|\sigma|^{2}-\langle\zeta,\sigma\rangle-\langle\sigma,\omega\rangle\big) \chi_{\nu}t_{\sigma}\Phi(\tau,\zeta,\omega)\] whenever $\sigma$ and $\nu$ are $M^{*}$ and $\sigma+M$ and $\nu+M$ lie in $H_{M}^{\perp}$, as well as \[\Phi\Big(A\tau,\tfrac{\zeta}{j(A,\tau)},\tfrac{\omega}{j^{c}(A,\tau)}\Big)=j(A,\tau)^{k}\mathbf{e}\Big(\tfrac{j_{A}'\langle\zeta,\omega\rangle}{j(A,\tau)}\Big)\rho_{\Delta_{H}}(A)\Phi(\tau,\zeta,\omega)\] whenever $A\in\Gamma$, and the usual growth conditions on the resulting Fourier coefficients as functions on $\mathcal{H}$. When $\Phi$ is also holomorphic in $\tau$, it is a weakly holomorphic Jacobi form (with these parameters), and one can extend the definitions of weak, holomorphic, and cuspidal Jacobi forms (or Jacobi cusp forms) to the vector-valued setting in an analogous manner, using the resulting Fourier expansions of the scalar-valued components and assumptions required for non-vanishing.

In this setting, Theorem \ref{VVmain} transforms into the following result.
\begin{thm}
Vector-valued Hermitian Jacobi forms $\Phi$ of weight $k$, index $(M,H)$, and representation $\rho_{\Delta_{H}}$ with respect to $\Gamma$ are in a bijective correspondence with modular forms $F$ of weight $k-b_{+}$ and representation $\rho_{D}$ with respect to $\Gamma$, with the map $F\mapsto\Phi$ given by $\Phi(\tau,\zeta,\omega):=\downarrow_{H}\big(\Theta_{M}(\tau,\zeta,\omega) \otimes F(\tau)\big)$. This bijection preserves weak holomorphicity, holomorphicity, and cuspidality. \label{VVhol}
\end{thm}
As before, taking $D=D_{\overline{M}}$ and $H$ the diagonal subgroup in Theorem \ref{VVhol} reproduces Theorem \ref{holJac}.

\smallskip

Since \cite{[BRZ]} will also be using the holomorphic version of Theorem \ref{VVorth}, we give it here as well. Let now $L$ be an positive definite even lattice of rank $b_{+}$ over $\mathbb{Z}$, take $D$, $H$ with the projections $H_{L}$ and $H_{D}$, and $\Delta_{H}$ as above. The theta function \[\Theta_{L}(\tau,\zeta):=\textstyle{\sum_{\gamma \in D_{L}}\big[\sum_{\lambda \in L+\gamma}\mathbf{e}\big(\tau\frac{\lambda^{2}}{2}+(\lambda,\zeta)\big)\big]\mathfrak{e}_{\gamma}}\] (with no Grassmannian variable) then satisfies the usual, holomorphic modularity of weight $\frac{b_{+}}{2}$ and representation $\rho_{L}$ with respect to $\operatorname{Mp}_{2}(\mathbb{Z})$ and periodicity with respect to $L^{*} \times L^{*}$ (with the operators $\chi_{\nu}$ and $t_{\sigma}$ as usual), and take a weight $k\in\frac{1}{2}\mathbb{Z}$ and a subgroup $\Gamma\subseteq\operatorname{Mp}_{2}(\mathbb{Z})$. Then, via Definition \ref{orthVV}, a \emph{Jacobi form of weight $k$, index $(L,H)$, and representation $\rho_{\Delta_{H}}$ with respect to $\Gamma$} is a smooth function $\Phi:\mathcal{H} \times L_{\mathbb{C}}\to\mathbb{C}$ which is holomorphic in $\zeta \in L_{\mathbb{C}}$, satisfies \[\Phi(\tau,\zeta+\tau\sigma+\nu)=\mathbf{e}\big(-\tau\frac{\sigma^{2}}{2}-(\sigma,\zeta)\big)\chi_{\nu}t_{\sigma}\Phi(\tau,\zeta)\] for any $\sigma$ and $\nu$ in $L^{*}$ with $D_{L}$-images in $H_{L}^{\perp}$ and \[\Phi\big(A\tau,\tfrac{\zeta}{j(A,\tau)}\big)=j(A,\tau)^{k}\mathbf{e}\Big(\tfrac{j_{A}'\zeta^{2}}{2j(A,\tau)}\Big)\rho_{\Delta_{H}}(A)\Phi(\tau,\zeta)\] wherever $A\in\Gamma$, and an appropriate growth condition. Weak holomorphicity, and holomorphicity are defined as usual, as are weak Jacobi forms and Jacobi cusp forms with these parameters.

The associated special case of Theorem \ref{VVorth} can now be given.
\begin{thm}
The Jacobi forms $\Phi$ of weight $k$, index $(L,H)$, and representation $\rho_{\Delta_{H}}$ with respect to $\Gamma$ are in one-to-one correspondence with modular forms $F$ of weight $k-\frac{b_{+}}{2}$ and representation $\rho_{D}$ with respect to $\Gamma$, which to $F$ associates the Jacobi form $\Phi:(\tau,\zeta)\mapsto\downarrow_{H}\big(\Theta_{L}(\tau,\zeta) \otimes F(\tau)\big)$. In this correspondence, $\Phi$ is weakly holomorphic if and only if $F$ is, and the same for holomorphicity and cuspidality. \label{holorth}
\end{thm}
Once again, if $D=D_{\overline{M}}$ and $H$ is the diagonal subgroup, then Theorem \ref{VVhol} yields the holomorphic case of Theorem 2.5 of \cite{[Ze3]}, and the known cases of theta expansions of scalar-valued Jacobi forms with positive definite lattice index. The analogues of Proposition \ref{FuparrowG} and Theorems \ref{peruparrow} and \ref{largerHperp} in the holomorphic special case (also in the Hermitian setting) look the same as the general case, and we shall not rewrite them here.

\smallskip

The classical Jacobi forms, from \cite{[EZ]} and others, had integers as their indices, with some higher rank ones having symmetric matrices as their indices (see, e.g., \cite{[Zi]}), and the $\zeta$-variable was replaced by a complex number $z$ (or a vector of complex numbers $\{z_{i}\}_{i=1}^{b_{+}}$ for some rank $b_{+}$). These matrices (or integers, as $1\times1$ matrices) were later realized to be just the Gram matrices of a basis of positive definite orthogonal $\mathbb{Z}$-lattice, that now serves as a more natural index as in \cite{[A]}, \cite{[Mo]}, and in other contexts in \cite{[Boy]}, where $\zeta$ was the complex linear combination of the chosen basis elements, with the coefficients $z_{i}$, $1 \leq i \leq b_{+}$ (or just $z$ when $b_{+}=1$). Note that in Fourier--Jacobi expansions of Siegel modular forms, the Fourier--Jacobi coefficients are Jacobi forms with matrix (or integer) index, and when establishing modularity results of various objects, as in \cite{[Zh]}, \cite{[BRa]}, or the orthogonal part of \cite{[BRZ]}, indeed this interpretation is the required one. Since \cite{[Xi]} and the unitary part of \cite{[BRZ]} need Hermitian Jacobi forms also of matrix index, We will now carry out the same transfer for defining those formally. We shall stick to the positive definite setting, as when the pairing gets involved with the decomposition associated with an element $v$ in some Grassmannian, more complicated notation is required, and the simpler special case suffices for the applications.

Since we shall work with a basis over $\mathcal{O}$, we thus take the $\mathcal{O}$-lattice $M$ to be a \emph{free} module over $\mathcal{O}$, and we fix some basis $M=\bigoplus_{i=1}^{b_{+}}\mathcal{O}\varepsilon_{i}$. Let $S$ be the matrix whose $ij$-entry is $\langle\varepsilon_{j},\varepsilon_{i}\rangle\in\mathcal{D}^{-1}$, which is Hermitian and thus its diagonal entries are in $\frac{1}{2}\mathbb{Z}$. The fact that $M$ is assumed to be even and positive definite implies that the diagonal entries of $S$ are integral and that $S$ is a positive definite matrix. We let $\pi$ be the map identifying a vector $a\in\mathbb{K}^{b_{+}}$ with the element $\sum_{i=1}^{b_{+}}a_{i}\varepsilon_{i}$ of $V:=M\otimes_{\mathbb{Z}}\mathbb{Q}=M\otimes_{\mathcal{O}}\mathbb{K}$, which thus identifies $\mathcal{O}^{b_{+}}$ with $M$.

Now, as $\frac{1}{\sqrt{-d}}$ generates $\mathcal{D}^{-1}$ over $\mathcal{D}$, we say that the \emph{dual basis} for $\{\varepsilon_{i}\}_{i=1}^{b_{+}}$ are the elements $\{\epsilon_{i}\}_{i=1}^{b_{+}} \subseteq M_{\mathbb{Q}}$ such that $\langle\epsilon_{j},\varepsilon_{i}\rangle=\frac{\delta_{ij}}{\sqrt{-d}}$. Therefore the transition matrix from the basis to its dual is $\frac{1}{\sqrt{-d}}S^{-t}=\frac{1}{\sqrt{-d}}\overline{S}^{-1}$ (recall that $S$ is Hermitian and invertible), and thus $M^{*}$ is spanned (freely) over $\mathcal{O}$ by $\big\{S^{-t}\varepsilon_{i}\big/\sqrt{-d}\big\}_{i=1}^{b_{+}}$, namely the dual lattice $M^{*}$ is the image of $S^{-t}\mathcal{O}^{b_{+}}\big/\sqrt{-d}$ under the map $\pi$ from the previous paragraph.

We shall be using, for our variables, the convention that the vector $\zeta$ is $\sum_{i=1}^{b_{+}}z_{i}\varepsilon_{i,\mathbb{C}}$ with some $z\in\mathbb{C}^{b_{+}}$, and $\omega=\sum_{i=1}^{b_{+}}w_{i}\varepsilon_{i,\overline{\mathbb{C}}}$ for $w\in\mathbb{C}^{b_{+}}$. We shall use the notation $r^{*}$ for the conjugate transpose of any vector $r\in\mathbb{C}^{b_{+}}$, and all this converts Theorem \ref{holTheta} into the following form.
\begin{thm}
Consider the $\mathbb{C}\big[\big(S^{-t}\mathcal{O}^{b_{+}}\big/\sqrt{-d}\big)/\mathcal{O}^{b_{+}}\big]$-valued function $\Theta_{S}$ on the product $\mathcal{H}\times\mathbb{C}^{b_{+}}\times\mathbb{C}^{b_{+}}$, which is defined by \[\Theta_{S}(\tau,z,w):=\textstyle{\sum_{a \in S^{-t}\mathcal{O}^{b_{+}}/\sqrt{-d}}\mathbf{e}(\tau a^{*}Sa+a^{*}Sz+w^{t}Sa)\mathfrak{e}_{a+\mathcal{O}^{b_{+}}}}.\] This function is holomorphic and satisfies the heat equation associated with $S$, its scalar-valued periodicity law is \[\Theta_{S}(\tau,z+\tau u+v,w+\tau\overline{u}+\overline{v})=\mathbf{e}\big(-\tau u^{*}Su-u^{*}Sz-w^{t}Su\big)\Theta_{S}(\tau,z,w)\] wherever $u$ and $v$ are in $\mathcal{O}^{b_{+}}$, and for $A\in\operatorname{U}(1,1)(\mathbb{Z})$ it satisfies \[\Theta_{S}\Big(A\tau,\tfrac{z}{j(A,\tau)},\tfrac{w}{j^{c}(A,\tau)}\Big)=j(A,\tau)^{b_{+}}\mathbf{e}\Big(\tfrac{j_{A}' \cdot w^{t}Sz}{j(A,\tau)}\Big)\rho_{M}(A)\Theta_{S}(\tau,z,w).\] \label{ThetaS}
\end{thm}
Note that for the vector $a$ in the sum from Theorem \ref{ThetaS}, as well as in $u$ representing the translationn by $\sigma \in M$, the action of $\mathcal{O}$ involves the endomorphism $J$, meaning that $a$ and $u$ have to be conjugated (hence taken to $a^{*}$) when it represents $\lambda \in M^{*}$ in the second variable of the Hermitian pairing (including for $|\lambda|^{2}=\langle\lambda,\lambda\rangle$), while for $w$ the extension by scalars is $\mathbb{C}$-linear also in this second variable, leaving $w$ unconjugated (which is also required for the holomorphicity). Also, recall from Corollary \ref{Konspaces} that in the basis $\varepsilon_{i,\mathbb{C}}$ associated with the coordinates from $z$, the action of $\mathbb{K}$ is linear, while using the basis $\varepsilon_{i,\overline{\mathbb{C}}}$, it is conjugate linear, explaining the conjugation in the second variable in the periodicity equation in that theorem (the scalars $j(A,\tau)$ and $j^{c}(A,\tau)$ in the modularity equation are not affected, as they do not involve the action of $\mathbb{K}$). We remark that $\Theta_{S}$ depends also on the order $\mathcal{O}$, despite it not appearing in our notation.

\smallskip

We can now adapt the definition of Jacobi forms to have matrix index as well. For $\Gamma$ as above, we say that $\Phi:\mathcal{H}\times\mathbb{C}^{b_{+}}\times\mathbb{C}^{b_{+}}\to\mathbb{C}$ is a \emph{Hermitian Jacobi form of weight $k$ and index $S$ with respect to $\Gamma$} in case it is holomorphic in the variables from $\mathbb{C}^{b_{+}}\times\mathbb{C}^{b_{+}}$, satisfies the equality \[\Phi(\tau,z+\tau u+v,w+\tau\overline{u}+\overline{v})=\mathbf{e}\big(-\tau u^{*}Su-u^{*}Sz-w^{t}Su\big)\Phi(\tau,z,w)\] for any $u$ and $v$ in $\mathcal{O}^{b_{+}}$ (with $\tau\in\mathcal{H}$ and $z$ and $w$ from $\mathbb{C}^{b_{+}}$), as well as \[\Phi\Big(A\tau,\tfrac{z}{j(A,\tau)},\tfrac{w}{j^{c}(A,\tau)}\Big)=j(A,\tau)^{k}\mathbf{e}\Big(\tfrac{j_{A}' \cdot w^{t}Sz}{j(A,\tau)}\Big)\Phi(\tau,z,w)\] for any $A\in\Gamma$ (and the growth conditions). The Fourier expansion from Equation \eqref{Fourposdef} thus involves the coefficients $\tilde{c}_{m,r}$ for $r \in S^{-t}\mathcal{O}^{b_{+}}/\sqrt{-d}$ (which are constants if and only if $\Phi$ is holomorphic) and the exponents $\mathbf{e}(m\tau+r^{*}Sz+w^{t}Sr)$, the equality of the coefficients becomes $\tilde{c}_{m,r}=\tilde{c}_{m+u^{*}Sr+r^{*}Su+u^{*}Su,r+u}$ for any $u\in\mathcal{O}^{b_{+}}$. The definitions of weakly holomorphic, weak, holomorphic, and cuspidal Hermitian Jacobi forms of index $S$ take the obvious analogous forms.

We can now state the matrix index analogue of Theorem \ref{holJac}.
\begin{thm}
The map taking a modular form $F$ of weight $k-b_{+}$ and representation $\rho_{\overline{M}}$ with respect to $\Gamma$ to the function $\Phi_{S}^{F}:\mathcal{H}\times\mathbb{C}^{b_{+}}\times\mathbb{C}^{b_{+}}\to\mathbb{C}$ defined by $\Phi_{S}^{F}(\tau,z,w):=[\Theta_{S}(\tau,z,w),F(\tau)]_{M}$ is a linear bijection from such modular forms onto the space of Hermitian Jacobi forms of weight $k$ and index $S$ with respect to $\Gamma$, which preserves weak holomorphicity, usual holomorphicity, and cuspidality. \label{JacS}
\end{thm}
For appropriate rank 1 lattices, Theorem \ref{JacS} reproduces some of the results of \cite{[Hav1]}, \cite{[Hav2]}, \cite{[De]}, and others (at least in the integral weight case).

\smallskip

The positive definite Hermitian matrix $S$ in Theorems \ref{ThetaS} and \ref{JacS} represented the pairings of elements from some even $\mathcal{O}$-lattice $M$, and were thus taken from $\mathcal{D}^{-1}$ (with diagonal entries from $\mathbb{Z}$). Applications like those from \cite{[Xi]} and \cite{[BRZ]} require, however, Jacobi forms whose index is a Hermitian matrix $S$ which does not necessarily satisfy these integrality conditions. These will no longer exist as scalar-valued Jacobi forms, and will have to be constructed as vector-valued Jacobi forms. Moreover, in these papers the Hermitian matrix does not have to be positive definite, but can be positive semidefinite as well.

So let $S$ be any positive semidefinite Hermitian matrix of size $c_{+} \times c_{+}$ over $\mathbb{K}$, consider the order $\mathcal{O}\subseteq\mathbb{K}$ as fixed, and let $\tilde{V}$ be a vector space over $\mathbb{K}$ with basis $\{\tilde{\varepsilon}_{i}\}_{i=1}^{c_{+}}$, which we endow with the Hermitian form defined by $\langle\tilde{\varepsilon}_{i},\tilde{\varepsilon}_{j}\rangle=S_{ji}$. This Hermitian form is thus positive semidefinite, in which the elements of $\tilde{V}$ that are perpendicular to all of $\tilde{V}$ are the combinations $\sum_{i=1}^{c_{+}}a_{i}\tilde{\varepsilon}_{i}$ for vectors $a$ that are annihilated by $S$ as a linear operator. We thus define $V$ to be the resulting quotient, denote its dimension by $b_{+}$ and the image of $\tilde{\varepsilon}_{i}\in\tilde{V}$ by $\varepsilon_{i} \in V$, and write $\pi:\mathbb{K}^{c_{+}} \to V$ for the map taking $a\in\mathbb{K}^{c_{+}}$ to $\sum_{i=1}^{c_{+}}a_{i}\varepsilon_{i}$ once again, which is now surjective with kernel equal to $\ker S$.

Set $M_{S}:=\sum_{i=1}^{c_{+}}\mathcal{O}\varepsilon_{i} \subseteq V$ (it is no longer an $\mathcal{O}$-lattice in the sense of Definition \ref{Olat} since it may be not projective and the pairings from $M_{S}$ need not necessarily by in $\mathcal{D}^{-1}$ anymore), and let $M_{S}^{*}$ be its dual $\{\lambda \in V|\langle\lambda,\tilde{M}_{S}\rangle\subseteq\mathcal{D}^{-1}\}$. Then, if $M$ is any even $\mathcal{O}$-lattice, as in Definition \ref{Olat}, that is contained in $M_{S}^{*}$, then $M_{S}$ is contained in $M^{*}$, namely we have $\varepsilon_{i} \in M^{*}$ for every $1 \leq i \leq c_{+}$. For such $M$ we set $\Omega:=\pi^{-1}(M)$ and $\Omega^{*}:=\pi^{-1}(M^{*})$, with the latter containing $\mathcal{O}^{c_{+}}$ since $\pi(\mathcal{O}^{c_{+}})=M_{S} \subseteq M^{*}$ by our assumption, and let $\overline{\Omega}:=\Omega/\ker S$ and $\overline{\Omega}^{*}:=\Omega^{*}/\ker S$ be the associated quotients. Then the quotients $D_{\Omega}:=\Omega^{*}/\Omega$, $D_{\overline{\Omega}}:=\overline{\Omega}^{*}/\overline{\Omega}$, and $D_{M}$ are all isomorphic (the latter via $\pi$), and we set $D_{M,S}$ to be the image $(M_{S}+M)/M$ of $M_{S}$ in $D_{M}$, as well as $D_{\Omega,S}$ its isomorphic image in $D_{\Omega}$, namely $(\mathcal{O}^{c_{+}}+\Omega)/\Omega$.

The orthogonal pairing takes the $\pi$-images of two elements $a$ and $b$ of $\Omega^{*}$ (and more generally, in $\mathbb{K}^{c_{+}}$) to $\operatorname{Tr}^{\mathbb{K}}_{\mathbb{Q}}(a^{*}Sb)\in\mathbb{Q}$, and by projecting to $\mathbb{Q}/\mathbb{Z}$ the result depends only on the class of $b$ (and also just that of $a$) in $D_{\Omega}$. We can thus write $\operatorname{Tr}^{\mathbb{K}}_{\mathbb{Q}}(a^{*}S\gamma)$ for the resulting well-defined expression in $\mathbb{Q}/\mathbb{Z}$ for $a\in\Omega^{*}$ and $\gamma \in D_{\Omega}$, and similarly $\operatorname{Tr}^{\mathbb{K}}_{\mathbb{Q}}(\delta^{*}S\gamma)\in\mathbb{Q}/\mathbb{Z}$ when $\delta$ is also from $D_{\Omega}$, which represents the pairing of their $D_{M}$-images.

The same considerations leading to Theorem \ref{ThetaS}, including the substitutions $\zeta=\sum_{i=1}^{c_{+}}z_{i}\varepsilon_{i,\mathbb{C}}$ and $\omega=\sum_{i=1}^{c_{+}}w_{i}\varepsilon_{i,\overline{\mathbb{C}}}$ for vectors $z$ and $w$ from $\mathbb{C}^{c_{+}}$, then transforms the theta function $\Theta_{M}$ into the function
\begin{equation}
\Theta_{S,\Omega}(\tau,z,w):=\textstyle{\sum_{\gamma \in D_{\Omega}}\big[\sum_{a\in\gamma+\Omega}\mathbf{e}(\tau a^{*}Sa+a^{*}Sz+w^{t}Sa)\big]\mathfrak{e}_{\gamma}} \label{ThetaOmega}
\end{equation}
(note that both sides indeed remain invariant when one changes $z$ or $w$ by an element of the appropriate complexification of $\ker S$). The map taking $\mathcal{O}^{c_{+}}$ to its image in $D_{\Omega,S} \subseteq D_{\Omega} \cong D_{M}$ (with the quadratic structure) combines with the operators from Definition \ref{opersDM} to produce the operators $t_{u}:\mathfrak{e}_{\gamma}\mapsto\mathfrak{e}_{\gamma-(u+\Omega)}$ and $\chi_{v}:\mathfrak{e}_{\gamma}\mapsto\mathbf{e}\big(\operatorname{Tr}^{\mathbb{K}}_{\mathbb{Q}}(v^{*}S\gamma)\big)\mathfrak{e}_{\gamma}$ for every $u$ and $v$ from $\mathcal{O}^{c_{+}}$. Then the periodicity from Theorem \ref{holTheta}, restricted to $\mathcal{O}^{c_{+}}\subseteq\Omega^{*}$, becomes
\begin{equation}
\Theta_{S,\Omega}(\tau,z+\tau u+v,w+\tau\overline{u}+\overline{v})=\mathbf{e}\big(-\tau u^{*}Su-u^{*}Sz-w^{t}Su\big)\chi_{v}t_{u}\Theta_{S,\Omega}(\tau,z,w) \label{perOmega}
\end{equation}
for any such $u$ and $v$ (this includes the scalar-valued periodicity for the sublattice $\Omega\cap\mathcal{O}^{c_{+}}=\pi^{-1}(M \cap M_{S}^{*})$, but not necessarily for all of $\Omega$ and $M$). The theta function $\Theta_{S,\Omega}$ also has the modularity property from Theorem \ref{ThetaS}, with the representation $\rho_{M}$ associated with our $\mathcal{O}$-lattice $M$.

Note that in the case considered in that theorem, the module $M_{S}$ is an even free $\mathcal{O}$-lattice and is thus contained in $M_{S}^{*}$, the lattice $M$ is $M_{S}$, and thus $\Omega$ (which also equals $\overline{\Omega}$, and $c_{+}=b_{+}$, since $S$ is positive definite now) is $\mathcal{O}^{b_{+}}$. Hence the function $\Theta_{S,\Omega}$ from Equation \eqref{ThetaOmega} is $\Theta_{S}$ defined there, and Equation \eqref{perOmega} reduces to the scalar-valued periodicity property from that theorem, since $t_{u}$ and $\chi_{v}$ become trivial for $u$ and $v$ in $\mathcal{O}^{b_{+}}=\Omega$.

\smallskip

Consider now any other discriminant form $D$ and a horizontal isotropic subgroup $H$ in $D_{\Omega} \oplus D$, such that the projection $H_{\Omega}$ of $H$ to $D_{\Omega}$ is perpendicular to $D_{\Omega,S}$. Then Lemma \ref{indchit} implies that $t_{u}$ and $\chi_{v}$ remain well-defined on the associated quotient $\Delta_{H}:=H^{\perp}/H$ for any $u$ and $v$ in $\mathcal{O}^{c_{+}}$. Recalling that for any $a\in\mathcal{O}^{c_{+}}$, its image in $D_{\Omega}$ must have quadratic value $a^{*}Sa+\mathbb{Z}\in\mathbb{Q}/\mathbb{Z}$, we may define vector-valued Hermitian Jacobi forms with index $S$ as follows.
\begin{defn}
Let $\Delta$ be a discriminant form, admitting a group homomorphism $\psi:\mathcal{O}^{c_{+}}\to\Delta$ such that $\psi(a)=0$ in case $Sa=0$ and $|\psi a|^{2}=a^{*}Sa+\mathbb{Z}$ for any $a\in\mathcal{O}^{c_{+}}$, and let $\Gamma$ be a finite index subgroup of $\operatorname{U}(1,1)(\mathbb{Z})$. Then a \emph{Hermitian Jacobi form of weight $k$, representation $\rho_{\Delta}$, and index $(S,\psi)$} is a smooth function $\Phi:\mathcal{H}\times\mathbb{C}^{c_{+}}\times\mathbb{C}^{c_{+}}\to\mathbb{C}[\Delta]$ which satisfies the periodicity property
\begin{equation}
\Phi(\tau,z+\tau u+v,w+\tau\overline{u}+\overline{v})=\mathbf{e}\big(-\tau u^{*}Su-u^{*}Sz-w^{t}Su\big)\chi_{\psi v}t_{\psi u}\Phi(\tau,z,w) \label{SperHerm}
\end{equation}
for every $u$ and $v$ in $\mathcal{O}^{c_{+}}$, the modularity property
\begin{equation}
\Phi\Big(A\tau,\tfrac{z}{j(A,\tau)},\tfrac{w}{j^{c}(A,\tau)}\Big)=j(A,\tau)^{k}\mathbf{e}\Big(\tfrac{j_{A}' \cdot w^{t}Sz}{j(A,\tau)}\Big)\rho_{\Delta}(A)\Phi(\tau,z,w), \label{SmodHerm}
\end{equation}
and the exponential growth condition on the resulting Fourier coefficients. \label{defindS}
\end{defn}

Before we state the result for the Hermitian Jacobi forms from Definition \ref{defindS}, we consider the orthogonal analogue, for which establishing the result is simpler, for reasons to be explained just before the result itself, which appears as Theorem \ref{HJS} below.

\smallskip

In a similar manner, if $S$ is any positive semidefinite symmetric matrix over $\mathbb{Q}$, of size $c_{+} \times c_{+}$, then we can endow the vector space $\tilde{V}:=\bigoplus_{i=1}^{c_{+}}\mathbb{Q}\tilde{\varepsilon}_{i}$ by the positive definite quadratic form taking $\sum_{i=1}^{c_{+}}a_{i}\tilde{\varepsilon}_{i}$ for $a\in\mathbb{Q}^{c_{+}}$ to $a^{t}Sa$. Then we get $(\tilde{\varepsilon}_{i},\tilde{\varepsilon}_{j})=2S_{ji}$, yielding a positive semidefinite symmetric bilinear form whose kernel again coincides with that of $S$, and we define the corresponding quotient $V$ of dimension $b_{+}$, the images $\varepsilon_{i}$, $1 \leq i \leq c_{+}$, and the surjection $\pi:\mathbb{Q}^{c_{+}} \to V$ as above. We again set $L_{S}=\sum_{i=1}^{c_{+}}\mathbb{Z}\varepsilon_{i} \subseteq V$ and its dual $L_{S}^{*}$ there, and given an even lattice $L \subseteq L_{S}^{*} \subseteq V$ we again define $\Omega:=\pi^{-1}(L)$ as contained in $\Omega^{*}:=\pi^{-1}(L^{*})$, and the quotients $\overline{\Omega}:=\Omega/\ker S$ and $\overline{\Omega}^{*}:=\Omega^{*}/\ker S$, with the isomorphic quotients $D_{\Omega}$, $D_{\overline{\Omega}}$, and $D_{L}$. Observing that the pairing of $\pi(a)$ with $\pi(b)$ is now $2a^{t}Sb$, and expressing the variable $\zeta \in V_{\mathbb{C}}$ of any orthogonal Jacobi theta function (such as the one from Equation \eqref{Thetaorth}, but for positive definite $L$) as $\sum_{i=1}^{c_{+}}z_{i}\varepsilon_{i}$ for some $z\in\mathbb{C}^{c_{+}}$, we can give the orthogonal analogue of Definition \ref{defindS}.
\begin{defn}
Consider a discriminant form $\Delta$, a subgroup $\Gamma$ of finite index in $\operatorname{SL}_{2}(\mathbb{Z})$ (or in its metaplectic cover), and take a map $\psi:\mathbb{Z}^{c_{+}}\to\Delta$ satisfying $\frac{(\psi a)^{2}}{2}=a^{t}Sa+\mathbb{Z}$ for every $a\in\mathbb{Z}^{c_{+}}$ and vanishing on elements of $\ker S$. We then define a \emph{Jacobi form of weight $k$, representation $\rho_{\Delta}$, and index $(S,\psi)$} to be a smooth function $\Phi:\mathcal{H}\times\mathbb{C}^{c_{+}}\to\mathbb{C}[\Delta]$ for which the equality
\begin{equation}
\Phi(\tau,z+\tau u+v)=\mathbf{e}\big(-\tau u^{t}Su-2u^{t}Sz\big)\chi_{\psi v}t_{\psi u}\Phi(\tau,z) \label{Sperorth}
\end{equation}
holds for any to elements $u$ and $v$ from $\mathbb{Z}^{c_{+}}$ as does the equality
\begin{equation}
\Phi\big(A\tau,\tfrac{z}{j(A,\tau)}\big)=j(A,\tau)^{k}\mathbf{e}\Big(\tfrac{j_{A}' \cdot z^{t}Sz}{j(A,\tau)}\Big)\rho_{\Delta}(A)\Phi(\tau,z) \label{Smodorth}
\end{equation}
for $A\in\Gamma$, with an exponential growth condition on the Fourier coefficients. \label{indSdef}
\end{defn}
The Jacobi forms from Definition \ref{indSdef} are slightly related, in the positive definite case, to those investigated in \cite{[W]}.

Let $\psi$ be as in Definition \ref{defindS}, and set $\Omega_{\mathbb{Z}}:=\ker\psi$, $\Omega:=\Omega_{\mathbb{Z}}+\ker S\subseteq\mathbb{Q}^{c_{+}}$ (this equals $\Omega_{\mathbb{Z}}$ in the positive definite case, but not otherwise), $\overline{\Omega}=\Omega/\ker S$, and $L:=\pi(\Omega)=\pi(\Omega_{\mathbb{Z}})$. The fact that $\psi$ preserves the quadratic value modulo $\mathbb{Z}$ implies that $L$ is an even lattice in $V$, and since changing an element of $\mathbb{Z}^{c_{+}}$ by $\Omega$ does not affect the quadratic value (by the same assumption), we deduce that $L_{S}=\pi(\mathbb{Z}^{c_{+}})$ is contained in $L^{*}$. We define $\Omega^{*}:=\pi^{-1}(L^{*})$ and $\overline{\Omega}:=\Omega^{*}/\ker S$, with the quotients $D_{\Omega} \cong D_{\overline{\Omega}} \cong D_{L}$ (thus inheriting the quadratic structure), and then $\psi$ identifies the quotient $D_{\Omega,S}:=\mathbb{Z}^{c_{+}}/\Omega_{\mathbb{Z}} \subseteq D_{\Omega}$ with a subgroup of $\Delta$. Set $H_{\Omega}:=D_{\Omega,S}^{\perp} \subseteq D_{\Omega}$, so that $D_{\Omega,S}$ is $H_{\Omega}^{\perp}$, and $H_{\Omega}^{\perp}$ is embedded in $\Delta$.

Now, note that Definition \ref{indSdef} involves only the discriminant form $\Delta$ and not $D$ (as does Definition \ref{defindS}), and the construction of $D$ from $\Delta$, $D_{L} \cong D_{\Omega}$, and the image $H_{L}$ of $H_{\Omega}$ is not always clear (see Remark \ref{noconstD}). We define $\Theta_{S,\Omega}$ in this setting to be the orthogonal analogue $\sum_{a\in\Omega^{*}}\mathbf{e}(\tau a^{t}Sa+2a^{t}Sz)\mathfrak{e}_{a+\Omega}$ of the theta function from Equation \eqref{ThetaOmega} (which satisfies the periodicity condition comparing $\Theta_{S,\Omega}(\tau,z+\tau u+v)$ with $\mathbf{e}(-\tau u^{t}Su-2u^{t}Sz)\chi_{v}t_{u}\Theta_{S,\Omega}(\tau,z)$ for any $u$ and $v$ in $\mathbb{Z}^{c_{+}}$), and obtain the following result.
\begin{thm}
Assume that there exists a discriminant form $D$, with a subgroup $H_{D}$ that is isomorphic to $H_{\overline{L}}$ such that if $H$ is the isotropic horizonal subgroup of $D_{\Omega} \oplus D$ resulting from this identification as in Lemma \ref{horizdecom}, then the associated quotient $\Delta_{H}:=H^{\perp}/H$ is isomorphic with $\Delta$ in a way that commutes with the embedddings of $H_{L}^{\perp}=H_{\Omega}^{\perp}$ into $\Delta_{H}$ and into $\Delta$. Then any Jacobi form $\Phi$ of weight $k$, representation $\rho_{\Delta}$, and index $(S,\psi)$ with respect to $\Gamma$, as in Definition \ref{indSdef}, is given by $\Phi(\tau,z)=\downarrow_{H}\big(\Theta_{\Omega,S}(\tau,z) \otimes F(\tau)\big)$ for a unique modular form $F$ of weight $k-\frac{b_{+}}{2}$ and representation $\rho_{D}$ with respect to $\Gamma$, which is holomorphic if and only if $\Phi$ is. \label{JacindS}
\end{thm}

\begin{proof}
The fact that $\psi$ is trivial on elements of $(\ker S)_{\mathbb{Z}}:=\ker S\cap\mathbb{Z}^{c_{+}}$ in Definition \ref{indSdef} implies that Equation \eqref{Sperorth} there yields the invariance of $\Phi$, for fixed $\tau\in\mathcal{H}$, under translations from the full lattice $(\ker S)_{\mathbb{Z}}\oplus\tau(\ker S)_{\mathbb{Z}}$ inside the complex vector space $(\ker S)_{\mathbb{C}}:=(\ker S)\otimes_{\mathbb{Q}}\mathbb{C}$. As $\Phi$ is holomorphic in $z$, this implies that $\Phi(\tau,z)$ does not change when $z$ is altered by an element of this vector space, meaning that its dependence on $z$ can be given in terms of $z+(\ker S)_{\mathbb{C}}$, or equivalently its isomorph $V_{\mathbb{C}}$ via the complexification of $\pi$. We thus write $\Phi(\tau,z)=\tilde{\Phi}(\tau,\zeta)$ for $\zeta=\sum_{i=1}^{c_{+}}z_{i}\varepsilon_{i}$ and some smooth function $\tilde{\Phi}$ which is holomorphic in $\zeta$.

By substituting $\sigma=\sum_{i=1}^{c_{+}}u_{i}\varepsilon_{i}$ and $\nu=\sum_{i=1}^{c_{+}}v_{i}\varepsilon_{i}$, and noting that $t_{\sigma}$ and $\chi_{\nu}$ act as $t_{\psi u}$ and $\chi_{\psi v}$ respectively, Equation \eqref{Sperorth} from Definition \ref{indSdef} for $\Phi$ becomes Equation \eqref{perorth} inside Definition \ref{orthVV} for the function $\tilde{\Phi}$ and the positive definite lattice $L$. Moreover, the fact that $\frac{\zeta^{2}}{2}$ equals $z^{t}Sz$ compares Equation \eqref{Smodorth} for $\Phi$ with Equation \eqref{modorth} for $\tilde{\Phi}$. This implies that $\tilde{\Phi}$ is a Jacobi form of weight $k$, index $(L,H)$ with our $H$, and representation $\rho_{\Delta}$, to which we can apply  Theorem \ref{VVorth} by the assumption on the existence of $D$. The fact that the theta function $\Theta_{L}(\tau,\zeta)$ from that theorem, evaluated at our $\zeta$, is just $\Theta_{S,\Omega}(\tau,z)$ establishes the desired result. This proves the theorem.
\end{proof}

\begin{rmk}
Some Jacobi forms $\Phi$ may have finer periodicity properties, thus increasing $\Omega$ to a larger sublattice in $\mathbb{Q}^{c_{+}}$ via Theorem \ref{peruparrow} (up to some arrow operators), and in some situations it may be more convenient to work with the maximal one from Remark \ref{maxlat}. Note that decreasing the subgroup $H_{L}^{\perp}$ as in Theorem \ref{largerHperp} means increasing the image $D_{\Omega,S}$ of $\mathbb{Z}^{c_{+}}$ in $D_{\Omega}$, producing functions that may be better described in terms of another map $\psi$, hence a different matrix $S$. \label{maxS}
\end{rmk}

\smallskip

In order to establish an analogue of Theorem \ref{JacindS} for the Hermitian Jacobi forms from Definition \ref{defindS}, we can use the same construction, but we need the image $\overline{\Omega}$ of the kernel $\Omega_{\mathcal{O}}$ of $\psi$ inside $\mathcal{O}^{c_{+}}$, or of $\Omega:=\Omega_{\mathcal{O}}+\ker S$, modulo $\ker S$ to be an $\mathcal{O}$-module, and a projective one (then it will be an $\mathcal{O}$-lattice as in Definition \ref{Olat}). The module condition can be obtained by assuming that $\Delta$ carries an $\mathcal{O}$-structure as in Remark \ref{JonD}, and $\psi$ to be a map of $\mathcal{O}$-modules. The projectivity is now immediate in case $\mathcal{O}=\mathcal{O}_{\mathbb{K}}$, but not at all clear otherwise. Since we need $\psi$ to embed $D_{\Omega,S}$ into $\Delta$, we cannot replace $\Omega$ by a subgroup that may have this properties in a straightforward manner. This is why we make the additional assumptions in the following result.
\begin{thm}
Take $S$ and $\psi$ as in Definition \ref{defindS} for which the quotient $\overline{\Omega}$ of $\ker\psi$ modulo its intersection with $\ker S$ is a projective $\mathcal{O}$-module, and let $D$ be a discriminant form satisfying the assumptions from Theorem \ref{JacindS}. In this case, for any Hermitian Jacobi form $\Phi$ of weight $k$, representation $\rho_{\Delta}$, and index $(S,\psi)$ with respect to $\Gamma$ there exists a unique modular form $F$, holomorphic precisely when $\Phi$ is and having weight $k-b_{+}$ and representation $\rho_{D}$ with respect to $\Gamma$, such that the equality $\Phi(\tau,z,w)=\downarrow_{H}\big(\Theta_{\Omega,S}(\tau,z,w) \otimes F(\tau)\big)$, with $\Theta_{\Omega,S}$ from Equation \eqref{ThetaOmega}, holds for every $\tau\in\mathcal{H}$ and $z$ and $w$ from $\mathbb{C}^{c_{+}}$. \label{HJS}
\end{thm}

\begin{proof}
We follow the proof of Theorem \ref{JacindS}. Equation \eqref{SperHerm} and the triviality assumption on $\psi$ yields the invariance of $\Phi$ under translating the pair $(z,w)$ by the sum of $\{(a,\overline{a})|a\in(\ker S)_{\mathbb{Z}}\}$ and $\tau\{(a,\overline{a})|a\in(\ker S)_{\mathbb{Z}}\}$ (via Corollary \ref{Konspaces}), which is a full lattice inside $(\ker S)_{\mathbb{C}}\oplus(\ker S)_{\overline{\mathbb{C}}}$. We can thus write $\Phi(\tau,z,w)$ as $\tilde{\Phi}(\tau,\zeta,\omega)$ with $\zeta=\sum_{i=1}^{c_{+}}z_{i}\varepsilon_{i,\mathbb{C}}$ and $\omega=\sum_{i=1}^{c_{+}}w_{i}\varepsilon_{i,\overline{\mathbb{C}}}$ for some smooth $\tilde{\Phi}$. Using the same substitutions (and Corollary \ref{Konspaces} again), Equations \eqref{SperHerm} and \eqref{SmodHerm} for $\Phi$ translate into Equations \eqref{perHMperp} and \eqref{VVmodJac} in Definition \ref{VVdef} for $\tilde{\Phi}$ (the latter uses the equality $\langle\zeta,\omega\rangle=w^{t}Sz$ obtained via Proposition \ref{decompair}), so that the assumptions on $\overline{\Omega}$ and $D$ show that $\tilde{\Phi}$ is a Hermitian Jacobi form of weight $k$, index $(M,H)$ for $M:=\pi(\ker\psi)$ and our $H$, and representation $\rho_{\Delta}$. Theorem \ref{VVmain} and the fact that substituting our $\zeta$ and $\omega$ inside the theta function from Equation \eqref{JacTheta} for our positive definite $\mathcal{O}$-lattice $M$ produces $\Theta_{S,\Omega}(\tau,z,w)$ from Equation \eqref{ThetaOmega} yields the desired assertion. This proves the theorem.
\end{proof}
Of course, Remark \ref{maxS} holds equally well in the Hermitian setting considered in Theorem \ref{HJS}.

\smallskip

We conclude with the following formula, which is also useful for several applications of orthogonal and Hermitian Jacobi forms of matrix index.
\begin{rmk}
The Fourier expansion from Equation \eqref{Fourposdef} for $\zeta$ and $\omega$ given in terms of $z$ and $w$ from $\mathbb{C}^{c_{+}}$, can be written in terms of coordinates from $\mathbb{K}$, also in the vector-valued setting considered in Definition \ref{defindS} and Theorem \ref{HJS}. In these coordinates, the expression $\langle\zeta,\lambda\rangle+\langle\lambda,\omega\rangle$ will be $r^{*}z+w^{t}r$ for some $r\in\mathbb{K}^{c_{+}}$ representing the coordinates of $\lambda$ in the appropriate dual space, and the fact that the Jacobi form $\Phi$ was seen in the proof of Theorem \ref{HJS} to be unaffected by changing $z$ or $w$ by elements of $\ker S$, we deduce that only vectors $r\in\mathbb{K}^{c_{+}}$ such that $r^{t}\ker S=\{0\}$, or equivalently $r^{*}\ker S=\{0\}$, may be considered (indeed, this is the dual condition to the division by $\ker S$ resulting from $\pi$). We thus extend the notation from Theorem \ref{ThetaS} and write $S^{-t}\mathcal{O}^{c_{+}}\big/\sqrt{-d}$ for the set of elements $r\in\mathbb{K}^{c_{+}}$ satisfying $S^{t}r\in(\mathcal{D}^{-1})^{c_{+}}$ and $r^{t}\ker S=\{0\}$ (this will indeed produce $S^{-t}\mathcal{O}^{c_{+}}\big/\sqrt{-d}$ in the positive definite setting, when $S$ can be inverted), and the Fourier expansion thus takes the form \[\Phi(\tau,z,w)=\textstyle{\sum_{\beta\in\Delta}\sum_{m\in\frac{\beta^{2}}{2}+\mathbb{Z}}\sum_{r \in S^{-t}\mathcal{O}^{c_{+}}/\sqrt{-d}}\tilde{c}_{m,r,\beta}(y)\mathbf{e}(m\tau+r^{*}z+w^{t}r)\mathfrak{e}_{\beta}}\] (with some conditions relating the possible indices $r$ that may show up for every $\beta\in\Delta$). Similarly, for orthogonal Jacobi forms of index a positive semidefinite symmetric matrix $S$ of that size, we denote by $S^{-1}\mathbb{Z}^{c_{+}}$ the set of $r\in\mathbb{Q}^{c_{+}}$ for which $Sr\in\mathbb{Z}^{c_{+}}$ and $r^{t}\ker S=\{0\}$ (again yielding the meaning of this notation for invertible, positive definite $S$), yielding the Fourier expansion \[\Phi(\tau,z)=\textstyle{\sum_{\beta\in\Delta}\sum_{m\in\frac{\beta^{2}}{2}+\mathbb{Z}}\sum_{r \in S^{-1}\mathbb{Z}^{c_{+}}}\tilde{c}_{m,r,\beta}(y)\mathbf{e}(m\tau+r^{t}z)\mathfrak{e}_{\beta}}.\] \label{FourS}
\end{rmk}

\noindent\textsc{Einstein Institute of Mathematics, the Hebrew University of Jerusalem, Edmund Safra Campus, Jerusalem 91904, Israel}

\noindent E-mail address: zemels@math.huji.ac.il


\begin{thebibliography}{}{}

\bibitem[A]{[A]} Ajouz, A., \textsc{Hecke Operators on Jacobi Forms of Lattice Index and the Relation to Elliptic Modular Forms}, Ph.D. thesis, University of Siegen (2015).
\bibitem[BK]{[BK]} B\"{o}cherer, S., Kohnen, W., \textsc{Estimates for Fourier Coefficients of Siegel Cusp Forms}. Math. Ann., vol 297 issue 1, 499–-517 (1993).
\bibitem[Bor]{[Bor]} Borcherds, R. E., \textsc{Automorphic Forms with Singularities on Grassmannians}, Invent. Math., vol. 132, 491--562 (1998).
\bibitem[Boy]{[Boy]} Boylan, H., \textsc{Jacobi Forms, Finite Quadratic Modules and Weil Representations over Number Fields}, Lecture Notes in Mathematics 2130, Springer International Publishing, Switzerland, xix+130pp (2015).
\bibitem[BS]{[BS]} Boylan, H., Skoruppa, N. P., \textsc{Jacobi Forms of Lattice Index I. Basic Theory}, pre-print, https://arxiv.org/abs/2309.04738 (2023)
\bibitem[Bra1]{[Bra1]} Braun, H., \textsc{Hermitian Modular Functions}, Ann of Math., vol 50 no. 4, 827--855 (1949).
\bibitem[Bra2]{[Bra2]} Braun, H., \textsc{Hermitian Modular Functions II}, Ann of Math., vol 51 no. 1, 92--104 (1950).
\bibitem[Bra3]{[Bra3]} Braun, H., \textsc{Hermitian Modular Functions III}, Ann of Math., vol 53 no. 1, 143--160 (1951).
\bibitem[BRR]{[BRR]} Bringmann, K., Raum, M., Richter, O. \textsc{Harmonic Maass-Jacobi Forms with Singularities and a Theta-Like Decomposition}, Trans. Math. Amer. Soc., vol 367, 6647--6670 (2015).
\bibitem[BRi]{[BRi]} Bringmann, K., Richter, O., \textsc{Zagier-Type Dualities and Lifting Maps for Harmonic Maass-Jacobi Forms}, Adv. Math., vol. 225 no. 4, 2298--2315 (2010).
\bibitem[BRa]{[BRa]} Bruinier, J. H., Westerholt-Raum, M. \textsc{Kudla's Modularity Conjecture and Formal Fourier--Jacobi Series}, Forum of Math. Pi, vol 3, (2015).
\bibitem[BRZ]{[BRZ]} Bruinier, J. H., Rosu, E., Zemel, S. \textsc{Modularity of Special 0-Cycles on Toroidal Compactifications of Shimura Varieties}, in preparation.
\bibitem[CWR]{[CWR]} Conrey, C., Westerholt-Raum, M., \textsc{Harmonic Maaß-Jacobi Forms of Degree 1 with Higher Rank Indices}, Int. J. Number Theory, vol 12 no. 7, 1871--1897 (2016).
\bibitem[De]{[De]} Dern, T., \textsc{Hermitesche Modulformen zweiten Grades}, Ph.D. thesis , RTWH Aachen (2001).
\bibitem[EZ]{[EZ]} Eichler, M., Zagier, D., \textsc{The Theory of Jacobi Forms}, Progress in Mathematics 55, Birkh\"{a}user, Boston, Basel, Stuttgart (1985).
\bibitem[FZ]{[FZ]} Farkas, H. M., Zemel, S., \textsc{Generalizations of Thomae's Formula for $Z_{n}$ curves}, DEVM 21, Springer--Verlag, xi+354pp (2011).
\bibitem[Hav1]{[Hav1]} Haverkamp, K., \textsc{Hermitesche Jacobiformen}, Schriftenreihe Math. Inst. Univ. M\"{u}nster 3. Ser 15, 105pp (1995).
\bibitem[Hav2]{[Hav2]} Haverkamp, K., \textsc{Hermitian Jacobi Forms}, Results Math., vol 29 issue 1, 78--89 (1996).
\bibitem[Hay]{[Hay]} Hayashida, S., \textsc{Skew-Holomorphic Jacobi Forms of Higher Degree}, in: Automorphic Forms and Zeta Functions--in Memory of Tsuneo Arakawa, 130-139, World Scientific (2006).
\bibitem[He]{[He]} Hentschel, M., \textsc{On Hermitian Theta Series and Modular Forms}, Ph.D. thesis , RTWH Aachen (2009).
\bibitem[Ho1]{[Ho1]} Hofmann, E., \textsc{Borcherds Products for $\operatorname{U}(1,1)$}, Int. J. Number Theory, issue 9 issue 7, 1801--1820 (2013).
\bibitem[Ho2]{[Ho2]} Hofmann, E., \textsc{Borcherds Products on Unitary Groups}, Math. Ann., vol 358 issue 3, 799--832 (2014).
\bibitem[K1]{[K1]} Kohnen, W., \textsc{Modular Forms of Half-Integral Weight on $\Gamma_{0}(4)$}, Math. Ann., vol 248, 249--266 (1980).
\bibitem[K2]{[K2]} Kohnen, W., \textsc{Newforms of Half-Integral Weight}, J. Reine Angew. Math., vol 333, 32--72 (1982).
\bibitem[KM]{[KM]} Kudla, S., Millson, J., \textsc{Intersection Numbers of Cycles on Locally Symmetric Spaces and Fourier Coefficients of Holomorphic Modular Forms in Several Complex Variables}, IHES Publi. Math. vol. 71, 121--172 (1990).
\bibitem[LZ]{[LZ]} Li, Y., Zemel, S., \textsc{Shimura Lift of Weakly Holomorphic Modular Forms}, Math. Z., vol 290, 37--61 (2018).
\bibitem[Ma]{[Ma]} Ma, S., \textsc{Quasi--Pullback of Borcherds Products}, Bull. London Math. Soc., vol. 51 issue 6, 1061--1078 (2019).
\bibitem[Mo]{[Mo]} Mocanu, A., \textsc{On Jacobi Forms of Lattice Index}, Ph.D. thesis, University of Nottingham (2019).
\bibitem[RR]{[RR]} Raum, M., Richter, O. \textsc{The Skew-Maass Lift I}, Res. Math. Sci., vol 6 paper 22, 1--59 (2019).
\bibitem[Sch]{[Sch]} Scheithauer, N. R., \textsc{The Weil representation of $SL_{2}(\mathbb{Z})$ and some Applications}, Int. Math. Res. Not., no. 8, 1488--1545 (2009).
\bibitem[Sk]{[Sk]} Skoruppa, N. P., \textsc{Developments in the Theory of Jacobi Forms}, in: Automorphic Functions and their Applications, 167–185, Acad. Sci. USSR, Inst. Appl. Math., Khabarovsk (1990).
\bibitem[Str]{[Str]} Str\"{o}mberg, F., \textsc{Weil Representations Associated to Finite Quadratic Modules}, Math. Z., vol 275 issue 1, 509--527 (2013).
\bibitem[WR]{[WR]} Westerholt-Raum, M., \textsc{H-Harmonic Maaß--Jacobi Forms of Degree 1: The Analytic Theory of Some Indefinite Theta Series}, Res. Math. Sci., vol 2 paper 12, 1--34 (2015).
\bibitem[W]{[W]} Williams, B., \textsc{Remarks on the Theta Decomposition of Vector-Valued Jacobi Forms}, J. Number Theory, vol 197, 250--267 (2019).
\bibitem[Xi]{[Xi]} Xia, J., \textsc{Some Cases of Kudla’s Modularity Conjecture for Unitary Shimura Varieties}, Forum of Math. Sigma, vol 10, paper 37, 1--31 (2022).
\bibitem[Ze1]{[Ze1]} Zemel, S., \textsc{A $p$-adic Approach to the Weil Representation of Discriminant Forms Arising from Even Lattices}, Math. Ann. Qu\'{e}bec, vol 39 issue 1, 61--89 (2015).
\bibitem[Ze2]{[Ze2]} Zemel, S., \textsc{Weight Changing Operators for Automorphic Forms on Grassmannians and Differential Properties of Certain Theta Lifts}, Nagoya Math. J., vol 228, 186--221 (2017).
\bibitem[Ze3]{[Ze3]} Zemel, S., \textsc{Jacobi Forms of Indefinite Lattice Index}, Res. Number Theory, vol 7, paper 58, 1--22 (2021).
\bibitem[Ze4]{[Ze4]} Zemel, S., \textsc{Integral Bases and Invariant Vectors for Weil Representations}, Res. Number Theory, vol 9, paper 5, 1--27 (2023).
\bibitem[Ze5]{[Ze5]} Zemel, S., \textsc{Seesaw Identities and Theta Contractions with Generalized Theta Functions, and Restrictions of Theta Lifts}, Ramanujan J., online first, https://link.springer.com/article/10.1007/s11139-023-00786-2 (2023).
\bibitem[Zh]{[Zh]} Zhang, W., \textsc{Modularity of Generating Functions of Special Cycles on Shimura Varieties}, Ph.D. thesis, Columbia University (2009).
\bibitem[Zi]{[Zi]} Ziegler, C., \textsc{Jacobi Forms of Higher Degree}, Abh. Math. Semin. Univ. Hambg., vol 59, 191--224 (1989).

\end{thebibliography}
\end{document}